\documentclass[11pt]{article} 

\usepackage[ruled,linesnumbered]{algorithm2e} 
\usepackage{amscd} 
\usepackage{amsfonts} 
\usepackage{amsmath} 
\usepackage{amssymb} 
\usepackage{amsthm} 
\usepackage{bm} 
\usepackage{booktabs} 
\usepackage{delarray} 
\usepackage{extarrows} 
\usepackage{geometry} 
\usepackage{graphicx} 
\usepackage{hyperref} 
\usepackage{longtable} 
\usepackage{tikz-cd} 
\usepackage[all]{xy} 
\usepackage[affil-it]{authblk} 

\usepackage{adjustbox}
\usepackage{booktabs}
\usepackage{caption}
\usepackage{color}
\usepackage{enumitem} 
\usepackage{float}
\usepackage{hyperref}
\usepackage{indentfirst} 
\usepackage{makecell} 
\usepackage{mathrsfs} 
\usepackage{mathtools} 
\usepackage{multirow} 
\usepackage{relsize} 
\usepackage{subfigure}

\setlist[enumerate]{itemsep=0pt,parsep=0pt}

\title{Lubin--Tate and multivariable $(\varphi,\OK\x)$-modules in dimension 2}
\author{Yitong Wang\thanks{E-mail address: \texttt{yitong.wang@universite-paris-saclay.fr}}}
\date{}

\allowdisplaybreaks[4] 

\def\AA{\mathbb{A}}

\def\FF{\mathbb{F}}
\def\GG{\mathbb{G}}

\def\RR{\mathbb{R}}

\def\ZZ{\mathbb{Z}}

\def\NNN{\mathbb{Z}_{\geq0}}

\def\bB{\mathbf{B}}

\def\cJ{\mathcal{J}}

\def\cO{\mathcal{O}}

\def\fm{\mathfrak{m}}

\def\alg{\operatorname{alg}}

\def\cont{\operatorname{cont}}

\def\cyc{\operatorname{cyc}}

\def\deg{\operatorname{deg}} 
 
\def\det{\operatorname{det}}

\def\dim{\operatorname{dim}}

\def\Ext{\operatorname{Ext}} 
\def\Fil{\operatorname{Fil}}

\def\Gal{\operatorname{Gal}}

\def\GL{\operatorname{GL}}

\def\Hom{\operatorname{Hom}}
\def\id{\operatorname{id}} 
\def\Im{\operatorname{Im}} 

\def\Ind{\operatorname{Ind}}

\def\Ker{\operatorname{Ker}}

\def\LT{\operatorname{LT}} 
\def\M{\operatorname{M}}
\def\Mat{\operatorname{Mat}}
\def\max{\operatorname{max}}

\def\mod{\operatorname{mod}}

\def\PGL{\operatorname{PGL}}

\def\soc{\operatorname{soc}}

\def\tr{\operatorname{tr}}

\def\unr{\operatorname{un}}

\def\cc{^{\circ}}
\def\ccc{^{\circ\circ}}

\def\eqdef{\overset{\mathrm{def}}{=}}
\def\Fp{\mathbb{F}_p}

\def\Fq{\mathbb{F}_q}

\def\into{\hookrightarrow}
\def\inv{^{-1}}
\def\ism{\stackrel{\sim}{\rightarrow}}

\def\loc{\textit{loc.cit.}}

\def\OK{\mathcal{O}_K}
\def\onto{\twoheadrightarrow}

\def\Qp{\mathbb{Q}_p}
\def\Qpbar{\overline{\mathbb{Q}}_p}
\def\rbar{\overline{r}}
\def\rhobar{\overline{\rho}}

\def\x{^{\times}}

\def\Zp{\mathbb{Z}_p}

\newcommand{\abs}[1]{|#1|}

\newcommand{\babs}[1]{\left|#1\right|}
\newcommand{\bang}[1]{\left\langle#1\right\rangle}
\newcommand{\bbbra}[1]{\left[#1\right]}
\newcommand{\bbra}[1]{\left(#1\right)}
\newcommand{\bigabs}[1]{\big|#1\big|}
\newcommand{\bigang}[1]{\big\langle#1\big\rangle}

\newcommand{\bigbra}[1]{\big(#1\big)}

\newcommand{\bigset}[1]{\big\{#1\big\}}
\newcommand{\bra}[1]{(#1)}
\newcommand{\dbra}[1]{(\!(#1)\!)}
\newcommand{\ddbra}[1]{[\![#1]\!]}
\newcommand{\norm}[1]{\|#1\|}

\newcommand{\ovl}[1]{\overline{#1}}
\newcommand{\pmat}[1]{\begin{pmatrix}#1\end{pmatrix}}
\newcommand{\set}[1]{\{ #1 \}}
\newcommand{\smat}[1]{\left(\begin{smallmatrix}#1\end{smallmatrix}\right)}
\newcommand{\sset}[1]{\left\{ #1 \right\}}
\newcommand{\un}[1]{\underline{#1}}
\newcommand{\wh}[1]{\widehat{#1}}

\begin{document}

\newtheorem{definition}{Definition}[section] 
\newtheorem{remark}[definition]{Remark}
\newtheorem{example}[definition]{Example}
\newtheorem{proposition}[definition]{Proposition}
\newtheorem{lemma}[definition]{Lemma}
\newtheorem{corollary}[definition]{Corollary}
\newtheorem{theorem}[definition]{Theorem}
\newtheorem{conjecture}[definition]{Conjecture}

\maketitle

\begin{abstract}
    Let $p$ be a prime number, $K$ a finite unramified extension of $\Qp$ and $\FF$ a finite extension of $\Fp$. For $\rhobar$ any reducible two-dimensional representation of $\Gal(\ovl{K}/K)$ over $\FF$, we compute explicitly the associated \'etale $(\varphi,\OK\x)$-module $D_A^{\otimes}(\rhobar)$ defined in \cite{BHHMS3}. Then we let $\pi$ be an admissible smooth representation of $\GL_2(K)$ over $\FF$ occurring in some Hecke eigenspaces of the mod $p$ cohomology and $\rhobar$ be its underlying two-dimensional representation of $\Gal(\ovl{K}/K)$ over $\FF$. Assuming that $\rhobar$ is maximally non-split, we prove under some genericity assumption that the associated \'etale $(\varphi,\OK\x)$-module $D_A(\pi)$ defined in \cite{BHHMS2} is isomorphic to $D_A^{\otimes}(\rhobar)$. This extends the results of \cite{BHHMS3}, where $\rhobar$ was assumed to be semisimple.
\end{abstract}

\tableofcontents

\section{Introduction}\label{dim2 Sec intro}

Let $p$ be a prime number. The mod $p$ Langlands correspondence for $\GL_2(\Qp)$ is completely known by the work of Breuil, Colmez, Emerton, etc. In particular, Colmez (\cite{Col10}) constructed a functor from the category of admissible finite length mod $p$ representations of $\GL_2(\Qp)$ to the category of finite-dimensional continuous mod $p$ representations of $\Gal(\Qpbar/\Qp)$, using Fontaine's category of $(\varphi,\Gamma)$-modules (\cite{Fon90}) as an intermediate step. This gives a functorial way to realize the mod $p$ Langlands correspondence for $\GL_2(\Qp)$.

However, the situation becomes much more complicated when we consider $\GL_2(K)$ for $K$ a nontrivial finite extension of $\Qp$. For example, there are many more supersingular representations of $\GL_2(K)$ (\cite{BP12}) and we don't have a classification of these representations. Motivated by the local-global compatibility result of Emerton (\cite{Eme11}) for $\GL_2(\Qp)$, we are particularly interested in the mod $p$ representations $\pi$ of $\GL_2(K)$ coming from the cohomology of towers of Shimura curves. 

\hspace{\fill}

We introduce the global setup following \cite{BHHMS3}. Let $F$ be a totally real number field that is unramified at places above $p$. Let $D$ be a quaternion algebra with center $F$ which is split at places above $p$ and at exactly one infinite place. For each compact open subgroup $U\subseteq(D\otimes_F\AA_F^{\infty})$ where $\AA_F^{\infty}$ is the set of finite ad\`eles of $F$, we denote by $X_U$ the associated smooth projective algebraic Shimura curve over $F$.

Let $\FF$ be a sufficiently large finite extension of $\Fp$. We fix an absolutely irreducible continuous representation $\rbar:\Gal(\ovl{F}/F)\to\GL_2(\FF)$. For $w$ a finite place of $F$, we write $\rbar_w\eqdef\rbar|_{\Gal(\ovl{F}_w/F_w)}$. We let $S_D$ be the set of finite places where $D$ ramifies, $S_{\rbar}$ be the set of finite places where $\rbar$ ramifies, and $S_p$ the set of places above $p$. We fix a place $v\in S_p$ and write $K\eqdef F_v$. We assume that
\begin{enumerate}
    \item 
    $p\geq5$, $\rbar|_{\Gal(\ovl{F}/F(\sqrt[p]{1}))}$ is absolutely irreducible and the image of $\rbar\bigbra{\!\Gal(\ovl{F}/F(\sqrt[5]{1}))}$ in $\PGL_2(\FF)$ is not isomorphic to $A_5$;
    \item 
    $\rbar_w$ is generic in the sense of \cite[Def.~11.7]{BP12} for $w\in S_p$;
    \item 
    $\rbar_w$ is non-scalar for $w\in S_D$.
\end{enumerate}
Then there is a so-called ``local factor'' defined in \cite[\S3.3]{BD14} and \cite[\S6.5]{EGS15} as follows:
\begin{equation}\label{dim2 Eq local factor}
    \pi\eqdef\Hom_{U^v}\bigg(\ovl{M}^v,\Hom_{\Gal(\ovl{F}/F)}\Big(\rbar,\varinjlim\limits_V H^1_{\text{\'et}}(X_V\times_F\ovl{F},\FF)\Big)\bigg)[\fm'],
\end{equation}
where the inductive limit runs over the compact open subgroups $V\subseteq(D\otimes_F\AA_F^{\infty})\x$, and we refer to \cite[\S3.3]{BD14} and \cite[\S6.5]{EGS15} for the definitions of the compact open subgroup $U^v\subseteq(D\otimes_F\AA_F^{\infty,v})\x$, the (finite-dimensional) irreducible smooth representation $\ovl{M}^v$ of $U^v$ over $\FF$, and the maximal ideal $\fm'$ in a certain Hecke algebra.

\hspace{\fill}

In \cite{BHHMS2}, Breuil-Herzig-Hu-Morra-Schraen attached to $\pi$ an \'etale $(\varphi,\OK\x)$-module $D_A(\pi)$ over $A$, which we briefly recall as follows. We write $f\eqdef[F_v:\Qp]$. We let $\Fq$ be the residue field of $F_v$ (hence $q=p^f$) and fix an embedding $\sigma_0:\Fq\into\FF$. Then we have $\FF\ddbra{\OK}=\FF\ddbra{Y_0,\ldots,Y_{f-1}}$ with $Y_j\eqdef\sum\nolimits_{a\in\Fq\x}\sigma_0(a)^{-p^j}\delta_{[a]}\in\FF\ddbra{\OK}$ for $0\leq j\leq f-1$, where $[a]\in\OK\x$ is the Techm\"uller lift of $a\in\Fq\x$ and $\delta_{[a]}$ is the corresponding element in $\FF\ddbra{\OK}$. We let $A$ be the completion of $\FF\ddbra{\OK}[1/(Y_0\cdots Y_{f-1})]$ with respect to the $(Y_0,\ldots,Y_{f-1})$-adic topology. There is an $\FF$-linear action of $\OK\x$ on $\FF\ddbra{\OK}$ given by multiplication on $\OK$, and an $\FF$-linear Frobenius $\varphi$ on $\FF\ddbra{\OK}$ given by multiplication by $p$ on $\OK$. They extend canonically by continuity to commuting continuous $\FF$-linear actions of $\varphi$ and $\OK\x$ on $A$. Then an \'etale $(\varphi,\OK\x)$-module over $A$ is by definition a finite free $A$-module endowed with a semi-linear Frobenius $\varphi$ and a commuting continuous semi-linear action of $\OK\x$ such that the image of $\varphi$ generates everything.

For $\pi$ as in (\ref{dim2 Eq local factor}), we let $\pi^{\vee}$ be its $\FF$-linear dual, which is a finitely generated $\FF\ddbra{I_1}$-module and is endowed with the $\fm_{I_1}$-adic topology, where $I_1\eqdef\smat{1+p\OK&\OK\\p\OK&1+p\OK}\subseteq\GL_2(\OK)$ and $\fm_{I_1}$ is the maximal ideal of $\FF\ddbra{I_1}$. We define $D_A(\pi)$ to be the completion of $\FF\ddbra{\OK}[1/(Y_0\cdots Y_{f-1})]\otimes_{\FF\ddbra{\OK}}\pi^{\vee}$ with respect to the tensor product topology, where we view $\pi^{\vee}$ as an $\FF\ddbra{\OK}$-module via $\FF\ddbra{\OK}\cong\FF\ddbra{\smat{1&\OK\\0&1}}\subseteq\FF\ddbra{I_1}$. The $\OK\x$-action on $\pi^{\vee}$ given by $f\mapsto f\circ\smat{a&0\\0&1}$ (for $a\in\OK\x$) extends by continuity to $D_A(\pi)$, and the $\psi$-action on $\pi^{\vee}$ given by $f\mapsto f\circ\smat{p&0\\0&1}$ induces a continuous $A$-linear isomorphism $\beta:D_A(\pi)\ism A\otimes_{\varphi,A}D_A(\pi)$ (\cite[Thm.~1.1]{Wang2}). In particular, the inverse $\beta\inv=\id\otimes\varphi$ makes $D_A(\pi)$ an \'etale $(\varphi,\OK\x)$-module (\cite[Cor.~3.1.2.9]{BHHMS2} and \cite[Remark.~2.6.2]{BHHMS3}).

In \cite{BHHMS3}, Breuil-Herzig-Hu-Morra-Schraen also gave a conjectural description of $D_A(\pi)$ in terms of $\rbar_v$. They constructed a functor $D_A^{\otimes}$ from the category of finite-dimensional continuous representations of $\Gal(\ovl{F}_v/F_v)$ over $\FF$ to the category of \'etale $(\varphi,\OK\x)$-modules over $A$, using the category of Lubin--Tate $(\varphi,\OK\x)$-modules as an intermediate step. We refer to \S\ref{dim2 Sec DA0} for the precise definition. Then they conjectured that $D_A(\pi)$ is isomorphic to $D_A^{\otimes}(\rbar_v(1))$ as \'etale $(\varphi,\OK\x)$-modules over $A$, where $\rbar_v(1)$ is the Tate twist of $\rbar_v$. We compute explicitly the structure of the \'etale $(\varphi,\OK\x)$-module $D_A^{\otimes}(\rbar_v(1))$ in Theorem \ref{dim2 Thm main}, extending the results of \cite{BHHMS3} where $\rbar_v$ was assumed to be semisimple.

We assume moreover that
\begin{enumerate}
    \item[(iv)]
    the framed deformation ring $R_{\rbar_w}$ of $\rbar_w$ over the Witt vectors $W(\FF)$ is formally smooth for $w\in(S_D\cup S_{\rbar})\setminus S_p$;
    \item[(v)]
    $\rbar_v$ is of the following form up to twist:
    \begin{equation*}
        \rbar_{v}|_{I_{F_v}}\cong\pmat{\omega_f^{\sum\nolimits_{j=0}^{f-1}(r_j+1)p^j}&*\\0&1}~\text{with}~\max\set{12,2f+1}\leq r_j\leq p-\max\set{15,2f+3}~\forall\,j,
    \end{equation*}
    where $I_{F_v}\subseteq\Gal(\ovl{F}_v/F_v)$ is the decomposition group. 
\end{enumerate}
Our main result is the following:

\begin{theorem}[\S\ref{dim2 Sec proof}]\label{dim2 Thm main intro}
    Let $\pi$ be as in \eqref{dim2 Eq local factor} and keep all the assumptions on $\rbar$. Assume moreover that $|W(\rbar_v)|=1$, where $W(\rbar_v)$ is the set of Serre weights of $\rbar_v$ defined in \cite[\S3]{BDJ10}. Then we have an isomorphism of \'etale $(\varphi,\OK\x)$-modules
    \begin{equation*}
        D_A(\pi)\cong D_A^{\otimes}(\rbar_v(1)).
    \end{equation*}
\end{theorem}

Theorem \ref{dim2 Thm main intro} is proved by \cite[Thm.~3.1.3]{BHHMS3} when $\rbar_v$ is semisimple. Using the explicit computation of $D_A^{\otimes}(\rbar_v(1))$ in Theorem \ref{dim2 Thm main} and the results of \cite{Wang2} on $D_A(\pi)$, we are reduced to the computation of some constants coming from the diagram $(\pi^{I_1}\into\pi^{K_1})$ in the sense of \cite{DL21}. When $|W(\rbar_v)|=1$ (i.e.\,$\rbar_v$ is maximally non-split), these constants are computed by \cite{BD14} in terms of the Fontaine–Laffaille module associated to $\rbar_v$ (\cite{FL82}). We remark that our method should apply to arbitrary $W(\rbar_v)$ once we compute the corresponding constants coming from the diagram $(\pi^{I_1}\into\pi^{K_1})$ in general.

The proof of Theorem \ref{dim2 Thm main intro} is very computational. There may exist a more conceptual proof one day, which will hopefully avoid the genericity assumptions on $\rbar_v$ and the technical computations, but such proof is not known so far.

\subsection*{Organization of the article}

In \S\ref{dim2 Sec LT}, we give an explicit parametrization of the Lubin--Tate $(\varphi,\OK\x)$-modules associated to reducible two-dimensional representations of $\Gal(\ovl{K}/K)$ over $\FF$ when $p\geq3$. In \S\ref{dim2 Sec A}, we construct explicitly some \'etale $(\varphi,\OK\x)$-modules over $A$ that will be needed in \S\ref{dim2 Sec DA0}, where we compute explicitly the associated \'etale $(\varphi,\OK\x)$-module $D_A^{\otimes}(\rhobar)$ for $\rhobar$ an arbitrary reducible two-dimensional representation of $\Gal(\ovl{K}/K)$ over $\FF$ in Theorem \ref{dim2 Thm main} when $p\geq5$. Finally, in \S\ref{dim2 Sec proof}, we combine all the previous results and the results of \cite{Wang2} and \cite{BD14} to finish the proof of Theorem \ref{dim2 Thm main intro}.

\subsection*{Acknowledgements}

We thank Christophe Breuil for suggesting this problem, and thank Christophe Breuil and Ariane M\'ezard for helpful discussions and a careful reading of the earlier drafts of this paper. 

This work was supported by the Ecole Doctorale de Math\'ematiques Hadamard (EDMH).

\subsection*{Notation}\label{dim2 Sec notation}

Let $p$ be an odd prime. We fix an algebraic closure $\Qpbar$ of $\Qp$. Let $K\subseteq\Qpbar$ be the unramified extension of $\Qp$ of degree $f\geq1$ with ring of integers $\OK$ and residue field $\Fq$ (hence $q=p^f$). We denote by $G_K\eqdef\Gal(\Qpbar/K)$ the absolute Galois group of $K$ and $I_K\subseteq G_K$ the inertia subgroup. Let $\FF$ be a large enough finite extension of $\FF_p$. Fix an embedding $\sigma_0:\Fq\into\FF$ and let $\sigma_j\eqdef\sigma_0\circ\varphi^j$ for $j\in\ZZ$, where $\varphi:x\mapsto x^p$ is the arithmetic Frobenius on $\Fq$. We identify $\cJ\eqdef\Hom(\Fq,\FF)$ with $\set{0,1,\ldots,f-1}$, which is also identified with $\ZZ/f\ZZ$ so that the addition and subtraction in $\cJ$ are modulo $f$. For $a\in\OK$, we denote by $\ovl{a}\in\Fq$ its reduction modulo $p$. For $a\in\Fq$, we also view it as an element of $\FF$ via $\sigma_0$.

For $F$ a perfect ring of characteristic $p$, we denote by $W(F)$ the ring of Witt vectors of $F$. For $x\in F$, we denote by $[x]\in W(F)$ its Techm\"uller lift.

Let $I\eqdef\smat{\OK\x&\OK\\p\OK&\OK\x}\subseteq\GL_2(\OK)$ be the Iwahori subgroup, $I_1\eqdef\smat{1+p\OK&\OK\\p\OK&1+p\OK}\subseteq\GL_2(\OK)$ be the pro-$p$ Iwahori subgroup, $K_1\eqdef1+p\M_2(\OK)\subseteq\GL_2(\OK)$ be the first congruence subgroup, $N_0\eqdef\smat{1&\OK\\0&1}$ and $H\eqdef\smat{[\Fq\x]&0\\0&[\Fq\x]}$.

For $P$ a statement, we let $\delta_P\eqdef1$ if $P$ is true and $\delta_P\eqdef0$ otherwise.

\hspace{\fill}

Throughout this article, we let $\rhobar:G_K\to\GL_2(\FF)$ be of the following form:
\begin{equation}\label{dim2 Eq rhobar}
    \rhobar\cong\pmat{\omega_f^h\unr(\lambda_0)&*\\0&\unr(\lambda_1)},
\end{equation}
where $0\leq h\leq q-2$, $\lambda_0,\lambda_1\in\FF\x$, for $\xi\in\FF\x$ we denote by $\unr(\xi):G_K\to\FF\x$ the unramified character sending geometric Frobenius elements to $\xi$, and $\omega_f:G_K\to\FF$ is the extension to $G_K$ of the fundamental character of level $f$ (associate to $\sigma_0$) such that $\omega_f(g)$ is the reduction modulo $p$ of $g(p_f)/p_f\in\mu_{q-1}(\ovl{K}\x)$ for all $g\in G_K$ and for any choice of a $(q-1)$-th root $p_f$ of $-p$. 

Then we can write $h=\sum\nolimits_{i=0}^{f-1}p^jh_j$ with $0\leq h_j\leq p-1$ for $0\leq j\leq f-1$ in a unique way. We extend the definition of $h_j$ to all $j\in\ZZ$ by the relation $h_{j+f}=h_j$ for all $j\in\ZZ$.
For $j\geq0$, we set
\begin{equation*}
    [h]_{j}\eqdef h_0+ph_1+\cdots+p^jh_j. 
\end{equation*}
In particular, we have $[h]_{f-1}=h$. We also define $[h]_{-1}\eqdef0$ and $[h]_{-2}\eqdef-h_{f-1}/p$, hence $[h]_{j+f}=h+q[h]_j$ for all $j\geq-2$.

\section{Lubin--Tate \texorpdfstring{$(\varphi,\OK\x)$}.-modules}\label{dim2 Sec LT}
 
In this section, we give an explicit parametrization of the Lubin--Tate $(\varphi,\OK\x)$-modules corresponding to $\rhobar$ as in (\ref{dim2 Eq rhobar}). The main result is Theorem \ref{dim2 Thm basis LT}.

Let $G_{\LT}$ be the unique (up to isomorphism) Lubin--Tate formal $\OK$-module over $\OK$ associated to the uniformizer $p$. We choose the formal variable $T_K$ of $G_{\LT}$ so that the logarithm (\cite[\S8.6]{Lan90}) is given by the power series $\sum\nolimits_{n=0}^{\infty}p^{-n}T_K^{q^n}$. For $a\in\OK$ we have power series $a_{\LT}(T_K)\in aT_K+T_K^2\OK\ddbra{T_{K}}$.

As in \cite[\S2.1]{BHHMS3}, there is a continuous $\FF$-linear endomorphism $\varphi$ of $\FF\otimes_{\Fp}\Fq\dbra{T_K}$ which is the $p$-th power map on $\Fq$ and satisfies $\varphi(T_K)=T_K^p$, and a continuous $\FF\otimes_{\Fp}\Fq$-linear action (commuting with $\varphi$) of $\OK\x$ on $\FF\otimes_{\Fp}\Fq\dbra{T_K}$ satisfying $a(T_K)=a_{\LT}(T_K)$ for $a\in\OK\x$, where we still denote by $a_{\LT}(T_K)\in\Fq\ddbra{T_K}$ the reduction modulo $p$ of $a_{\LT}(T_K)\in\OK\ddbra{T_K}$.
Then there is a covariant exact equivalence of categories compatible with tensor products between the category of finite-dimensional continuous representations of $\Gal(\ovl{K}/K)$ over $\FF$ and the category of \'etale $(\varphi,\OK\x)$-modules over $\FF\otimes_{\Fp}\Fq\dbra{T_K}$.

For $D_K$ an \'etale $\varphi$-module over $\FF\otimes_{\Fp}\Fq\dbra{T_K}$, the isomorphism 
\begin{equation}\label{dim2 Eq decomposition LT}
\begin{aligned}
    \FF\otimes_{\Fp}\Fq\dbra{T_K}&\ism\FF\dbra{T_{K,\sigma_0}}\times\FF\dbra{T_{K,\sigma_1}}\times\cdots\times\FF\dbra{T_{K,\sigma_{f-1}}}\\
    \lambda\otimes\bra{\scalebox{0.9}{$\sum$}_{n\gg-\infty}c_nT_K^n}&\mapsto\bigbra{\bra{\scalebox{0.9}{$\sum$}_{n\gg-\infty}\lambda\sigma_0(c_n)T_{K,\sigma_0}^n},\ldots,\bra{\scalebox{0.9}{$\sum$}_{n\gg-\infty}\lambda\sigma_{f-1}(c_n)T_{K,\sigma_{f-1}}^n}}
\end{aligned}
\end{equation}
induces a decomposition 
\begin{equation*}
    D_K\ism D_{K,\sigma_0}\times\cdots\times D_{K,\sigma_{f-1}}.
\end{equation*}
For each $0\leq i\leq f-1$, the functor $D_{K}\mapsto D_{K,\sigma_i}$ induces an equivalence of categories between the category of \'etale $(\varphi,\OK\x)$-modules over $\FF\otimes_{\Fp}\Fq\dbra{T_K}$ and the category of \'etale $(\varphi_q,\OK\x)$-modules over $\FF\dbra{T_{K,\sigma_i}}$. Here $\varphi_q\eqdef\varphi^f$, and $\FF\dbra{T_{K,\sigma_i}}$ is endowed with an $\FF$-linear endomorphism $\varphi_q$ such that $\varphi_q(T_{K,\sigma_i})=T_{K,\sigma_i}^q$, and a continuous $\FF$-linear action (commuting with $\varphi_q$) of $\OK\x$ such that $a(T_{K,\sigma_i})=a_{\LT}(T_{K,\sigma_i})$ for $a\in\OK\x$, where $a_{\LT}(T_{K,\sigma_i})\in\FF\ddbra{T_{K,\sigma_i}}$ is the image of $a_{\LT}(T_K)\in\Fq\ddbra{T_{K}}$ in $\FF\ddbra{T_{K,\sigma_i}}$ via the embedding $\sigma_i:\Fq\into \FF$.

For $\rhobar$ a finite-dimensional continuous representation of $\Gal(\ovl{K}/K)$ over $\FF$, we denote by $D_K(\rhobar)$ the associated \'etale $(\varphi,\OK\x)$-module over $\FF\otimes_{\Fp}\Fq\dbra{T_K}$, and for each $0\leq i\leq f-1$ we denote by $D_{K,\sigma_i}(\rhobar)$ the associated \'etale $(\varphi_q,\OK\x)$-module over $\FF\dbra{T_{K,\sigma_i}}$. 

For $a\in\OK\x$, we set
\begin{equation*}
    f_a^{\LT}\eqdef\ovl{a}T_{K}/a(T_{K})\in1+T_{K}\FF\ddbra{T_{K}}.
\end{equation*}
We still denote by $f_a^{\LT}$ its image in $\FF\dbra{T_{K,\sigma_0}}$ via $\sigma_0$ when there is no possible confusion. 

Any (continuous) character of $G_K$ over $\FF$ is of the form $\omega_f^h\unr(\lambda)$ for $0\leq h\leq q-2$ and $\lambda\in\FF\x$. By \cite[Lemma~2.1.8]{BHHMS3}, the corresponding \'etale $(\varphi_q,\OK\x)$-module $D_{K,\sigma_0}\bigbra{\omega_f^h\unr(\lambda)}$ can be described as follows ($a\in\OK\x$):
\begin{equation}\label{dim2 Eq character LT}
\left\{\begin{array}{cll}
    D_{K,\sigma_0}\bigbra{\omega_f^h\unr(\lambda)}&=&\FF\dbra{T_{K,\sigma_0}}e\\
    \varphi_q(e)&=&\lambda T_{K,\sigma_0}^{-(q-1)h}e\\
    a(e)&=&\bbra{f_a^{\LT}}^he.
\end{array}\right.
\end{equation}

\begin{lemma}\label{dim2 Lem fa LT}
    We have $f_a^{\LT}=1$ for $a\in[\Fq\x]$. More generally, we have for $a\in\OK\x$
    \begin{equation*}
        \bbra{f_a^{\LT}}\inv\in1+c_aT_{K}^{q-1}-c_a^{p^{f-1}}T_{K}^{(q-1)(p^{f-1}+1)}+T_{K}^{(q-1)(2p^{f-1}+1)}\Fq\ddbra{T_{K}^{q-1}},
    \end{equation*}
    where $c_a\in\Fq$ is the reduction modulo $p$ of $(1-a^{q-1})/p\in\OK$.
\end{lemma}

\begin{proof}
    By \cite[Lemma 8.6.2]{Lan90} we have equality in $\OK\ddbra{T_K}$
    \begin{equation}\label{dim2 Eq fa LT log}
        \sum\limits_{n=0}^{\infty}\frac{a_{\LT}(T_K)^{q^n}}{p^n}=a\sum\limits_{n=0}^{\infty}\frac{T_K^{q^n}}{p^n}.
    \end{equation}
    In particular, for $a\in[\Fq\x]$ we have $a_{\LT}(T_K)=aT_K$, which implies $f_a^{\LT}=1$. Then the commutativity of the actions of $\OK\x$ and $[\Fq\x]$ implies that $a_{\LT}(T_K)\in aT_K\bigbra{1+T_K^{q-1}\OK\ddbra{T_{K}^{q-1}}}$ for $a\in\OK\x$, and we write in $\OK\ddbra{T_{K}^{q-1}}$
    \begin{equation}\label{dim2 Eq fa LT a-action}
        a_{\LT}(T_K)=aT_K\bbra{1+\sum\limits_{i=1}^{\infty}x_a(i)T_K^{(q-1)i}}
    \end{equation}
    for $x_a(i)\in\OK$. Then by \eqref{dim2 Eq fa LT log} we have
    \begin{equation}\label{dim2 Eq fa LT congruence}
        1+\sum\limits_{i=1}^{2p^{f-1}}x_a(i)T_{K}^{(q-1)i}+\frac{a^{q-1}T_K^{q-1}}{p}\bbra{1+\sum\limits_{i=1}^{2p^{f-1}}x_a(i)T_{K}^{(q-1)i}}^q\equiv1+\frac{T_K^{q-1}}{p}\quad\mod T_K^{(q-1)(2p^{f-1}+1)}.
    \end{equation}
    Comparing the coefficients of $T_K^{q-1}$, we get $x_a(1)=(1-a^{q-1})/p$. Also, each term of the expansion $\bigbra{1+\sum\nolimits_{i=1}^{2p^{f-1}}\!\!x_a(i)T_{K}^{(q-1)i}}^q$ has the form 
    \begin{equation}\label{dim2 Eq fa LT expansion}
        \frac{q!}{n_0!\cdots n_{2p^{f-1}}!}\prod\limits_{i=1}^{2p^{f-1}}x_a(i)^{n_i}T_K^{(q-1)\sum\nolimits_{i=1}^{2p^{f-1}}\!\!in_i}
    \end{equation}
    with $0\leq n_i\leq q$ and $\sum\nolimits_{i=0}^{2p^{f-1}}n_i=q$. 
    
    \hspace{\fill}
    
    \noindent\textbf{Claim.} For the terms in \eqref{dim2 Eq fa LT expansion} such that $\sum\nolimits_{i=1}^{2p^{f-1}}\!\!in_i\leq2p^{f-1}-1$, we have $v_p\bigbra{q!/\bra{n_0!\cdots n_{2p^{f-1}}!}}\geq2$ except in the following two cases: 
    \begin{enumerate}
        \item[(a)] 
        $n_0=q$ and $n_i=0$ for $i\neq0$, in which case the term in (\ref{dim2 Eq fa LT expansion}) is $1$;
        \item[(b)]
        $n_0=(p-1)p^{f-1}$, $n_1=p^{f-1}$ and $n_i=0$ for $i>1$, in which case the term in (\ref{dim2 Eq fa LT expansion}) is congruent to $px_a(1)^{p^{f-1}}T_K^{(q-1)p^{f-1}}$ modulo $p^2$. 
    \end{enumerate}
    
    \proof Recall that $v_p(n!)=(n-S_p(n))/(p-1)$, where $S_p(n)$ is the sum of the digits in the $p$-adic expansion of $n$. Hence we have
    \begin{equation*}
        v_p\bbra{\frac{q!}{n_0!\cdots n_{2p^{f-1}}!}}=\frac{1}{p-1}\bbbra{\bbra{\sum\limits_{i=0}^{2p^{f-1}}S_p(n_i)}-1}.
    \end{equation*}
    If $v_p\bigbra{q!/\bra{n_0!\cdots n_{2p^{f-1}}!}}\leq1$, then we have $\sum\nolimits_{i=0}^{2p^{f-1}}\!\!S_p(n_i)\leq p$, which implies that each $n_i$ must be a multiple of $p^{f-1}$, hence (a) and (b) are the only possibilities since $\sum\nolimits_{i=1}^{2p^{f-1}}\!\!in_i\leq2p^{f-1}-1$. Moreover, we have by Lucas theorem
    \begin{equation*}
        \frac{1}{p}\cdot\frac{q!}{\bbra{(p-1)p^{f-1}}!\bbra{p^{f-1}}!}=\binom{p^f-1}{p^{f-1}-1}\equiv1\quad\mod p,
    \end{equation*}
    hence the term in (\ref{dim2 Eq fa LT expansion}) in case (b) is congruent to $px_a(1)^{p^{f-1}}T_K^{(q-1)p^{f-1}}$ modulo $p^2$.\qed
    
    \hspace{\fill}

    By the claim above and (\ref{dim2 Eq fa LT congruence}), for $1\leq i\leq 2p^{f-1}$ we have $x_a(i)\in p\OK$ except possibly in the following two cases:
    \begin{enumerate}
        \item 
        $x_a(1)=(1-a^{q-1})/p$;
        \item
        $x_a(p^{f-1}+1)\equiv -a^{q-1}x_a(1)^{p^{f-1}}\equiv-x_a(1)^{p^{f-1}}\ \mod p$.
    \end{enumerate}
    Then by reducing \eqref{dim2 Eq fa LT a-action} modulo $p$ we have
    \begin{equation*}
        \bbra{f_a^{\LT}}\inv=a_{\LT}(T_{K})/(\ovl{a}T_{K})\in1+c_aT_{K}^{q-1}-c_a^{p^{f-1}}T_{K}^{(q-1)(p^{f-1}+1)}+T_{K}^{(q-1)(2p^{f-1}+1)}\Fq\ddbra{T_{K}^{q-1}},
    \end{equation*}
    which completes the proof.
\end{proof}

\begin{remark}\label{dim2 Rk ca LT}
    The map $\OK\x\to\Fq$, $a\mapsto c_a$ is a group homomorphism and satisfies:
    \begin{enumerate}
        \item 
        If $a\in[\Fq\x]$, then $c_a=0$.
        \item
        If $a=1+pb$ for some $b\in\OK$, then $c_a=\ovl{b}$.
    \end{enumerate}
\end{remark}

Since $a(T_{K,\sigma_0})=\ovl{a}T_{K,\sigma_0}$ for $a\in[\Fq\x]$ by Lemma \ref{dim2 Lem fa LT}, we have $\FF\dbra{T_{K,\sigma_0}}^{[\Fq\x]}=\FF\dbra{T_{K,\sigma_0}^{q-1}}$. Then for $\rhobar$ as in (\ref{dim2 Eq rhobar}), we have
$D_{K,\sigma_0}(\rhobar)\cong\FF\dbra{T_{K,\sigma_0}}\otimes_{\FF\dbra{T_{K,\sigma_0}^{q-1}}}D_{K,\sigma_0}(\rhobar)^{[\Fq\x]}$, where $D_{K,\sigma_0}(\rhobar)^{[\Fq\x]}$ has the following form (using \eqref{dim2 Eq character LT}, and $a\in\OK\x$):
\begin{equation*}
\left\{\begin{array}{cll}
    D_{K,\sigma_0}(\rhobar)^{[\Fq\x]}&=&\FF\dbra{T_{K,\sigma_0}^{q-1}}e_0\oplus\FF\dbra{T_{K,\sigma_0}^{q-1}}e_1\\
    \varphi_q(e_0\ e_1)&=&(e_0\ e_1)\Mat(\varphi_q)\\
    a(e_0\ e_1)&=&(e_0\ e_1)\Mat(a)
\end{array}\right.
\end{equation*}
with
\begin{equation*}
\left\{\begin{array}{cll}
    \Mat(\varphi_q)&=&\pmat{\lambda_0T_{K,\sigma_0}^{-(q-1)h}&\lambda_1D\\0&\lambda_1}\\
    \Mat(a)&=&\pmat{\bbra{f_a^{\LT}}^h&E_a\\0&1}
\end{array}\right.
\end{equation*}
for some $D\in\FF\dbra{T_{K,\sigma_0}^{q-1}}$ and $E_a\in\FF\dbra{T_{K,\sigma_0}^{q-1}}$.

\begin{definition}\label{dim2 Def WLT}
    Let $0\leq h\leq q-2$ and $\lambda_0,\lambda_1\in\FF\x$. We define $W^{\LT}$ to be the set of equivalence classes of tuples $[B]=\bigbra{D,(E_a)_{a\in\OK\x}}$ such that 
    \begin{enumerate}
        \item 
        $D\in\FF\dbra{T_{K,\sigma_0}^{q-1}}$, $E_a\in\FF\dbra{T_{K,\sigma_0}^{q-1}}$ for all $a\in\OK\x$, and the map $\OK\x\to\FF\dbra{T_{K,\sigma_0}^{q-1}}$, $a\mapsto E_a$ is continuous;
        \item
        $E_{ab}=E_a+\bbra{f_a^{\LT}}^ha(E_{b})$ for all $a,b\in\OK\x$;
        \item
        $\bigbra{\id-\lambda_0\lambda_1\inv T_{K,\sigma_0}^{-(q-1)h}\varphi_q}(E_a)=\bigbra{\id-\bbra{f_a^{\LT}}^ha}(D)$ for all $a\in\OK\x$;
        \item
        two tuples $\bigbra{D,(E_a)_{a\in\OK\x}}$ and $\bigbra{D',(E_a')_{a\in\OK\x}}$ are equivalent if and only if there exists $b\in\FF\dbra{T_{K,\sigma_0}^{q-1}}$ such that
        \begin{equation*}
        \left\{\begin{aligned}
            D'&=D+\bbra{\id-\lambda_0\lambda_1\inv T_{K,\sigma_0}^{-(q-1)h}\varphi_q}(b)\\
            E_a'&=E_a+\bbra{\id-\bbra{f_a^{\LT}}^ha}(b)\quad\forall\,a\in\OK\x.
        \end{aligned}\right.
        \end{equation*}
    \end{enumerate}
    It has a natural structure of an $\FF$-vector space.
\end{definition}

By the definition of $W^{\LT}$ and the equivalence of categories $\rhobar\mapsto D_{K,\sigma_0}(\rhobar)$, there is an isomorphism of $\FF$-vector spaces 
\begin{equation}\label{dim2 Eq Fon Ext LT}
    W^{\LT}\cong\Ext^1\bbra{D_{K,\sigma_0}\bigbra{\!\unr(\lambda_1)},D_{K,\sigma_0}\bigbra{\omega_f^h\unr(\lambda_0)}}\cong H^1\bbra{G_K,\FF\bigbra{\omega_f^h\unr(\lambda_0\lambda_1\inv)}},
\end{equation}
where $\Ext^1$ is defined in the category of \'etale $(\varphi_q,\OK\x)$-modules over $\FF\dbra{T_{K,\sigma_0}}$. For each $[B]\in W^{\LT}$, we denote by $D([B])$ the corresponding \'etale $(\varphi_q,\OK\x)$-module over $\FF\dbra{T_{K,\sigma_0}}$. Note that $D([B])\cong D(\lambda[B])$ as \'etale $(\varphi_q,\OK\x)$-modules over $\FF\dbra{T_{K,\sigma_0}}$ for $\lambda\in\FF\x$.

\begin{lemma}\label{dim2 Lem congruence LT}
    Let $0\leq h\leq q-2$.
    \begin{enumerate}
        \item 
        For $i\geq-1$ and $a\in\OK\x$, we have 
        \begin{equation*}
            \bbra{\id-\bbra{f_a^{\LT}}^ha}\!\bbra{T_{K,\sigma_0}^{-(q-1)[h]_i}}\in T_{K,\sigma_0}^{q-1}\FF\ddbra{T_{K,\sigma_0}^{q-1}}.
        \end{equation*}
        \item
        For $i\geq-1$ and $a\in\OK\x$, we have
        \begin{equation*}
            \bbra{\id-\bbra{f_a^{\LT}}^ha}\!\bbra{T_{K,\sigma_0}^{-(q-1)([h]_i+p^{i+1})}}\in(h_{i+1}-1)c_a^{p^{i+1}}T_{K,\sigma_0}^{-(q-1)[h]_i}+T_{K,\sigma_0}^{q-1}\FF\ddbra{T_{K,\sigma_0}^{q-1}}.
        \end{equation*}
        \item
        For $i\geq f-1$ such that $h_i=1$ and $a\in\OK\x$, we have
        \begin{equation*}
            \bbra{\id-\bbra{f_a^{\LT}}^ha}\!\bbra{T_{K,\sigma_0}^{-(q-1)([h]_i+p^{i+1-f})}}\in-c_a^{p^{i+1}}T_{K,\sigma_0}^{-(q-1)[h]_i}\!+\!c_a^{p^i}T_{K,\sigma_0}^{-(q-1)[h]_{i-1}}\!+\!T_{K,\sigma_0}^{q-1}\FF\ddbra{T_{K,\sigma_0}^{q-1}}.
        \end{equation*}
    \end{enumerate}    
\end{lemma}

\begin{proof}
    For $s\in\ZZ$ and $a\in\OK\x$, by definition we have
    \begin{equation}\label{dim2 Eq congruence LT formula}
        \bbra{\id-\bbra{f_a^{\LT}}^ha}\!\bbra{T_{K,\sigma_0}^{-(q-1)s}}=T_{K,\sigma_0}^{-(q-1)s}\bbra{1-\bbra{f_a^{\LT}}^{h+(q-1)s}}.
    \end{equation}

    (i). Take $s=[h]_i$. Since $h+(q-1)[h]_i=[h]_{i+f}-[h]_i$ is a multiple of $p^{i+1}$ and $p^{i+1}\geq[h]_i+1$, we deduce from \eqref{dim2 Eq congruence LT formula} and Lemma \ref{dim2 Lem fa LT} that
    \begin{equation*}
        \bbra{\id-\bbra{f_a^{\LT}}^ha}\!\bbra{T_{K,\sigma_0}^{-(q-1)[h]_i}}\in T_{K,\sigma_0}^{-(q-1)[h]_i}\bbra{T_{K,\sigma_0}^{(q-1)p^{i+1}}\FF\ddbra{T_{K,\sigma_0}^{q-1}}}\subseteq T_{K,\sigma_0}^{q-1}\FF\ddbra{T_{K,\sigma_0}^{q-1}}.
    \end{equation*}
    
    (ii). Take $s=[h]_i+p^{i+1}$. We have 
    \begin{equation*}
        h+(q-1)([h]_i+p^{i+1})=[h]_{i+f}-[h]_i+qp^{i+1}-p^{i+1}\in (h_{i+1}-1)p^{i+1}+p^{i+2}\ZZ.
    \end{equation*}
    Then using $p^{i+1}\geq[h]_i+1$, we deduce from \eqref{dim2 Eq congruence LT formula} and Lemma \ref{dim2 Lem fa LT} that
    \begin{equation*}
    \begin{aligned}
        &\bbra{\id-\bbra{f_a^{\LT}}^ha}\!\bbra{T_{K,\sigma_0}^{-(q-1)([h]_i+p^{i+1})}}\\
        &\hspace{1.5cm}\in T_{K,\sigma_0}^{-(q-1)([h]_i+p^{i+1})}\bbra{(h_{i+1}-1)c_a^{p^{i+1}}T_{K,\sigma_0}^{(q-1)p^{i+1}}+T_{K,\sigma_0}^{2(q-1)p^{i+1}}\FF\ddbra{T_{K,\sigma_0}^{q-1}}}\\
        &\hspace{1.5cm}\subseteq(h_{i+1}-1)c_a^{p^{i+1}}T_{K,\sigma_0}^{-(q-1)[h]_i}+T_{K,\sigma_0}^{q-1}\FF\ddbra{T_{K,\sigma_0}^{q-1}}.
    \end{aligned}
    \end{equation*}
    
    (iii). Take $s=[h]_i+p^{i+1-f}$. We have 
    \begin{equation*}
    \begin{aligned}
        h+(q-1)([h]_i+p^{i+1-f})&=[h]_{i+f}-[h]_i+p^{i+1}-p^{i+1-f}\in-p^{i+1-f}+p^{i+1}\ZZ.
    \end{aligned}
    \end{equation*}
    Then we deduce from \eqref{dim2 Eq congruence LT formula} and Lemma \ref{dim2 Lem fa LT} that
    \begin{equation*}
    \begin{aligned}
        &\bbra{\id-\bbra{f_a^{\LT}}^ha}\!\bbra{T_{K,\sigma_0}^{-(q-1)([h]_i+p^{i+1-f})}}\\
        &\hspace{1.5cm}\in T_{K,\sigma_0}^{-(q-1)([h]_i+p^{i+1-f})}
        \begin{aligned}[t]
            &\Big(-c_a^{p^{i+1}}T_{K,\sigma_0}^{(q-1)p^{i+1-f}}+c_a^{p^i}T_{K,\sigma_0}^{(q-1)(p^{f-1}+1)p^{i+1-f}}\\
            &\hspace{5cm}+T_{K,\sigma_0}^{(q-1)(2p^{f-1}+1)p^{i+1-f}}\FF\ddbra{T_{K,\sigma_0}^{q-1}}\Big)
        \end{aligned}\\
        &\hspace{1.5cm}\subseteq-c_a^{p^{i+1}}T_{K,\sigma_0}^{-(q-1)[h]_i}+c_a^{p^i}T_{K,\sigma_0}^{-(q-1)[h]_{i-1}}+T_{K,\sigma_0}^{q-1}\FF\ddbra{T_{K,\sigma_0}^{q-1}},
    \end{aligned}
    \end{equation*}
    where the first inclusion uses $p\geq3$ (hence $p^f\geq2p^{f-1}+1$), and the second inclusion uses $h_i=1$ (hence $[h]_i=[h]_{i-1}+p^i<2p^i$).
\end{proof}

\begin{definition}\label{dim2 Def DLT}
    Let $0\leq h\leq q-2$, $\lambda_0,\lambda_1\in\FF\x$ and $0\leq j\leq f-1$. We define $D_j^{\LT},D_{\tr}^{\LT},D_{\unr}^{\LT}\in\FF\dbra{T_{K,\sigma_0}^{q-1}}$ as follows:
    \begin{enumerate}
    \item 
    If $h_j\neq0$, we define 
    \begin{equation*}
        D_j^{\LT}\eqdef T_{K,\sigma_0}^{-(q-1)[h]_{j-1}}.
    \end{equation*}
    If $h_j=0$, we let $0\leq r\leq f-1$ such that $h_{j+1}=\cdots=h_{j+r}=1$ and $h_{j+r+1}\neq1$, then we define
    \begin{equation*}
    \begin{aligned}
        D_j^{\LT}&\eqdef\lambda_0\lambda_1\inv\bbbra{T_{K,\sigma_0}^{-(q-1)([h]_{f+j+r}+p^{f+j+r+1})}+(h_{j+r+1}-1)\sum\limits_{i=0}^rT_{K,\sigma_0}^{-(q-1)([h]_{f+j+i}+p^{j+i+1})}}\\
        &=\lambda_0\lambda_1\inv
        \begin{aligned}[t]
            &\Bigg[T_{K,\sigma_0}^{-(q-1)\bra{h+q([h]_{j-1}+p^j(p+p^2+\cdots+p^{r+1}))}}\\
            &\hspace{1.5cm}+(h_{j+r+1}-1)\sum\limits_{i=0}^rT_{K,\sigma_0}^{-(q-1)\bra{h+q([h]_{j-1}+p^j((p+p^2+\cdots+p^i))+p^{j+i+1})}}\Bigg].
        \end{aligned}
    \end{aligned}
    \end{equation*}
    \item
    If $h=1+p+\cdots+p^{f-1}$ and $\lambda_0\lambda_1\inv=1$, we define
    \begin{equation*}
        D_{\tr}^{\LT}\eqdef\sum\limits_{i=0}^{f-1}T_{K,\sigma_0}^{-(q-1)([h]_{f+i-1}+p^{i})}=\sum\limits_{i=0}^{f-1}T_{K,\sigma_0}^{-(q-1)(1+p+\cdots+p^{i-1}+2p^{i}+p^{i+1}+\cdots+p^{f+i-1})}.
    \end{equation*}
    Otherwise (i.e.\,either $h\neq1+p+\cdots+p^{f-1}$ or $\lambda_0\lambda_1\inv\neq1$), we define $D_{\tr}^{\LT}\eqdef0$.
    \item
    If $h=0$ and $\lambda_0\lambda_1\inv=1$, we define $D_{\unr}^{\LT}\eqdef1$. Otherwise, we define $D_{\unr}^{\LT}\eqdef0$.
    \end{enumerate}
\end{definition}

\begin{corollary}\label{dim2 Cor congruence LT}
    Let $0\leq h\leq q-2$ and $\lambda_0,\lambda_1\in\FF\x$.
    \begin{enumerate}
    \item 
    For all $0\leq j\leq f-1$ and $a\in\OK\x$, we have
    \begin{equation*}
        \bbra{\id-\bbra{f_a^{\LT}}^ha}\!\bigbra{D_j^{\LT}}\in T_{K,\sigma_0}^{q-1}\FF\ddbra{T_{K,\sigma_0}^{q-1}}.
    \end{equation*}
    \item 
    If $h=1+p+\cdots+p^{f-1}$ and $\lambda_0\lambda_1\inv=1$, then for all $a\in\OK\x$, we have
    \begin{equation*}
        \bbra{\id-\bbra{f_a^{\LT}}^ha}\!\bigbra{D_{\tr}^{\LT}}\in\bbra{\id-T_{K,\sigma_0}^{-(q-1)h}\varphi_q}\!\bbra{c_a^{p^{f-1}}T_{K,\sigma_0}^{-(q-1)(1+p+\cdots+p^{f-2})}}+T_{K,\sigma_0}^{q-1}\FF\ddbra{T_{K,\sigma_0}^{q-1}}.
    \end{equation*}
    \end{enumerate}
\end{corollary}

\begin{proof}
    This follows from Lemma \ref{dim2 Lem congruence LT}. Note that for $i$ such that $h_i=0$ we have $[h]_{i}=[h]_{i-1}$.
\end{proof}

\begin{lemma}\label{dim2 Lem unique LT}
    Let $0\leq h\leq q-2$ and $\lambda_0,\lambda_1\in\FF\x$.
    \begin{enumerate}
        \item
        For any $y\in T_{K,\sigma_0}^{q-1}\FF\ddbra{T_{K,\sigma_0}^{q-1}}$, the equation $\bigbra{\id-\lambda_0\lambda_1\inv T_{K,\sigma_0}^{-(q-1)h}\varphi_q}(x)=y$ has a unique solution in $T_{K,\sigma_0}^{q-1}\FF\ddbra{T_{K,\sigma_0}^{q-1}}$, given by the convergent series $x=\sum\nolimits_{n=0}^{\infty}\bigbra{\lambda_0\lambda_1\inv T_{K,\sigma_0}^{-(q-1)h}\varphi_q}^n(y)$.
        \item
        For any $y\in\FF\dbra{T_{K,\sigma_0}^{q-1}}$, the equation $\bigbra{\id-\lambda_0\lambda_1\inv T_{K,\sigma_0}^{-(q-1)h}\varphi_q}(x)=y$ has at most one solution in $\FF\dbra{T_{K,\sigma_0}^{q-1}}$ unless $h=0$ and $\lambda_0\lambda_1\inv=1$.
        \item 
        We let 
        \begin{equation}\label{dim2 Eq unique LT y}
            y=\sum\limits_{i=0}^{m}a_{i}T_{K,\sigma_0}^{-(q-1)(h+qh+q^2i))}+\sum\limits_{j=0}^{n}b_{j}T_{K,\sigma_0}^{-(q-1)(h+qj)}+\sum\limits_{k=0}^{t}c_{k}T_{K,\sigma_0}^{-(q-1)k}
        \end{equation}
        with $m,n\geq-1$, $t\geq0$, $a_i,b_j,c_k\in\FF$, $a_m\neq0$, $b_n\neq0$, $c_t\neq0$ and $t\notin h+q\ZZ$. If $m,n<t$, then the equation $\bigbra{\id-\lambda_0\lambda_1\inv T_{K,\sigma_0}^{-(q-1)h}\varphi_q}(x)=y$ has no solution in $\FF\dbra{T_{K,\sigma_0}^{q-1}}$.
    \end{enumerate}
\end{lemma}

\begin{proof}
    (i). The proof is similar to that of Lemma \ref{dim2 Lem unique Ainfty'} below using $h<q-1$. We omit the details. 
    
    (ii). It suffices to show that the equality 
    \begin{equation}\label{dim2 Eq unique LT 1}
        \varphi_q(x)=\lambda_0\inv\lambda_1T_{K,\sigma_0}^{(q-1)h}x
    \end{equation}
    for $x\in\FF\dbra{T_{K,\sigma_0}^{q-1}}$ implies $x=0$ unless $h=0$ and $\lambda_0\lambda_1\inv=1$. 
    
    First we assume that $h\neq0$. If $x\neq0$, we assume that the lowest degree term of $x$ has degree $(q-1)s$ for $s\in\ZZ$, then the lowest degree on both sides of (\ref{dim2 Eq unique LT 1}) are $(q-1)qs$ and $(q-1)(s+h)$, which cannot be equal since $0<h<q-1$. Hence we must have $x=0$.

    Next we assume that $h=0$ and $\lambda_0\lambda_1\inv\neq1$. We let $m\geq0$ be large enough so that $(\lambda_0\lambda_1\inv)^{m}=1$ and $q^m\geq|\FF|$, then $\varphi_q^m$ acts as $x\mapsto x^{q^m}$ on $\FF\dbra{T_{K,\sigma_0}^{q-1}}$, and by (\ref{dim2 Eq unique LT 1}) we have $x^{q^m}=\varphi_q^m(x)=x$, hence $x\in\FF$. Since $\lambda_0\lambda_1\inv\neq1$, by (\ref{dim2 Eq unique LT 1}) again we conclude that $x=0$.

    (iii). Suppose that $\bigbra{\id-\lambda_0\lambda_1\inv T_{K,\sigma_0}^{-(q-1)h}\varphi_q}(x)=y$ for $x\in\FF\dbra{T_{K,\sigma_0}^{q-1}}$. Then we have
    \begin{equation}\label{dim2 Eq unique LT 2}
        \bbra{\id-\lambda_0\lambda_1\inv T_{K,\sigma_0}^{-(q-1)h}\varphi_q}(z)=\sum\limits_{k=0}^tc'_kT_{K,\sigma_0}^{-(q-1)k},
    \end{equation}
    where
    \begin{equation*}
        z\eqdef x+\bigbra{\lambda_0\inv\lambda_1}^2\sum\limits_{i=0}^ma_iT_{K,\sigma_0}^{-(q-1)i}+\lambda_0\inv\lambda_1\sum\limits_{i=0}^ma_iT_{K,\sigma_0}^{-(q-1)(h+qi)}+\lambda_0\inv\lambda_1\sum\limits_{j=0}^nb_jT_{K,\sigma_0}^{-(q-1)j}
    \end{equation*}
    and $c'_k\in\FF$, and we have $c'_t=c_t\neq0$ since $m,n<t$. 
    
    We write $z=c_sT_{K,\sigma_0}^{-(q-1)s}+\text{(terms with degree $>-(q-1)s$)}$. Since the RHS of \eqref{dim2 Eq unique LT 2} does not lie in $T_{K,\sigma_0}^{q-1}\FF\ddbra{T_{K,\sigma_0}^{q-1}}$, we must have $s\geq0$ (since $h<q-1$), hence the lowest degree term of the LHS of \eqref{dim2 Eq unique LT 2} has degree $-(q-1)(h+qs)$. However, the lowest degree term of the RHS of \eqref{dim2 Eq unique LT 2} has degree $-(q-1)t$, which does not lie in $-(q-1)(h+q\ZZ)$ by assumption. This is a contradiction.
\end{proof}

\begin{proposition}\label{dim2 Prop pair LT}
    Let $0\leq h\leq q-2$ and $\lambda_0,\lambda_1\in\FF\x$.
\begin{enumerate}
    \item 
    For all $0\leq j\leq f-1$, the tuple $\bigbra{D,(E_a)_{a\in\OK\x}}$ with
    \begin{equation*}
    \begin{cases}
        D&=D_j^{\LT}\\
        E_a&=E_{j,a}^{\LT}\eqdef\bbra{\id-\lambda_0\lambda_1\inv T_{K,\sigma_0}^{-(q-1)h}\varphi_q}\inv\bbbra{\bbra{\id-\bbra{f_a^{\LT}}^ha}\!\bigbra{D_j^{\LT}}}\\
        &=\sum\limits_{n=0}^{\infty}\bbra{\lambda_0\lambda_1\inv T_{K,\sigma_0}^{-(q-1)h}\varphi_q}^n\bbbra{\bbra{\id-\bbra{f_a^{\LT}}^ha}\!\bigbra{D_j^{\LT}}}
    \end{cases}
    \end{equation*}
    defines an element of $W^{\LT}$. We denote it by $[B_j^{\LT}]$.
    \item
    If $h=1+p+\cdots+p^{f-1}$ and $\lambda_0\lambda_1\inv=1$, then the tuple $\bigbra{D,(E_a)_{a\in\OK\x}}$ with
    \begin{equation*}
    \begin{cases}
        D&=D_{\tr}^{\LT}\\
        E_a&=E_{\tr,a}^{\LT}\eqdef\bbra{\id-T_{K,\sigma_0}^{-(q-1)(1+p+\cdots+p^{f-1})}\varphi_q}\inv\bbbra{\bbra{\id-\bbra{f_a^{\LT}}a}\!\bigbra{D_{\tr}^{\LT}}}\\
        &=c_a^{p^{f-1}}T_{K,\sigma_0}^{-(q-1)(1+p+\cdots+p^{f-2})}+\sum\limits_{n=0}^{\infty}\bbra{T_{K,\sigma_0}^{-(q-1)(1+p+\cdots+p^{f-1})}\varphi_q}^n\Big[\bbra{\id-\bbra{f_a^{\LT}}a}\!\bigbra{D_{\tr}^{\LT}}\\
        &\hspace{1.5cm}-\bbra{\id-T_{K,\sigma_0}^{-(q-1)(1+p+\cdots+p^{f-1})}\varphi_q}\!\bbra{c_a^{p^{f-1}}T_{K,\sigma_0}^{-(q-1)(1+p+\cdots+p^{f-2})}}\Big]  
    \end{cases}
    \end{equation*}
    defines an element of $W^{\LT}$. We denote it by $[B_{\tr}^{\LT}]$. Otherwise, we define $E_{\tr,a}^{\LT}\eqdef0$ for all $a\in\OK\x$ and $[B_{\tr}^{\LT}]\eqdef[0]$ in $W^{\LT}$.
    \item
    If $h=0$ and $\lambda_0\lambda_1\inv=1$, then the tuple $\bigbra{D,(E_a)_{a\in\OK\x}}$ with
    \begin{equation*}
    \begin{cases}
        D&=D_{\unr}^{\LT}=1\\
        E_a&=E_{\unr,a}^{\LT}\eqdef0
    \end{cases}    
    \end{equation*}
    defines an element of $W^{\LT}$. We denote it by $[B_{\unr}^{\LT}]$. Otherwise, we define $E_{\unr,a}^{\LT}\eqdef0$ for all $a\in\OK\x$ and $[B_{\unr}^{\LT}]\eqdef[0]$ in $W^{\LT}$.
\end{enumerate}
\end{proposition}

\begin{proof}
    (iii) is direct. For (i) and (ii), each $E_a$ is well-defined by Corollary \ref{dim2 Cor congruence LT} and Lemma \ref{dim2 Lem unique LT}(i), and condition (ii) in Definition \ref{dim2 Def WLT} is guaranteed by the uniqueness of solution in Lemma \ref{dim2 Lem unique LT}(i),(ii).
\end{proof}

\begin{remark}\label{dim2 Rk LT h=0}
    Suppose that $h=0$ and $\lambda_0\lambda_1\inv=1$. For $0\leq j\leq f-1$, we let $[B_j]$ be the element of $W^{\LT}$ defined by the tuple $\bigbra{D,(E_a)_{a\in\OK\x}}$ with $D=0$ and $E_a=c_a^{p^j}$. Then we have $[B_j]=-[B_{j+1}^{\LT}]$ for $0\leq j\leq f-2$ and $[B_{f-1}]=-[B_0^{\LT}]$ in $W^{\LT}$.
\end{remark}


\begin{theorem}\label{dim2 Thm basis LT}
    Let $0\leq h\leq q-2$ and $\lambda_0,\lambda_1\in\FF\x$.
    \begin{enumerate}
        \item 
        If $h=0$ and $\lambda_0\lambda_1\inv=1$, then $\bigset{[B_0^{\LT}],\ldots,[B_{f-1}^{\LT}],[B_{\unr}^{\LT}]}$ form a basis of $W^{\LT}$. 
        \item
        If $h=1+p+\cdots+p^{f-1}$ and $\lambda_0\lambda_1\inv=1$, then $\bigset{[B_0^{\LT}],\ldots,[B_{f-1}^{\LT}],[B_{\tr}^{\LT}]}$ form a basis of $W^{\LT}$. 
        \item
        In the remaining cases, $\bigset{[B_0^{\LT}],\ldots,[B_{f-1}^{\LT}]}$ form a basis of $W^{\LT}$.
    \end{enumerate}
\end{theorem}

\begin{remark}
    If $h=1+p+\cdots+p^{f-1}$ and $\lambda_0\lambda_1\inv=1$, then $\bigset{[B_0^{\LT}],\ldots,[B_{f-1}^{\LT}]}$ form a basis of the subspace of $W^{\LT}$ which corresponds to peu ramifi\'ees representations under \eqref{dim2 Eq Fon Ext LT}.
\end{remark}

\begin{proof}[Proof of Theorem \ref{dim2 Thm basis LT}]
    By (\ref{dim2 Eq Fon Ext LT}), we have $\dim_{\FF}W^{\LT}=\dim_{\FF}H^1\bigbra{G_K,\FF\bigbra{\omega_f^h\unr(\lambda_0\lambda_1\inv)}}=f$ except the cases ($h=0$, $\lambda_0\lambda_1\inv=1$) and ($h=1+p+\cdots+p^{f-1}$, $\lambda_0\lambda_1\inv=1$), in which case the dimension is $f+1$. So it is enough to show that the elements of $W^{\LT}$ as in the statements are $\FF$-linearly independent (using Definition \ref{dim2 Def WLT}(iv)). 

    \hspace{\fill}

    (iii). Suppose that $\sum\nolimits_{j=0}^{f-1}c_j[B_j^{\LT}]=[0]$ in $W^{\LT}$. By definition, there exists $b\in\FF\dbra{T_{K,\sigma_0}^{q-1}}$ such that 
    \begin{equation}\label{dim2 Eq basis-1}
        \bbra{\id-\lambda_0\lambda_1\inv T_{K,\sigma_0}^{-(q-1)h}\varphi_q}(b)=\scalebox{1}{$\sum\limits_{j=0}^{f-1}$}c_jD_j^{\LT}.
    \end{equation}

    \hspace{\fill}

    \noindent\textbf{Step 1.} Assuming $h\neq0$, we prove that $c_j=0$ for all $j$ such that $h_j=0$.
    
    By symmetry (since one can replace $D_{K,\sigma_0}$ with $D_{K,\sigma_i}$ if necessary), it is enough to prove that $c_{f-2}=0$ assuming $h_{f-2}=0$ (which implies $f\geq2$ since $h\neq0$). Suppose on the contrary that $c_{f-2}\neq0$. 

    For each $0\leq j\leq f-1$ such that $h_j=0$, we let $0\leq r\leq f-1$ be the corresponding integer in Definition \ref{dim2 Def DLT}(i). Since $h_{f-2}=0$, we have $r\leq f-2$ if $j=f-1$ and $r+j\leq f-3$ if $0\leq j\leq f-3$. 
    \begin{enumerate}
        \item[$\bullet$]
        If $j+r\geq f-1$, then we have 
        \begin{equation*}
            [h]_{f+j+r}+p^{f+j+r+1}=h+qh+q^2\bigbra{[h]_{j+r-f}+p^{j+r+1-f}}\leq h+qh+q^2\bigbra{[h]_{f-2}+p^{f-1}}.        \end{equation*}
        \item[$\bullet$]
        If $j+r\leq f-2$, then we have
        \begin{equation*}
            [h]_{f+j+r}+p^{f+j+r+1}=h+q\bigbra{[h]_{j+r}+p^{j+r+1}}\leq h+q\bigbra{[h]_{f-2}+p^{f-1}}.
        \end{equation*}
        \item[$\bullet$] 
        If $0\leq i\leq r$ such that $j+i\geq f-1$, then we have (since $r\neq f-1$ if $j=f-1$)
        \begin{equation*}
            [h]_{f+j+i}+p^{j+i+1}=h+q\bigbra{[h]_{j+i}+p^{j+i+1-f}}< h+q\bigbra{[h]_{2f-2}+p^{f-1}}.
        \end{equation*}
        \item[$\bullet$] 
        If $0\leq i\leq r$ such that $j+i\leq f-2$, then we have 
        $[h]_{f+j+i}+p^{j+i+1}\leq[h]_{2f-2}+p^{f-1}$, with equality holds if and only if $j+i=f-2$, which implies $j=f-2$ and $i=0$ since $r+j\leq f-3$ if $0\leq j\leq f-3$.
    \end{enumerate}
    In particular, by the definition of $D_j^{\LT}$ together with $c_{f-2}\neq0$ and $[h]_{f-2}<[h]_{2f-2}$ (since $h\neq0$), the RHS of \eqref{dim2 Eq basis-1} has the form \eqref{dim2 Eq unique LT y} with $t=[h]_{2f-2}+p^{f-1}$ and $m,n<t$. Then we deduce a contradiction by Lemma \ref{dim2 Lem unique LT}(iii).

    \hspace{\fill}

    \noindent\textbf{Step 2.} Assuming $h\neq0$, we prove that $c_j=0$ for all $j$.

    By Step 1, we already know that $c_j=0$ for all $0\leq j\leq f-1$ such that $h_j=0$. Suppose on the contrary that $c_j\neq0$ for some $j$. We let $j_0$ be the largest integer in $\set{0,1,\ldots,f-1}$ such that $h_{j_0}\neq0$. Then we have $[h]_{j_0-1}\notin h+q\ZZ$. By the definition of $D_j^{\LT}$ (in the case $h_j\neq0$) the RHS of \eqref{dim2 Eq basis-1} has the form (\ref{dim2 Eq unique LT y}) with $m=n=-1$ and $t=[h]_{j_0-1}$. Then we deduce a contradiction by Lemma \ref{dim2 Lem unique LT}(iii).

    \hspace{\fill}

    \noindent\textbf{Step 3.} Assuming $h=0$ (hence $\lambda_0\lambda_1\inv\neq1$ by assumption), we prove that $c_j=0$ for all $j$.
    
    By definition we have $D_j^{\LT}=\lambda_0\lambda_1\inv T_{K,\sigma_0}^{-(q-1)p^{f+j+1}}-\lambda_0\lambda_1\inv T_{K,\sigma_0}^{-(q-1)p^{j+1}}$ for all $0\leq j\leq f-1$. Then by replacing $b$ with $b+\bigbra{\lambda_0\inv\lambda_1-1}c_{f-1}+\sum\nolimits_{j=0}^{f-1}c_jT_{K,\sigma_0}^{-(q-1)p^{j+1}}$ in (\ref{dim2 Eq basis-1}), the RHS of (\ref{dim2 Eq basis-1}) becomes $\sum\nolimits_{j=0}^{f-1}c'_jT_{K,\sigma_0}^{-(q-1)p^j}$ with $c'_0=\bigbra{\lambda_0\inv\lambda_1-1}c_{f-1}$ and $c'_j=\bigbra{1-\lambda_0\lambda_1\inv}c_{j-1}$ for $1\leq j\leq f-1$. Suppose on the contrary that $c_j\neq0$ for some $j$. We let $j_0$ be the largest integer in $\set{0,1,\ldots,f-1}$ such that $c'_{j_0}\neq0$ (which exists since $\lambda_0\lambda_1\inv\neq1$). Then we deduce a contradiction by Lemma \ref{dim2 Lem unique LT}(iii) with $m=n=-1$ and $t=p^{j_0}$. 
    
    \hspace{\fill}

    (i). Let $h=0$ and $\lambda_0\lambda_1\inv=1$. Suppose that $c_{\unr}[B_{\unr}^{\LT}]+\sum\nolimits_{j=0}^{f-1}c_j[B_j^{\LT}]=[0]$ in $W^{\LT}$. By Proposition \ref{dim2 Prop pair LT}(iii) and Remark \ref{dim2 Rk LT h=0}, the element $c_{\unr}[B_{\unr}^{\LT}]+\sum\nolimits_{j=0}^{f-1}c_j[B_j^{\LT}]\in W^{\LT}$ is represented by the tuple $\bigbra{D,(E_a)_{a\in\OK\x}}$ with 
    \begin{equation*}
    \begin{cases}
        D&=c_{\unr}\\
        E_a&=-c_0c_a^{p^{f-1}}-\sum\nolimits_{j=1}^{f-1}c_jc_a^{p^{j-1}}.
    \end{cases}
    \end{equation*}
    Since $\Im(\id-\varphi_q)\cap\FF=\set{0}$, we deduce from Definition \ref{dim2 Def WLT}(iv) that $c_{\unr}=0$. Since the characters $c_a,c_a^p,\ldots,c_a^{p^{f-1}}$ are linearly independent (using for example Remark \ref{dim2 Rk ca LT}(ii)) and since $\Ker(\id-\varphi_q)=\FF$, we deduce from Definition \ref{dim2 Def WLT}(iv) that $c_j=0$ for all $j$.

    \hspace{\fill}
    
    (ii). Let $h=1+p+\cdots+p^{f-1}$ and $\lambda_0\lambda_1\inv=1$. Suppose that $c_{\tr}[B_{\tr}^{\LT}]+\sum\nolimits_{j=0}^{f-1}c_j[B_j^{\LT}]=[0]$ in $W^{\LT}$. If $c_{\tr}=0$, then the proof of (iii) shows that $c_j=0$ for all $j$, which proves (ii). If $c_{\tr}\neq0$, then by the definition of $D_{\tr}^{\LT}$ and $D_j^{\LT}$ (in the case $h_j\neq0$), and since $[h]_{f+i-1}+p^i\notin h+q\ZZ$ for all $0\leq i\leq f-1$, the sum $c_{\tr}D_{\tr}^{\LT}+\sum\nolimits_{j=0}^{f-1}c_jD_j^{\LT}$ has the form (\ref{dim2 Eq unique LT y}) with $m=n=-1$ and $t=[h]_{2f-2}+p^{f-1}$. Then we deduce a contradiction by Lemma \ref{dim2 Lem unique LT}(iii).
\end{proof}

\section{\'Etale \texorpdfstring{$(\varphi,\OK\x)$}.-modules over \texorpdfstring{$A$}.}\label{dim2 Sec A}
 
In this section, we give an explicit construction of some \'etale $(\varphi,\OK\x)$-modules over $A$ of rank $2$ that will be needed in \S\ref{dim2 Sec DA0}. The main construction is Proposition \ref{dim2 Prop pair A}. We also give a comparison between some of these \'etale $(\varphi,\OK\x)$-modules that are constructed using different systems of variables, see Proposition \ref{dim2 Prop XY}.

First we recall the definition of the ring $A$. Let $\fm_{\OK}$ be the maximal ideal of the Iwasawa algebra $\FF\ddbra{\OK}$. For $j\in\cJ$, we define
\begin{equation*}
    Y_j\eqdef\sum\limits_{a\in\Fq\x}a^{-p^j}\delta_{[a]}\in\fm_{\OK}\setminus\fm_{\OK}^2,
\end{equation*}
where $\delta_{[a]}\in\FF\ddbra{\OK}$ corresponds to $[a]\in\OK$. Then we have $\FF\ddbra{\OK}=\FF\ddbra{Y_0,\ldots,Y_{f-1}}$ and $\fm_{\OK}=(Y_0,\ldots,Y_{f-1})$. Consider the multiplicative subset $S\eqdef\sset{(Y_0\cdots Y_{f-1})^n:n\geq0}$ of $\FF\ddbra{\OK}$. Then $A\eqdef\wh{\FF\ddbra{\OK}_S}$ is the completion of the localization $\FF\ddbra{\OK}_S$ with respect to the $\fm_{\OK}$-adic filtration 
\begin{equation*}
    F_n\bbra{\FF\ddbra{\OK}_S}=\bigcup\limits_{k\geq0}\frac{1}{(Y_0\cdots Y_{f-1})^k}\fm_{\OK}^{kf-n},
\end{equation*}
where $\fm_{\OK}^m\eqdef\FF\ddbra{\OK}$ if $m\leq0$. We denote by $F_nA$ ($n\in\ZZ$) the induced filtration on $A$ and endow $A$ with the associated topology (\cite[\S1.3]{LvO96}). There is an $\FF$-linear action of $\OK\x$ on $\FF\ddbra{\OK}$ given by multiplication, and an $\FF$-linear Frobenius $\varphi$ on $\FF\ddbra{\OK}$ given by multiplication by $p$. They extend canonically by continuity to commuting continuous $\FF$-linear actions of $\varphi$ and $\OK\x$ on $A$ which satisfies (for each $j\in\cJ$)
\begin{equation}\label{dim2 Eq phiOK Y}
\begin{aligned}
    \varphi(Y_j)&=Y_{j-1}^p;\\
    [a](Y_j)&=a^{p^j}Y_j~\forall\,a\in\Fq\x.
\end{aligned}    
\end{equation}

Then we introduce another system of variables for $\FF\ddbra{\OK}$ following \cite{BHHMS3}. For $R$ a perfectoid $\FF$-algebra, we denote by $R^{\circ}$ the subring of power-bounded elements in $R$ and by $R^{\circ\circ}\subseteq R^{\circ}$ the subset of topologically nilpotent elements. We let $\mathbf{B}^+(R)$ be the Fr\'echet $K$-algebra defined as the completion of $W(R^{\circ})[1/p]$ for the family of norms $\abs{\cdot}_{\rho}$ for $0\leq\rho\leq1$ given by $\bigabs{\sum\nolimits_{n\gg-\infty}[x_n]p^n}_{\rho}\eqdef\sup\nolimits_{n\in\ZZ}\set{\abs{x_n}\rho^n}$.
Then as in \cite[p.27]{BHHMS3}, there exist elements $X_0,\ldots,X_{f-1}\in\FF\ddbra{\OK}$ satisfying $\FF\ddbra{\OK}=\FF\ddbra{X_0,\ldots,X_{f-1}}$ and such that for any perfectoid $\FF$-algebra $R$ we have an isomorphism of $K$-vector spaces
\begin{equation}\label{dim2 Eq iso for Xj}
\begin{aligned}
    \Hom_{\FF\text{-}\alg}^{\cont}\bigbra{\FF\ddbra{K},R}=\Hom_{\FF\text{-}\alg}^{\cont}\bigbra{\FF\ddbra{\OK},R}&\cong\mathbf{B}^+(R)^{\varphi_q=p^f}\\
    \bigbra{X_i\mapsto x_i\in R^{\circ\circ}}_{0\leq i\leq f-1}&\mapsto\sum\limits_{i=0}^{f-1}\sum\limits_{n\in\ZZ}[x_i^{p^{-i-nf}}]p^{i+nf},
\end{aligned}
\end{equation}
where $\FF\ddbra{K}$ is the $\fm_{\OK}$-adic completion of $\FF[K]\otimes_{\FF[\OK]}\FF\ddbra{\OK}$ and $K$ acts on $\FF\ddbra{K}$ by multiplication. By \cite[(41)]{BHHMS3} we have (for $0\leq i\leq f-1$)
\begin{equation}\label{dim2 Eq phiOK X}
\begin{aligned}
    \varphi(X_i)&=X_{i-1}^p;\\
    [a](X_i)&=a^{p^i}X_i\ \forall\,a\in\FF_q\x,
\end{aligned}
\end{equation}
where we extend the definition of $X_i$ to all $i\in\ZZ$ by the relation $X_{i+f}=X_i$.

By considering the $[\Fq\x]$-action in (\ref{dim2 Eq phiOK Y}) and (\ref{dim2 Eq phiOK X}) (see \cite[(55)]{BHHMS3}), for each $0\leq i\leq f-1$ there exists $\mu_i\in\FF\x$ such that 
\begin{equation}\label{dim2 Eq XY}
    Y_i=\mu_iX_i+\text{(degree $\geq2$ in the variables $X_i$)}\text{ and }Y_i\in \mu_iX_i(1+F_{1-p}A).
\end{equation}
In particular, for each $i$ we have $Y_i^{1-\varphi}/X_i^{1-\varphi}\in1+F_{1-p}A$. Here, for $a\in A\x$ and $k=\sum_{i=0}^mk_i\varphi^i\in\ZZ[\varphi]$ with $m\in\NNN$ and $k_i\in\ZZ$ for all $0\leq i\leq m$, we write $a^k\eqdef\prod\nolimits_{i=0}^m\varphi^i(a^{k_i})\in A\x$. This makes $A\x$ a $\ZZ[\varphi]$-module. Moreover, $1+F_{-1}A$ is a $\Zp[\varphi]$-module by completeness.

For $a\in\OK\x$ and $0\leq j\leq f-1$, we set:
\begin{equation*}
    \begin{aligned}
        f_{a,j}&\eqdef\ovl{a}^{p^j}X_j/a(X_j)\in1+F_{1-p}A;\\
        f_{a,\sigma_j}&\eqdef\ovl{a}^{p^j}Y_j/a(Y_j)\in1+F_{1-p}A  .
    \end{aligned}
\end{equation*}
As in \cite[(25)]{BHHMS3}, for $0\leq h\leq q-2$ and $\lambda\in\FF\x$ we define the \'etale $(\varphi_q,\OK\x)$-module $D_{A,\sigma_0}\bigbra{\omega_f^h\unr(\lambda)}$ over $A$ as follows $(a\in\OK\x)$:
\begin{equation}\label{dim2 Eq character A}
\left\{\begin{array}{cll}
    D_{A,\sigma_0}(\omega_f^h\unr(\lambda))&=&Ae\\
    \varphi_q(e)&=&\lambda X_0^{h(1-\varphi)}e\\
    a(e)&=&f_{a,0}^{h(1-\varphi)/(1-q)}e.
\end{array}\right.
\end{equation}
Using (\ref{dim2 Eq XY}), we get an isomorphic \'etale $(\varphi_q,\OK\x)$-module over $A$ if we replace $X_0$ by $Y_0$ (and thus $f_{a,0}$ by $f_{a,\sigma_0}$).

\begin{definition}\label{dim2 Def WX}
    Let $0\leq h\leq q-2$ and $\lambda_0,\lambda_1\in\FF\x$. We define $W^{X}$ to be the set of equivalence classes of tuples $[B]=\bigbra{D,(E_a)_{a\in\OK\x}}$ such that 
    \begin{enumerate}
        \item 
        $D\in A$, $E_a\in A$ for all $a\in\OK\x$, and the map $\OK\x\to A$, $a\mapsto E_a$ is continuous;
        \item
        $E_{ab}=E_a+f_{a,0}^{h(1-\varphi)/(1-q)}a(E_{b})$ for all $a,b\in\OK\x$;
        \item
        $\bigbra{\id-\lambda_0\lambda_1\inv X_0^{h(1-\varphi)}\varphi_q}(E_a)=\bigbra{\id-f_{a,0}^{h(1-\varphi)/(1-q)}a}(D)$ for all $a\in\OK\x$;
        \item
        two tuples $\bigbra{D,(E_a)_{a\in\OK\x}}$ and $\bigbra{D',(E_a')_{a\in\OK\x}}$ are equivalent if and only if there exists $b\in A$ such that
        \begin{equation*}
        \left\{\begin{aligned}
            D'&=D+\bbra{\id-\lambda_0\lambda_1\inv X_0^{h(1-\varphi)}\varphi_q}(b)\\
            E_a'&=E_a+\bbra{\id-f_{a,0}^{h(1-\varphi)/(1-q)}a}(b)\quad\forall\,a\in\OK\x.
        \end{aligned}\right.
        \end{equation*}
    \end{enumerate}
    It has a natural structure of an $\FF$-vector space.

    We define $W^Y$ in a similar way replacing $X_0$ by $Y_0$.
\end{definition}

By the definition of $W^X$, there is an isomorphism of $\FF$-vector spaces 
\begin{equation*}
    W^{X}\cong\Ext^1\bbra{D_{A,\sigma_0}\bigbra{\!\unr(\lambda_1)},D_{A,\sigma_0}\bigbra{\omega_f^h\unr(\lambda_0)}},
\end{equation*}
where $\Ext^1$ is defined in the category of \'etale $(\varphi_q,\OK\x)$-modules over $A$. For $[B]=\bigbra{D,(E_a)_{a\in\OK\x}}\in W^{X}$, we denote by $D([B])$ the corresponding \'etale $(\varphi_q,\OK\x)$-module over $A$. It has an $A$-basis with respect to which the matrices of the actions of $\varphi_q$ and $\OK\x$ have the form (using (\ref{dim2 Eq character A}))
\begin{equation*}
\left\{\begin{array}{cll}
    \Mat(\varphi_q)&=&\pmat{\lambda_0 X_0^{h(1-\varphi)}&\lambda_1D\\0&\lambda_1}\\
    \Mat(a)&=&\pmat{f_{a,0}^{h(1-\varphi)/(1-q)}&E_a\\0&1}\quad\forall\,a\in\OK\x.
\end{array}\right.
\end{equation*}
Note that $D([B])\cong D(\lambda[B])$ as \'etale $(\varphi_q,\OK\x)$-modules over $A$ for $\lambda\in\FF\x$.

We denote by $A_{\infty}$ the completed perfection of $A$ (see \cite[Lemma~2.4.2(i)]{BHHMS3}).

\begin{lemma}\label{dim2 Lem fa A}
    Let $0\leq j\leq f-1$. We have $f_{a,j}=f_{a,\sigma_j}=1$ for all $a\in[\Fq\x]$. More generally we have for $a\in\OK\x$
    \begin{equation}\label{dim2 Eq fa A statement}
    \begin{aligned}
        f_{a,j}\inv&\in1+c_a^{p^j}X_j^{\varphi-1}-c_a^{p^{j-1}}X_j^{\varphi-1}X_{j-1}^{\varphi-1}+F_{3-3p}A;\\
        f_{a,\sigma_j}\inv&\in1+c_a^{p^j}Y_j^{\varphi-1}-c_a^{p^{j-1}}Y_j^{\varphi-1}Y_{j-1}^{\varphi-1}+F_{3-3p}A,
    \end{aligned}
    \end{equation}
    where $c_a$ is as in Lemma \ref{dim2 Lem fa LT}.
\end{lemma}

\begin{proof}
    Recall that we have $\FF\ddbra{\OK}=\FF\ddbra{X_0,\ldots,X_{f-1}}=\FF\ddbra{Y_0,\ldots,Y_{f-1}}$ with maximal ideal $\fm_{\OK}=(X_0,\ldots,X_{f-1})=(Y_0,\ldots,Y_{f-1})$.
    
    If $a\in[\Fq\x]$, then we have $f_{a,j}=f_{a,\sigma_j}=1$ for all $0\leq j\leq f-1$ by (\ref{dim2 Eq phiOK Y}) and (\ref{dim2 Eq phiOK X}).

    If $a=1+p^2b$ for some $b\in\OK$. Then for each $x\in\OK$, we have (recall that $\delta_x\in\FF\ddbra{\OK}$ corresponds to $x$)
    \begin{equation*}
        a(\delta_x)=\delta_{(1+p^2b)x}=\delta_x+(\delta_{p^2b}-1)\delta_x=\delta_x+\bigbra{1+(\delta_b-1)^{p^2}}\delta_x\in\delta_x+\fm_{\OK}^{p^2}.
    \end{equation*}
    From this we deduce that (for all $0\leq j\leq f-1$)
    \begin{equation*}
    \begin{aligned}
        a(X_j)&\in X_j(1+F_{1-p^2}A);\\
        a(Y_j)&\in Y_j(1+F_{1-p^2}A).
    \end{aligned}
    \end{equation*}
    Hence the lemma holds (since $p^2-1\geq3p-3$ and $c_a=0$ for $a=1+p^2b$).

    \hspace{\fill}
    
    It remains to prove the lemma for $a=1+p[\mu]$ with $\mu\in\Fq\x$. We refer to \cite[\S1.10.2]{FF18} for the definition of the ring of Witt bi-vectors $BW(A_{\infty})$. Since the isomorphism (\ref{dim2 Eq iso for Xj}) respects the $\OK\x$-actions, we have equality in $\mathbf{B}^+(A_{\infty})^{\varphi_q=p^f}$ (which equals $BW(A_{\infty})^{\varphi_q=p^f}$ by \cite[Prop.~4.2.1]{FF18}):
    \begin{equation}\label{dim2 Eq faX-0}
    \begin{aligned}
        \sum\limits_{i=0}^{f-1}\sum\limits_{n\in\ZZ}[a(X_i)^{p^{-i-nf}}]p^{i+nf}&=a\sum\limits_{i=0}^{f-1}\sum\limits_{n\in\ZZ}[X_i^{p^{-i-nf}}]p^{i+nf}\\
        &=\sum\limits_{i=0}^{f-1}\sum\limits_{n\in\ZZ}[X_i^{p^{-i-nf}}]p^{i+nf}+\sum\limits_{i=0}^{f-1}\sum\limits_{n\in\ZZ}[\mu X_i^{p^{-i-nf}}]p^{i+nf+1}\\
        &=\sum\limits_{i=0}^{f-1}\sum\limits_{n\in\ZZ}\bbra{[X_i^{p^{-i-nf}}]+[(\mu^{p^i}X_{i-1}^p)^{p^{-i-nf}}]}p^{i+nf},
    \end{aligned}
    \end{equation}
    where the last equality follows from a reindexation.

    For $n\in\NNN$, we let $S_n\in\ZZ[a_0,\ldots,a_n,b_0,\ldots,b_n]$ be the additional law of the Witt vectors, given inductively by the equalities in $\ZZ[a_0,\ldots,a_n,b_0,\ldots,b_n]$
    \begin{equation}\label{dim2 Eq faX-6}
        \sum\limits_{i=0}^np^ia_{i}^{p^{n-i}}+\sum\limits_{i=0}^np^ib_{i}^{p^{n-i}}=\sum\limits_{i=0}^np^iS_{i}^{p^{n-i}}.
    \end{equation}
    By [FF18,~\S1.10.2], the additional law in the ring of Witt bi-vectors $BW$ is given by
    \begin{equation*}
        \sum\limits_{i\in\ZZ}[a_i^{p^{-i}}]p^i+\sum\limits_{i\in\ZZ}[b_i^{p^{-i}}]p^i=\sum\limits_{i\in\ZZ}[c_i^{p^{-i}}]p^i,
    \end{equation*}
    where $c_i\eqdef\lim\nolimits_{n\to\infty}c_{i,n}\in\ZZ\ddbra{\ \ldots,a_i,\ \ldots,b_i}$ with
    \begin{equation*}
        c_{i,n}\eqdef S_n(a_{i-n},a_{i-n+1},\ldots,a_i,b_{i-n},b_{i-n+1},\ldots,b_i)\in\ZZ[a_{i-n},\ldots,a_i,b_{i-n},\ldots,b_i].
    \end{equation*}
    In particular, for $i\in\ZZ$ we have 
    \begin{equation}\label{dim2 Eq faX-5}
    \begin{aligned}
        c_{i,0}&=a_i+b_i\in\ZZ[a_i,b_i];\\
        c_{i,1}&=a_i+b_i-\sum\limits_{s=1}^{p-1}\frac{\binom{p}{s}}{p}a_{i-1}^{p-s}b_{i-1}^s\in\ZZ[a_{i-1},a_i,b_{i-1},b_i].
    \end{aligned}
    \end{equation}
    Moreover, for $i\in\ZZ$ and $n\geq0$, we have in $\ZZ[a_{i-n-1},\ldots,a_i,b_{i-n-1},\ldots,b_i]$
    \begin{align}
        \sum\limits_{\ell=0}^np^{\ell}a_{i-n+\ell}^{p^{n-\ell}}+\sum\limits_{\ell=0}^np^{\ell}b_{i-n+\ell}^{p^{n-\ell}}&=\sum\limits_{\ell=0}^np^{\ell}c_{i-n+\ell,\ell}^{p^{n-\ell}};\label{dim2 Eq faX-1}\\
        \sum\limits_{\ell=0}^{n+1}p^{\ell}a_{i-(n+1)+\ell}^{p^{n+1-\ell}}+\sum\limits_{\ell=0}^{n+1}p^{\ell}b_{i-(n+1)+\ell}^{p^{n+1-\ell}}&=\sum\limits_{\ell=0}^{n+1}p^{\ell}c_{i-(n+1)+\ell,\ell}^{p^{n+1-\ell}}.\label{dim2 Eq faX-2}
    \end{align}
    Considering $(\ref{dim2 Eq faX-2})-p\cdot(\ref{dim2 Eq faX-1})$ and using $c_{i-(n+1)}=a_{i-(n+1)}+b_{i-(n+1)}$, we get
    \begin{equation*}
        a_{i-(n+1)}^{p^{n+1}}+b_{i-(n+1)}^{p^{n+1}}=\bigbra{a_{i-(n+1)}+b_{i-(n+1)}}^{p^{n+1}}+\sum\limits_{\ell=1}^{n+1}p^{\ell}\bbra{c_{i-(n+1)+\ell,\ell}^{p^{n+1-\ell}}-c_{i-(n+1)+\ell,\ell-1}^{p^{n+1-\ell}}}.
    \end{equation*}
    Hence we have
    \begin{equation}\label{dim2 Eq faX-3}
    \begin{aligned}
        c_{i,n+1}-c_{i,n}=\frac{1}{p^{n+1}}\Bigg[
        &a_{i-(n+1)}^{p^{n+1}}+b_{i-(n+1)}^{p^{n+1}}-\bigbra{a_{i-(n+1)}+b_{i-(n+1)}}^{p^{n+1}}\\
        &\hspace{4cm}-\sum\limits_{\ell=1}^{n}p^{\ell}\bbra{c_{i-(n+1)+\ell,\ell}^{p^{n+1-\ell}}-c_{i-(n+1)+\ell,\ell-1}^{p^{n+1-\ell}}}\Bigg].
    \end{aligned}
    \end{equation}
    From (\ref{dim2 Eq faX-3}) and using induction on $n$, we deduce that for $i\in\ZZ$ and $n\geq1$, 
    \begin{enumerate}
        \item[$\bullet$]
        each term of $c_{i,n+1}-c_{i,n}$ involves both the variable $a_k$ for some $k\leq i$ and the variable $b_{\ell}$ for some $\ell\leq i$;
        \item[$\bullet$]
        the minimal degree (in the variables $a_{k},b_{k}$ for $k\leq i$) of each term of $c_{i,n+1}-c_{i,n}$ is at least $2p-1$, and tends to infinity as $n$ tends to $\infty$.
    \end{enumerate}
    In particular, using (\ref{dim2 Eq faX-5}) we have
    \begin{equation}\label{dim2 Eq faX-4}
        c_i=a_i+b_i-\sum\limits_{s=1}^{p-1}\frac{\binom{p}{s}}{p}a_{i-1}^{p-s}b_{i-1}^s+\bbra{\deg\geq2p-1},
    \end{equation}
    where each term of $\bbra{\deg\geq2p-1}$ involves both the variable $a_k$ for some $k\leq i$ and the variable $b_{\ell}$ for some $\ell\leq i$, and has degree at least $2p-1$. Then combining (\ref{dim2 Eq faX-0}), (\ref{dim2 Eq faX-6}) and (\ref{dim2 Eq faX-4}), we conclude that (for $a=1+p[\mu]$ and $0\leq j\leq f-1$)
    \begin{equation*}
    \begin{aligned}
        a(X_j)&\in X_j+\mu^{p^j}X_{j-1}^p-\sum\limits_{s=1}^{p-1}\frac{\binom{p}{s}}{p}X_{j-1}^{p-s}\bbra{\mu^{p^{j-1}}X_{j-2}^p}^s+\bbra{\deg\geq3p-2}\\
        &=X_j+\mu^{p^j}X_{j-1}^p-\mu^{p^{j-1}}X_{j-1}^{p-1}X_{j-2}^p+\bbra{\deg\geq3p-2}\\
        &\subseteq X_j\bbra{1+\mu^{p^j}X_j^{\varphi-1}-\mu^{p^{j-1}}X_j^{\varphi-1}X_{j-1}^{\varphi-1}+F_{3-3p}A},
    \end{aligned}
    \end{equation*}
    which proves the first formula in \eqref{dim2 Eq fa A statement}.

    \hspace{\fill}
    
    Next we turn to the variables $Y_j$, still with $a=1+p[\mu]$ for some $\mu\in[\Fq\x]$. 

    \hspace{\fill}

    \noindent\textbf{Claim.} We have $\delta_1\in1-Y_0-\cdots-Y_{f-1}+\fm_{\OK}^2$ in $\FF\ddbra{\OK}=\FF\ddbra{Y_0,\ldots,Y_{f-1}}$.
    
    \proof Recall that $Y_j\eqdef\sum\nolimits_{\lambda\in\Fq\x}\lambda^{-p^j}\delta_{[\lambda]}\in\FF\ddbra{\OK}$ for $0\leq j\leq f-1$. On one hand, we have
    \begin{equation}\label{dim2 Eq faY-1}
        \sum\limits_{i=0}^{q-2}\sum\limits_{\lambda\in\Fq}\lambda^i\delta_{[\lambda]}=\sum\limits_{\lambda\in\Fq}\bbbra{\sum\limits_{i=0}^{q-2}\lambda^i}\delta_{[\lambda]}=1-\delta_1\in\fm_{\OK},
    \end{equation}
    where we use the convention that $0^0\eqdef1$. On the other hand, for each $0\leq i\leq q-2$, if we write $i=\sum\nolimits_{j=0}^{f-1}i_jp^j$ with $0\leq i_j\leq p-1$, then by \cite[Lemma~3.2.2.5(i)]{BHHMS2} we have in $\FF\ddbra{\OK}$
    \begin{equation}\label{dim2 Eq faY-2}
        \sum\limits_{\lambda\in\Fq}\lambda^i\delta_{[\lambda]}\equiv(-1)^{f-1}\bbbra{\scalebox{1}{$\prod\limits_{j=0}^{f-1}$}i_j!}\scalebox{1}{$\prod\limits_{j=0}^{f-1}$}Y_j^{p-1-i_j}\quad\mod\fm_{\OK}^p.
    \end{equation}
    Combining (\ref{dim2 Eq faY-1}) and (\ref{dim2 Eq faY-2}), we deduce that $\delta_1\in1-a_0Y_0-\cdots-a_{f-1}Y_{f-1}+\fm_{\OK}^2$ in $\FF\ddbra{\OK}$ with $a_j=(-1)^{f-1}(p-2)!\bbra{(p-1)!}^{f-1}=1$ in $\FF$ for all $0\leq j\leq f-1$.\qed

   \hspace{\fill}

   For each $0\leq j\leq f-1$, by the claim above we have (for $a=1+p[\mu]$)
   \begin{align*}
       a(Y_j)&=\sum\limits_{\lambda\in\Fq\x}\lambda^{-p^j}\delta_{(1+p[\mu])[\lambda]}=\sum\limits_{\lambda\in\Fq\x}\lambda^{-p^j}\delta_{[\lambda]}\cdot[\mu\lambda](\delta_{p})\\
       &\in\sum\limits_{\lambda\in\Fq\x}\lambda^{-p^j}\delta_{[\lambda]}\cdot[\mu\lambda]\bbra{(1-Y_0-\cdots-Y_{f-1}+\fm_{\OK}^2)^p}\\
       &=\sum\limits_{\lambda\in\Fq\x}\lambda^{-p^j}\delta_{[\lambda]}\cdot[\mu\lambda]\bbra{1-Y_0^p-\cdots-Y_{f-1}^p}+\fm_{\OK}^{2p}\\
       &=\sum\limits_{\lambda\in\Fq\x}\lambda^{-p^j}\delta_{[\lambda]}-\sum\limits_{i=0}^{f-1}\sum\limits_{\lambda\in\Fq\x}\lambda^{-p^j}\delta_{[\lambda]}(\mu\lambda)^{p^{i+1}}Y_i^p+\fm_{\OK}^{2p}\quad\text{(by (\ref{dim2 Eq phiOK Y}))}\\
       &=Y_j-\sum\limits_{i=0}^{f-1}\mu^{p^{i+1}}\bbbra{\sum\limits_{\lambda\in\Fq\x}\lambda^{p^{i+1}-p^j}\delta_{[\lambda]}}Y_i^p+\fm_{\OK}^{2p}.
   \end{align*}
   If $f=1$, then by (\ref{dim2 Eq faY-2}) we have
   \begin{equation*}
       \sum\limits_{\lambda\in\Fp\x}\lambda^{p-1}\delta_{[\lambda]}=\sum\limits_{\lambda\in\Fp\x}\delta_{[\lambda]}\equiv Y_0^{p-1}-1\quad\mod\fm_{\OK}^p.
   \end{equation*}
   If $f>1$, then by (\ref{dim2 Eq faY-2}), we deduce that
   \begin{equation*}
       \sum\limits_{\lambda\in\Fq\x}\lambda^{p^{i+1}-p^j}\delta_{[\lambda]}\in
       \begin{cases}
           -1+\fm_{\OK}^p&\text{if}\ i\equiv j-1\ \mod f\\
           Y_{j-1}^{p-1}+\fm_{\OK}^p&\text{if}\ i\equiv j-2\ \mod f\\
           \fm_{\OK}^p&\text{otherwise}.
       \end{cases}
   \end{equation*}
   In both cases, we conclude that
   \begin{equation}\label{dim2 Eq faY-3}
       a(Y_j)\in Y_j+\mu^{p^j}Y_{j-1}^p-\mu^{p^{j-1}}Y_{j-1}^{p-1}Y_{j-2}^p+\fm_{\OK}^{2p}.
   \end{equation}
    Using (\ref{dim2 Eq phiOK Y}) and the commutativity of the actions of $a$ and $[\Fq\x]$ on $A$, we deduce that each term in $\fm_{\OK}^{2p}$ of (\ref{dim2 Eq faY-3}) has degree congruent to $1$ modulo $p-1$, hence we have (for $a=1+p[\mu]$)
    \begin{equation*}
        a(Y_j)\in Y_j\bbra{1+\mu^{p^j}Y_j^{\varphi-1}-\mu^{p^{j-1}}Y_j^{\varphi-1}Y_{j-1}^{\varphi-1}+F_{3-3p}A},
    \end{equation*}
    which proves the second formula in (\ref{dim2 Eq fa A statement}).
\end{proof}

\begin{lemma}\label{dim2 Lem congruence A}
    Let $0\leq h\leq q-2$.
    \begin{enumerate}
        \item 
        For $i\geq-1$ and $a\in\OK\x$, we have
        \begin{equation*}
            \bbra{\id-f_{a,0}^{h(1-\varphi)/(1-q)}a}\!\bbra{X_0^{[h]_i(1-\varphi)}}\in F_{1-p}A.
        \end{equation*}
        \item
        For $i\geq-1$ and $a\in\OK\x$, we have
        \begin{equation*}
            \bbra{\id-f_{a,0}^{h(1-\varphi)/(1-q)}a}\!\bbra{X_0^{([h]_i-p^{i+1})(1-\varphi)}X_1^{p^{i+1}(1-\varphi)}}\in F_{1-p}A.
        \end{equation*}
        \item
        For $i\geq-1$ and $a\in\OK\x$, we have
        \begin{equation*}
            \bbra{\id-f_{a,0}^{h(1-\varphi)/(1-q)}a}\!\bbra{X_0^{([h]_i+p^{i+1})(1-\varphi)}}\in(h_{i+1}-1)c_a^{p^{i+1}}X_0^{[h]_i(1-\varphi)}+F_{1-p}A.
        \end{equation*}
        \item
        For $i\geq0$ such that $h_i=1$ and $a\in\OK\x$, we have
        \begin{equation*}
            \bbra{\id-f_{a,0}^{h(1-\varphi)/(1-q)}a}\!\bbra{X_0^{[h]_i(1-\varphi)}X_1^{p^i(1-\varphi)}}\in-c_a^{p^{i+1}}X_0^{[h]_i(1-\varphi)}+c_a^{p^i}X_0^{[h]_{i-1}(1-\varphi)}+F_{1-p}A.
        \end{equation*}
    \end{enumerate}
\end{lemma}

\begin{proof}
    We prove (iv), the others being similar and simpler. By definition we have
    \begin{align*}
        &\bbra{\id-f_{a,0}^{h(1-\varphi)/(1-q)}a}\!\bbra{X_0^{[h]_i(1-\varphi)}X_1^{p^i(1-\varphi)}}\\
        &\hspace{1.5cm}=X_0^{[h]_i(1-\varphi)}X_1^{p^i(1-\varphi)}\bbra{1-f_{a,0}^{\bra{h+(q-1)[h]_i}(1-\varphi)/(1-q)}f_{a,1}^{-p^i(1-\varphi)}}\\
        &\hspace{1.5cm}\in X_0^{[h]_i(1-\varphi)}X_1^{p^i(1-\varphi)}\bbra{1-f_{a,0}^{p^{i+1}\Zp(1-\varphi)}f_{a,1}^{-p^i(1-\varphi)}}\\
        &\hspace{1.5cm}\subseteq X_0^{[h]_i(1-\varphi)}X_1^{p^i(1-\varphi)}\bbra{-c_a^{p^{i+1}}X_1^{p^i(\varphi-1)}+c_a^{p^i}X_1^{p^i(\varphi-1)}X_0^{p^i(\varphi-1)}+F_{3(1-p)p^i}A}\\
        &\hspace{1.5cm}\subseteq-c_a^{p^{i+1}}X_0^{[h]_i(1-\varphi)}+c_a^{p^i}X_0^{[h]_{i-1}(1-\varphi)}+F_{1-p}A,
    \end{align*}
    where the second inclusion follows from Lemma \ref{dim2 Lem fa A} and uses $p\geq3$ (hence $p^{i+1}(p-1)\geq3p^i(p-1)$), and the last inclusion uses $h_i=1$.
\end{proof}

\begin{definition}\label{dim2 Def DX}
    Let $0\leq h\leq q-2$, $\lambda_0,\lambda_1\in\FF\x$ and $0\leq j\leq f-1$. We define $D_j^X,D_j^{\prime X},D_{\tr}^X,D_{\unr}^X\in A$ as follows:
    \begin{enumerate}
    \item 
    If $h_j\neq0$, we define 
    \begin{equation*}
        D_j^X\eqdef X_0^{[h]_{j-1}(1-\varphi)};
    \end{equation*}
    If $h_j=0$, we let $0\leq r\leq f-1$ such that $h_{j+1}=\cdots=h_{j+r}=1$ and $h_{j+r+1}\neq1$, then we define
    \begin{equation*}
    \begin{aligned}
        D_j^X&\eqdef X_0^{\bra{[h]_{j+r}+p^{j+r+1}}(1-\varphi)}\!+\!(h_{j+r+1}\!-\!1)\sum\limits_{i=0}^rX_0^{[h]_{j+i}(1-\varphi)}X_1^{p^{j+i}(1-\varphi)}\\
        &=
        \begin{aligned}[t]
            &X_0^{\bra{[h]_{j-1}+p^j(p+p^2+\cdots+p^{r+1})}(1-\varphi)}\!+\!(h_{j+r+1}\!-\!1)\sum\limits_{i=0}^rX_0^{\bra{[h]_{j-1}+p^j(p+p^2+\cdots+p^i)}(1-\varphi)}X_1^{p^{j+i}(1-\varphi)}.
        \end{aligned}
    \end{aligned}
    \end{equation*}
    \item
    We define
    \begin{equation*}
        D_j^{\prime X}\eqdef X_0^{\bra{[h]_{j-1}-p^{j}}(1-\varphi)}X_1^{p^{j}(1-\varphi)}.
    \end{equation*}
    \item
    If $h=1+p+\cdots+p^{f-1}$ and $\lambda_0\lambda_1\inv=1$, we define
    \begin{equation*}
        D_{\tr}^X\eqdef\sum\limits_{i=0}^{f-1}X_0^{[h]_i(1-\varphi)}X_1^{p^i(1-\varphi)}=\sum\limits_{i=0}^{f-1}X_0^{(1+p+\cdots+p^i)(1-\varphi)}X_1^{p^i(1-\varphi)}.
    \end{equation*}
    Otherwise (i.e.\,either $h\neq1+p+\cdots+p^{f-1}$ or $\lambda_0\lambda_1\inv\neq1$), we define $D_{\tr}^X\eqdef0$.
    \item
    If $h=0$ and $\lambda_0\lambda_1\inv=1$, we define $D_{\unr}^X\eqdef1$. Otherwise, we define $D_{\unr}^X\eqdef0$.
    \end{enumerate}
\end{definition}

\begin{corollary}\label{dim2 Cor congruence A}
    Let $0\leq h\leq q-2$ and $\lambda_0,\lambda_1\in\FF\x$.
    \begin{enumerate}
    \item 
    For all $0\leq j\leq f-1$ and $a\in\OK\x$, we have
    \begin{equation*}
    \begin{aligned}
        \bbra{\id-f_{a,0}^{h(1-\varphi)/(1-q)}a}\bra{D_j^X}&\in F_{1-p}A;\\
        \bbra{\id-f_{a,0}^{h(1-\varphi)/(1-q)}a}\bra{D_j^{\prime X}}&\in F_{1-p}A.
    \end{aligned}
    \end{equation*}
    \item 
    If $h=1+p+\cdots+p^{f-1}$ and $\lambda_0\lambda_1\inv=1$, then for all $a\in\OK\x$, we have
    \begin{equation*}
        \bbra{\id-f_{a,0}^{h(1-\varphi)/(1-q)}a}\bra{D_{\tr}^X}\in c_a\bbra{1-X_0^{h(1-\varphi)}}+F_{1-p}A.
    \end{equation*}
    \end{enumerate}
\end{corollary}

\begin{proof}
    This follows from Lemma \ref{dim2 Lem congruence A}. Note that for $i$ such that $h_i=0$ we have $[h]_{i}=[h]_{i-1}$.
\end{proof}

\begin{lemma}\label{dim2 Lem unique A}
    Let $0\leq h\leq q-2$ and $\lambda_0,\lambda_1\in\FF\x$.
    \begin{enumerate}
        \item
        For any $y\in F_{1-p}A$, the equation $\bigbra{\id-\lambda_0\lambda_1\inv X_0^{h(1-\varphi)}\varphi_q}(x)=y$ has a unique solution in $F_{1-p}A$, given by the convergent series $x=\sum\nolimits_{n=0}^{\infty}\bigbra{\lambda_0\lambda_1\inv X_0^{h(1-\varphi)}\varphi_q}^n(y)$.
        \item
        For any $y\in A$, the equation $\bigbra{\id-\lambda_0\lambda_1\inv X_0^{h(1-\varphi)}\varphi_q}(x)=y$ has at most one solution in $A$ unless $h=0$ and $\lambda_0\lambda_1\inv=1$.
    \end{enumerate}
\end{lemma}

\begin{proof}
    The proof is similar to that of Lemma \ref{dim2 Lem unique LT}. We omit the details.
\end{proof}

\begin{proposition}\label{dim2 Prop pair A}
    Let $0\leq h\leq q-2$ and $\lambda_0,\lambda_1\in\FF\x$.
    \begin{enumerate}
        \item 
        For all $0\leq j\leq f-1$, the tuple $\bigbra{D,(E_a)_{a\in\OK\x}}$ with
        \begin{equation*}
        \begin{cases}
            D&=D_j^X\\
            E_a&=E_{j,a}^X\eqdef\bbra{\id-\lambda_0\lambda_1\inv X_0^{h(1-\varphi)}\varphi_q}\inv\bbbra{\bbra{\id-f_{a,0}^{h(1-\varphi)/(1-q)}a}(D_j^X)}\\
            &=\sum\limits_{n=0}^{\infty}\bbra{\lambda_0\lambda_1\inv X_0^{h(1-\varphi)}\varphi_q}^n\bbbra{\bbra{\id-f_{a,0}^{h(1-\varphi)/(1-q)}a}(D_j^X)}
        \end{cases}
        \end{equation*}
        defines an element of $W^X$. We denote it by $[B_j^X]$. We define the element $[B_j^{\prime X}]\in W^X$ in a similar way, replacing $D^X_j$ with $D^{\prime X}_j$.
        \item
        If $h=1+p+\cdots+p^{f-1}$ and $\lambda_0\lambda_1\inv=1$, then the tuple $\bigbra{D,(E_a)_{a\in\OK\x}}$ with
        \begin{equation*}
        \begin{cases}
            D&=D_{\tr}^X\\
            E_a&=E_{\tr,a}^X\eqdef\bbra{\id- X_0^{h(1-\varphi)}\varphi_q}\inv\bbbra{\bbra{\id-f_{a,0}^{h(1-\varphi)/(1-q)}a}(D_{\tr}^X)}\\
            &=c_a+\sum\limits_{n=0}^{\infty}\bbra{ X_0^{h(1-\varphi)}\varphi_q}^n\left[\bbra{\id-f_{a,0}^{h(1-\varphi)/(1-q)}a}(D_{\tr}^X)-c_a\bbra{1-X_0^{h(1-\varphi)}}\right]
        \end{cases}
        \end{equation*}
        defines an element of $W^X$. We denote it by $[B_{\tr}^X]$. Otherwise, we define $E_{\tr,a}^{X}\eqdef0$ for all $a\in\OK\x$ and $[B_{\tr}^X]\eqdef[0]$ in $W^X$.
        \item
        If $h=0$ and $\lambda_0\lambda_1\inv=1$, then the tuple $\bigbra{D,(E_a)_{a\in\OK\x}}$ with
        \begin{equation*}
        \begin{cases}
            D&=D_{\unr}^X=1\\
            E_a&=E_{\unr,a}^X\eqdef0
        \end{cases}    
        \end{equation*}
        defines an element of $W^X$. We denote it by $[B_{\unr}^X]$. Otherwise, we define $E_{\unr,a}^{X}\eqdef0$ for all $a\in\OK\x$ and $[B_{\unr}^X]\eqdef[0]$ in $W^X$.
    \end{enumerate}
\end{proposition}

\begin{proof}
    (iii) is direct. For (i) and (ii), each $E_a$ is well-defined by Corollary \ref{dim2 Cor congruence A} and Lemma \ref{dim2 Lem unique A}(i), and condition (ii) in Definition \ref{dim2 Def WX} is guaranteed by the uniqueness of solution in Lemma \ref{dim2 Lem unique A}(i),(ii).
\end{proof}


By Lemma \ref{dim2 Lem fa A}, we can give similar definitions for the variables $Y_i$ instead of $X_i$. We have the following partial comparison result.

\begin{proposition}\label{dim2 Prop XY}
    Suppose that $c_0,\ldots,c_{f-1},c_0',\ldots,c_{f-1}',c_{\unr}\in\FF$ such that $c_j=0$ if $h_j=0$, then we have an isomorphism of \'etale $(\varphi_q,\OK\x)$-modules over $A$:
    \begin{equation*}
        D\bbra{\scalebox{1}{$\sum\limits_{j=0}^{f-1}$}c_j[B_j^X]+\scalebox{1}{$\sum\limits_{j=0}^{f-1}$}c_j'[B_{j}^{\prime X}]+c_{\unr}[B_{\unr}^X]}\cong D\bbra{\scalebox{1}{$\sum\limits_{j=0}^{f-1}$}c_j[B_j^Y]+\scalebox{1}{$\sum\limits_{j=0}^{f-1}$}c_j'[B_{j}^{\prime Y}]+c_{\unr}[B_{\unr}^Y]}.
    \end{equation*}
\end{proposition}

\begin{proof}
    Let $e_0^X,e_1^X$ be an $A$-basis of $D\bigbra{\sum\nolimits_{j=0}^{f-1}c_j[B_j^X]+\sum\nolimits_{j=0}^{f-1}c_j'[B_{j}^{\prime X}]+c_{\unr}[B_{\unr}^X]}$ with respect to which the matrices of the actions of $\varphi_q$ and $\OK\x$ have the form
    \begin{equation*}
    \begin{cases}
        \Mat_A^X(\varphi_q)&=\pmat{\lambda_0 X_0^{h(1-\varphi)}&\lambda_1D^X\\0&\lambda_1}\\
        \Mat_A^X(a)&=\pmat{f_{a,0}^{h(1-\varphi)/(1-q)}&E_a^X\\0&1}\quad\forall\,a\in\OK\x,
    \end{cases}
    \end{equation*}
    where 
    \begin{equation*}
    \begin{cases}
        D^X&\eqdef\sum\limits_{j=0}^{f-1}c_jD_j^X+\sum\limits_{j=0}^{f-1}c_j'D_{j}^{\prime X}+c_{\unr}D_{\unr}^X\\
        E_a^X&\eqdef\sum\limits_{j=0}^{f-1}c_jE_{j,a}^X+\sum\limits_{j=0}^{f-1}c_j'E_{j,a}^{\prime X}+c_{\unr}E_{\unr,a}^X\quad\forall\,a\in\OK\x.
    \end{cases}   
    \end{equation*}
    We have similar definitions replacing $X$ with $Y$. To prove the proposition, it is enough to find a change of basis formula $(e_0^Y\ e_1^Y)=(e_0^X\ e_1^X)Q$ for some $Q=\smat{b_{00}&b_{01}\\0&b_{11}}\in I_2+M_2(F_{1-p}A)$ such that $Q\inv\Mat_A^X(\varphi_q)\varphi_q(Q)=\Mat_A^Y(\varphi_q)$, or equivalently
    \begin{equation}\label{dim2 Eq XY Q1}
        \pmat{b_{00}\inv&-b_{00}\inv b_{01}b_{11}\inv\\0&b_{11}\inv}\!\pmat{\lambda_0 X_0^{h(1-\varphi)}&\lambda_1D^X\\0&\lambda_1}\!\pmat{\varphi_q(b_{00})&\varphi_q(b_{01})\\0&\varphi_q(b_{11})}=\pmat{\lambda_0Y_0^{h(1-\varphi)}&\lambda_1D^Y\\0&\lambda_1}.
    \end{equation}
    Then the $\OK\x$-actions also agree by Lemma \ref{dim2 Lem unique A}(i) using $E_a^X,E_a^Y\in F_{1-p}A$.

    \hspace{\fill}    

    Comparing the (2,2)-entries of (\ref{dim2 Eq XY Q1}), we have $b_{11}=1$.
    
    Comparing the (1,1)-entries of (\ref{dim2 Eq XY Q1}), we need to solve $\varphi_q(b_{00})b_{00}\inv=\bigbra{Y_0^{1-\varphi}/X_0^{1-\varphi}}^h$. So we can take $b_{00}=\bigbra{Y_0^{1-\varphi}/X_0^{1-\varphi}}^{h/(q-1)}$, which makes sense since $Y_0^{1-\varphi}/X_0^{1-\varphi}\in1+F_{1-p}A$ by (\ref{dim2 Eq XY}).

    Comparing the (1,2)-entries of (\ref{dim2 Eq XY Q1}), we need to solve
    \begin{equation*}
        b_{00}\inv\lambda_0 X_0^{h(1-\varphi)}\varphi_q(b_{01})+b_{00}\inv\lambda_1 D^X\varphi_q(b_{11})-b_{00}\inv b_{01}b_{11}\inv\lambda_1\varphi_q(b_{11})=\lambda_1 D^Y.
    \end{equation*}
    Replacing $b_{00}$, $b_{11}$ by their previous values, we get
    \begin{equation}\label{dim2 Eq XY Q2}
        \bbra{\id-\lambda_0\lambda_1\inv X_0^{h(1-\varphi)}\varphi_q}(b_{01})=D^X-D^Y\bbra{Y_0^{1-\varphi}/X_0^{1-\varphi}}^{h/(q-1)}.
    \end{equation}
    Then we deduce from  Lemma \ref{dim2 Lem unique A}(i) and the claim below that there is a unique solution of $b_{01}\in F_{1-p}A$, which completes the proof.

    \hspace{\fill}

    \noindent\textbf{Claim.} Then RHS of \eqref{dim2 Eq XY Q2} is in $F_{1-p}A$.

    \proof We assume that $D^X=D_j^X$ for $0\leq j\leq f-1$ such that $h_j\neq0$, the cases $D^X=D_j^{\prime X}$ and $D^X=D_{\unr}$ being similar. Then we have
    \begin{align*}
        &D_j^X-D_j^Y\bbra{Y_0^{1-\varphi}/X_0^{1-\varphi}}^{h/(q-1)}\\
        &\hspace{1.5cm}=X_0^{[h]_{j-1}(1-\varphi)}-Y_0^{[h]_{j-1}(1-\varphi)}\bbra{Y_0^{1-\varphi}/X_0^{1-\varphi}}^{h/(q-1)}\\
        &\hspace{1.5cm}=X_0^{[h]_{j-1}(1-\varphi)}\bbbra{1-\bbra{Y_0^{1-\varphi}/X_0^{1-\varphi}}^{[h]_{j-1}+h/(q-1)}}\\
        &\hspace{1.5cm}\in X_0^{[h]_{j-1}(1-\varphi)}\bbbra{1-\bbra{Y_0^{1-\varphi}/X_0^{1-\varphi}}^{p^j\Zp}}\\
        &\hspace{1.5cm}\subseteq X_0^{[h]_{j-1}(1-\varphi)}F_{(1-p)p^j}A\subseteq F_{1-p}A,
    \end{align*}
    which completes the proof.
\end{proof}
    
\begin{remark}
    In general, we do not know how to write $D\bigbra{[B_j^X]}$ (in the case $h_j=0$) and $D\bigbra{[B_{\tr}^X]}$ in terms of elements of $W^Y$.
\end{remark}

\section{The \'etale \texorpdfstring{$(\varphi,\OK\x)$}.-module \texorpdfstring{$D_A^{\otimes}(\rhobar)$}.}\label{dim2 Sec DA0}
 
In this section, we recall the definition of the functor $\rhobar\mapsto D_A^{\otimes}(\rhobar)$ defined in \cite{BHHMS3} and give an explicit computation of $D_A^{\otimes}(\rhobar)$ for all reducible two-dimensional $\rhobar$ when $p\geq5$.

Recall that $A_{\infty}$ is the completed perfection of $A$. The actions of $\varphi$ and $\OK\x$ on $A$ extends naturally to $A_{\infty}$, and $A_{\infty}\x$ becomes a $\Qp[\varphi]$-module.

\begin{proposition}[\cite{BHHMS3},~Cor.~2.6.6]\label{dim2 Prop equiv}
    The functor $D\mapsto A_{\infty}\otimes_AD$ induces an equivalence of categories between the category of \'etale $(\varphi_q,\cO_K\x)$-modules over $A$ and the category of \'etale $(\varphi_q,\cO_K\x)$-modules over $A_{\infty}$, which is rank-preserving and compatible with tensor products.
\end{proposition}

As in \cite{BHHMS3}, we let
\begin{equation*}
    A_{\infty}'\eqdef\FF\dbra{T_{K,0}^{1/p^{\infty}}}\bang{\bbra{\frac{T_{K,i}}{T_{K,0}^{p^i}}}^{\pm1/p^{\infty}},1\leq i\leq f-1}.
\end{equation*}
There is an $\FF$-linear Frobenius $\varphi$ on $A_{\infty}'$ given by (for each $0\leq i\leq f-1$)
\begin{equation}\label{dim2 Eq phi TKi}
    \varphi(T_{K,i})=T_{K,i+1},
\end{equation}
where we use the convention that $T_{K,f}\eqdef T_{K,0}^q$. There is also an $(\cO_K\x)^f$-action on $A_{\infty}'$ commuting with $\varphi_q(\eqdef\varphi^f)$ given by ($a_i\in\cO_K\x$)
\begin{equation*}
    (a_0,\ldots,a_{f-1})(T_{K,i})=a_i(T_{K,i}),
\end{equation*}
where $\cO_K\x$ acts on each variable $T_{K,i}$ in the same way as they act on $T_{K,\sigma_0}$ in \S\ref{dim2 Sec LT}.

For $0\leq i\leq f-1$ and $a\in\OK\x$, we define $j_i(a)\in(K\x)^f$ to be $a$ in the $i$-th coordinate and $1$ otherwise. There is an inclusion $\iota_i:\FF\dbra{T_{K,\sigma_0}}\into A_{\infty}'$ defined by $T_{K,\sigma_0}\mapsto T_{K,i}$, which commutes with $\varphi_q$, and the action of $a\in\cO_K\x$ on $\FF\dbra{T_{K,\sigma_0}}$ is identified with the action of $j_i(a)$ on $A_{\infty}'$. In particular, we regard $\FF\dbra{T_{K,\sigma_0}}$ as a subfield of $A_{\infty}'$ via the inclusion $\iota_0$. By \cite[Prop.~2.4.4]{BHHMS3}, we can also regard $A_{\infty}$ as a subring of $A_{\infty}'$, which is compatible with $\varphi$, and the action of $a\in\cO_K\x$ on $A_{\infty}$ is identified with the action of $(a,1,\ldots,1)$ on $A_{\infty}'$. Moreover, if we denote $\Delta_1\eqdef\Ker\bbra{(\cO_K\x)^f\to\cO_K\x}$ the kernel of the multiplication map, then we have $A_{\infty}=\bbra{A_{\infty}'}^{\Delta_1}$ (see the paragraph before \cite[Thm.~2.5.1]{BHHMS3}).

For $\rhobar$ a finite-dimensional continuous representation of $\Gal(\ovl{K}/K)$ over $\FF$ and $0\leq i\leq f-1$, we define
\begin{equation*}
    D_{A_{\infty}}^{(i)}(\rhobar)\eqdef\bbra{A_{\infty}'\otimes_{\iota_i,\FF\dbra{T_{K,\sigma_0}}}D_{K,\sigma_0}(\rhobar)}^{\Delta_1}.
\end{equation*}
We endow it with a $\varphi_q$-action given by $\varphi_q=\varphi_q\otimes\varphi_q$, and an $\cO_K\x$-action such that $a\in\cO_K\x$ acts by $j_i(a)\otimes a$. By the result of \cite{BHHMS3}, these actions are well-defined and make $D_{A_{\infty}}^{(i)}(\rhobar)$ an \'etale $(\varphi_q,\cO_K\x)$-module over $A_{\infty}$. Moreover, there is an isomorphism 
\begin{equation*}
    \phi_i:D_{A_{\infty}}^{(i)}(\rhobar)\ism D_{A_{\infty}}^{(i+1)}(\rhobar)
\end{equation*}
given by $\phi_i(x\otimes v)\eqdef\varphi(x)\otimes v$ if $i<f-1$, and $\phi_i(x\otimes v)\eqdef\varphi(x)\otimes\varphi_q(v)$ if $i=f-1$. Finally, we define the \'etale $(\varphi,\OK\x)$-module over $A_{\infty}$:
\begin{equation*}
    D_{A_{\infty}}^{\otimes}(\rhobar)\eqdef\bigotimes\limits_{i=0}^{f-1}D_{A_{\infty}}^{(i)}(\rhobar),
\end{equation*}
where the $\varphi$-action is given by $\varphi(v_0\otimes\cdots\otimes v_{f-1})\eqdef\phi_{f-1}(v_{f-1})\otimes\phi_0(v_0)\otimes\cdots\otimes\phi_{f-2}(v_{f-2})$, and the $\OK\x$-action is the diagonal action.

By the equivalence of categories in Proposition \ref{dim2 Prop equiv}, up to isomorphism there are unique \'etale $(\varphi_q,\cO_K\x)$-modules $D_A^{(i)}(\rhobar)$ for $0\leq i\leq f-1$ and $D_A^{\otimes}(\rhobar)$ over $A$ such that
\begin{equation*}
\begin{aligned}
    A_{\infty}\otimes_AD_A^{(i)}(\rhobar)&\cong D_{A_{\infty}}^{(i)}(\rhobar);\\
    A_{\infty}\otimes_AD_A^{\otimes}(\rhobar)&\cong D_{A_{\infty}}^{\otimes}(\rhobar).
\end{aligned}
\end{equation*}

\begin{lemma}\label{dim2 Lem u}
    There exists a unique element $u\in T_{K,0}(1+(A_{\infty}')\ccc)\subseteq A_{\infty}'$ such that:
    \begin{enumerate}
        \item $u^{q-1}=X_0^{\varphi-1}\in A\subseteq A_{\infty}\subseteq A_{\infty}'$;
        \item for any $(a_0,\ldots,a_{f-1})\in\Delta_1$, we have $(a_0,\ldots,a_{f-1})(u)=\ovl{a_0}u$, hence $$(a_0,\ldots,a_{f-1})\bigbra{uT_{K,0}\inv}=f_{a_0}^{\LT}uT_{K,0}\inv;$$
        \item for any $a\in\cO_K\x$, we have $(a,1,\ldots,1)(u)=\ovl{a}f_{a,0}^{(1-\varphi)/(q-1)}u$, hence $$(a,1,\ldots,1)\bigbra{uT_{K,0}\inv}=f_a^{\LT}f_{a,0}^{(1-\varphi)/(q-1)}uT_{K,0}\inv;$$ 
        \item $\varphi_q(u)=u^q$.
    \end{enumerate}
\end{lemma}

\begin{proof}
    (i),(ii),(iii) follow from [BHHMS3,~Lemma~2.9.2] and (iv) follows from [BHHMS3,~Remark~2.9.4]. 
\end{proof}

\begin{lemma}\label{dim2 Lem norm}
    There is a unique multiplicative norm $|\cdot|$ on $A_{\infty}'$ inducing the topology of $A_{\infty}'$ such that $|T_{K,0}|=p\inv$. It also satisfies:
    \begin{enumerate}
        \item
        $|T_{K,i}|=p^{-p^i}$ for all $0\leq i\leq f-1$;
        \item
        $|\varphi(x)|=|x|^p\ \forall\,x\in A_{\infty}'$;
        \item
        for any $(a_0,\ldots,a_{f-1})\in(\cO_K\x)^f$, we have $|(a_0,\ldots,a_{f-1})(x)|=|x|\ \forall\,x\in A_{\infty}'$;
        \item $|X_i|=|Y_i|=p^{-(1+p+\cdots+p^{f-1})}$ for all $0\leq i\leq f-1$. In particular, for any $x\in F_{1-p}A$, we have $|x|\leq p^{-(q-1)}$.
    \end{enumerate}
\end{lemma}

\begin{proof}
    Recall that the desired norm on $A_{\infty}'$ is the unique multiplicative extension to $A_{\infty}'$ of the Gauss norm on the ring $\FF\dbra{T_{K,0}}\bigang{T_{K,i}/T_{K,0}^{p^i},1\leq i\leq f-1}$ with $T_{K,0}$-adic topology such that $|T_{K,0}|=p\inv$ (see \cite[Lemma~2.4.7(iii)]{BHHMS3} and the proof of \cite[Lemma~2.4.2(iii)]{BHHMS3}). In particular, for $0\leq i\leq f-1$ we have $|T_{K,i}|=\abs{T_{K,i}/T_{K,0}^{p^i}}\cdot\abs{T_{K,0}}^{p^i}=p^{-p^i}$, which proves (i).
    
    The assignment $\norm{x}\eqdef|\varphi(x)|$ is a multiplicative norm on $A_{\infty}'$ inducing the topology of $A_{\infty}'$ such that $\norm{T_{K,0}}=p^{-p}$. By uniqueness we get $|\varphi(x)|=|x|^p\ \forall\,x\in A_{\infty}'$, which proves (ii). Similarly, for any $(a_0,\ldots,a_{f-1})\in(\cO_K\x)^f$, the assignment $\norm{x}'\eqdef|(a_0,\ldots,a_{f-1})(x)|$ is a multiplicative norm on $A_{\infty}'$ inducing the topology of $A_{\infty}'$ such that $\norm{T_{K,0}}=p^{-1}$. By uniqueness we get $|(a_0,\ldots,a_{f-1})(x)|=|x|\ \forall\,x\in A_{\infty}'$, which proves (iii).
    
    Then we prove (iv). Recall from \cite[(63)]{BHHMS3} that we have $X_0=T_{K,0}\cdots T_{K,f-1}(1+w_0)$ for some $|w_0|<1$. Then we deduce from (i) that $\babs{X_0}=\babs{T_{K,0}\cdots T_{K,f-1}}=p^{-(1+p+\cdots+p^{f-1})}.$ By the proof of \cite[Lemma~2.4.2(iii)]{BHHMS3}, we have $|X_i|=|X_0|=p^{-(1+p+\cdots+p^{f-1})}$ for $1\leq i\leq f-1$. Finally, we deduce from (\ref{dim2 Eq XY}) that $|Y_i|=|X_i|=p^{-(1+p+\cdots+p^{f-1})}$ for $0\leq i\leq f-1.$
\end{proof}

For $r\in\RR_{>0}$, we denote $B(r)\eqdef\bigset{x\in A_{\infty}':|x|\leq p^{-r}}$ and $B\cc(r)\eqdef\bigset{x\in A_{\infty}':|x|<p^{-r}}$.

\begin{lemma}\label{dim2 Lem uTinv simple}
    Let $u\in A_{\infty}'$ be as in Lemma \ref{dim2 Lem u}, then we have
        \begin{equation*}
            uT_{K,0}\inv\in1+B\!\bbra{\tfrac{(q-1)(p-1)}{p}}.
        \end{equation*}
\end{lemma}

\begin{proof}
    This essentially follows from the proof of \cite[Lemma~2.9.2]{BHHMS3} with $c=q-1$. See Lemma \ref{dim2 Lem uTinv}(ii) below for a more precise relation.
\end{proof}

\begin{lemma}\label{dim2 Lem op}
    We have the following equalities of operators on $A_{\infty}'$:
    \begin{enumerate}
        \item
        for $a\in\cO_{K}\x$ and $h\in\ZZ$, we have
        \begin{equation*}
            \bbra{T_{K,0}^{-(q-1)h}\varphi_q}\circ\bbra{\bbra{f_a^{\LT}}^h(a,1,\ldots,1)}=\bbra{\bbra{f_a^{\LT}}^h(a,1,\ldots,1)}\circ\bbra{T_{K,0}^{-(q-1)h}\varphi_q};
        \end{equation*}
        \item
        for $(a_0,\ldots,a_{f-1})\in\Delta_1$ and $h\in\ZZ$, we have
        \begin{equation*}
            \bbra{T_{K,0}^{-(q-1)h}\varphi_q}\circ\bbra{\bbra{f_{a_0}^{\LT}}^h(a_0,\ldots,a_{f-1})}=\bbra{\bbra{f_{a_0}^{\LT}}^h(a_0,\ldots,a_{f-1})}\circ\bbra{T_{K,0}^{-(q-1)h}\varphi_q};
        \end{equation*}
        \item
        for $h\in\ZZ$, we have
        \begin{equation*}
            \bbra{T_{K,0}^{-(q-1)h}\varphi_q}\circ\bbra{\bigbra{uT_{K,0}\inv}^{-h}}=\bigbra{uT_{K,0}\inv}^{-h} X_0^{h(1-\varphi)}\varphi_q.
        \end{equation*}
    \end{enumerate}
\end{lemma}

\begin{proof}
    All the equalities are direct calculations, (i) and (ii) using the definition of $f_a^{\LT}$, and (iii) using Lemma \ref{dim2 Lem u}(i),(iv). We omit the details. Here we recall that we identify $T_{K,\sigma_0}\in\FF\dbra{T_{K,\sigma_0}}$ with $T_{K,0}\in A_{\infty}'$ via the inclusion $\iota_0$.
\end{proof}

\begin{lemma}\label{dim2 Lem unique Ainfty'}
    Let $0\leq h\leq q-2$ and $\lambda_0,\lambda_1\in\FF\x$. Then for any $y\in A_{\infty}'$ with $|y|<p^{-h}$, the equation $\bigbra{\id-\lambda_0\lambda_1\inv T_{K,0}^{-(q-1)h}\varphi_q}(x)=y$ has a unique solution $x\in A_{\infty}'$ with $|x|<p^{-h}$, given by the convergent series $x=\sum\nolimits_{n=0}^{\infty}\bigbra{\lambda_0\lambda_1\inv T_{K,0}^{-(q-1)h}\varphi_q}^n(y)$.
\end{lemma}

\begin{proof}
    For any $x\in A_{\infty}'$, we have (by Lemma \ref{dim2 Lem norm}(i),(ii)) $\bigabs{\lambda_0\lambda_1\inv T_{K,0}^{-(q-1)h}\varphi_q(x)}=|x|^qp^{(q-1)h}.$ In particular, if $|x|<p^{-h}$ and $x\neq0$, then we have $\bigabs{\lambda_0\lambda_1\inv T_{K,0}^{-(q-1)h}\varphi_q(x)}<|x|.$ If $x_1,x_2\in A_{\infty}'$ such that $|x_1|,|x_2|<p^{-h}$ and $\bigbra{\id-\lambda_0\lambda_1\inv T_{K,0}^{-(q-1)h}\varphi_q}(x_1)=\bigbra{\id-\lambda_0\lambda_1\inv T_{K,0}^{-(q-1)h}\varphi_q}(x_2)$, then we have $|x_1-x_2|=\bigabs{\lambda_0\lambda_1\inv T_{K,0}^{-(q-1)h}\varphi_q(x_1-x_2)}$, which implies $x_1=x_2$. This proves uniqueness. Then given $|y|<p^{-h}$, one easily checks that the element $x\eqdef\sum\nolimits_{n=0}^{\infty}\bigbra{\lambda_0\lambda_1\inv T_{K,0}^{-(q-1)h}\varphi_q}^n(y)$ converges, and satisfies $\bigbra{\id-\lambda_0\lambda_1\inv T_{K,0}^{-(q-1)h}\varphi_q}(x)=y$ and $|x|=|y|<p^{-h}$.
\end{proof}

\begin{definition}\label{dim2 Def Hj}
    Let $0\leq h\leq q-2$ and $0\leq j\leq f-1$. We define $H_j\in\ZZ$ as follows:
    \begin{enumerate}
        \item 
        If $h_{j-1}\neq p-1$, we define $H_j\eqdef0$.
        \item 
        If $h_{j-1}=p-1$ and $h_j\neq 0$, we define $H_j\eqdef h_j$.
        \item If $h_{j-1}=p-1$ and $h_j=0$, we let $0\leq r\leq f-1$ such that $h_{j+1}=\cdots=h_{j+r}=1$ and $h_{j+r+1}\neq1$, then we define $H_j\eqdef h_{j+r+1}-1$.
    \end{enumerate}
\end{definition}

\begin{definition}
    Let $\rhobar$ be as in (\ref{dim2 Eq rhobar}). Suppose that (see Theorem \ref{dim2 Thm basis LT})
    \begin{equation*}
        D_{K,\sigma_0}(\rhobar)\cong D\bbra{\scalebox{1}{$\sum\limits_{j=0}^{f-1}$}c_j[B_j^{\LT}]+c_{\tr}[B_{\tr}^{\LT}]+c_{\unr}[B_{\unr}^{\LT}]}
    \end{equation*}
    for some $c_0,\ldots,c_{f-1},c_{\tr},c_{\unr}\in\FF$, then we define (see Proposition \ref{dim2 Prop pair A} for the notation)
    \begin{equation*}
        D_{A,\sigma_0}(\rhobar)\eqdef D\bbra{\scalebox{1}{$\sum\limits_{j=0}^{f-1}$}c_j\bbra{[B_j^X]+H_j[B_{j-1}^{\prime X}]}+c_{\tr}[B_{\tr}^X]+c_{\unr}[B_{\unr}^X]},
    \end{equation*}
    where we use the convention that $[B_{-1}^{\prime X}]\eqdef \lambda_0\lambda_1\inv[B_{f-1}^{\prime X}]$ in $W^X$. This is an \'etale $(\varphi_q,\OK\x)$-module of rank $2$ over $A$ and is well-defined up to isomorphism. 
\end{definition}

\begin{theorem}\label{dim2 Thm main}
    Suppose that $p\geq5$, then for $\rhobar$ as in (\ref{dim2 Eq rhobar}), we have an isomorphism of \'etale $(\varphi_q,\OK\x)$-modules over $A$:
    \begin{equation*}
        D_{A}^{(0)}(\rhobar)\cong D_{A,\sigma_0}(\rhobar).
    \end{equation*}
\end{theorem}

\begin{proof}
    By Proposition \ref{dim2 Prop equiv}, it suffices to show that 
    \begin{equation}\label{dim2 Eq Delta1}
        A_{\infty}\otimes_AD_{A,\sigma_0}(\rhobar)=\bbra{A_{\infty}'\otimes_{\FF\dbra{T_{K,\sigma_0}}}D_{K,\sigma_0}(\rhobar)}^{\Delta_1}.
    \end{equation}
    
    Let $e_0^{\LT},e_1^{\LT}$ be an $\FF\dbra{T_{K,\sigma_0}}$-basis of $D_{K,\sigma_0}(\rhobar)$ with respect to which the matrices of the actions of $\varphi_q$ and $\OK\x$ have the form
    \begin{equation*}
    \begin{cases}
        \Mat_K(\varphi_q)&=\pmat{\lambda_0T_{K,\sigma_0}^{-(q-1)h}&\lambda_1D^{\LT}\\0&\lambda_1}\\
        \Mat_K(a)&=\pmat{\bbra{f_a^{\LT}}^h&E_a^{\LT}\\0&1}\quad\forall\,a\in\OK\x,
    \end{cases}
    \end{equation*}
    where 
    \begin{equation*}
    \begin{cases}
        D^{\LT}&\eqdef\sum\limits_{j=0}^{f-1}c_jD_j^{\LT}+c_{\tr}D_{\tr}^{\LT}+c_{\unr}D_{\unr}^{\LT}\\
        E_a^{\LT}&\eqdef\sum\limits_{j=0}^{f-1}c_jE_{j,a}^{\LT}+c_{\tr}E_{\tr,a}^{\LT}+c_{\unr}E_{\unr,a}^{\LT}\quad\forall\,a\in\OK\x.
    \end{cases}   
    \end{equation*}
    Let $e_0^X,e_1^X$ be an $A$-basis of $D_{A,\sigma_0}(\rhobar)$ with respect to which the matrices of the actions of $\varphi_q$ and $\OK\x$ have the form
    \begin{equation*}
    \begin{cases}
        \Mat_A(\varphi_q)&=\pmat{\lambda_0 X_0^{h(1-\varphi)}&\lambda_1D^X\\0&\lambda_1}\\
        \Mat_A(a)&=\pmat{f_{a,0}^{h(1-\varphi)/(1-q)}&E_a^X\\0&1}\quad\forall\,a\in\OK\x
    \end{cases}
    \end{equation*}
    where 
    \begin{equation*}
    \begin{cases}
        D^X&\eqdef\sum\limits_{j=0}^{f-1}c_j\bbra{D_j^X+H_jD_{j-1}^{\prime X}}+c_{\tr}D_{\tr}^X+c_{\unr}D_{\unr}^X\\
        E_a^X&\eqdef\sum\limits_{j=0}^{f-1}c_j\bbra{E_{j,a}^X+H_jE_{j-1,a}^{\prime X}}+c_{\tr}E_{\tr,a}^X+c_{\unr}E_{\unr,a}^X\quad\forall\,a\in\OK\x.
    \end{cases}   
    \end{equation*}
    To prove (\ref{dim2 Eq Delta1}), it is enough to find a change of basis formula $(e_0^X\ e_1^X)=(e_0^{\LT}\ e_1^{\LT})Q$ for some $Q=\smat{b_{00}&b_{01}\\0&b_{11}}\in\GL_2\bbra{A_{\infty}'}$, such that
    \begin{enumerate}
        \item $Q\inv\Mat_K(\varphi_q)\varphi_q(Q)=\Mat_A(\varphi_q)$;
        \item
        $Q\inv\Mat_K(a)a(Q)=\Mat_A(a)\ \forall\,a\in\cO_K\x$;
        \item
        the basis $(e_0^X\ e_1^X)=(e_0^{\LT}\ e_1^{\LT})Q$ is fixed by $(a_0,\ldots,a_{f-1})\ \forall\,(a_0,\ldots,a_{f-1})\in\Delta_1$.
    \end{enumerate}
    More concretely, we are going to solve the equation 
    \begin{equation}\label{dim2 Eq Q2-1}
        \pmat{b_{00}\inv&-b_{00}\inv b_{01}b_{11}\inv\\0&b_{11}\inv}\!\pmat{\lambda_0T_{K,0}^{-(q-1)h}&\lambda_1D^{\LT}\\0&\lambda_1}\!\pmat{\varphi_q(b_{00})&\varphi_q(b_{01})\\0&\varphi_q(b_{11})}=\pmat{\lambda_0 X_0^{h(1-\varphi)}&\lambda_1D^X\\0&\lambda_1},
    \end{equation}
    and then check that the following equalities hold:
    \begin{align}
    &\begin{aligned}\label{dim2 Eq Q2-2}
        &\pmat{b_{00}\inv&-b_{00}\inv b_{01}b_{11}\inv\\0&b_{11}\inv}\!\pmat{\bbra{f_a^{\LT}}^h&E_a^{\LT}\\0&1}\!\pmat{(a,1,\ldots,1)(b_{00})&(a,1,\ldots,1)(b_{01})\\0&(a,1,\ldots,1)(b_{11})}\\
        &\hspace{1.5cm}=\pmat{f_{a,0}^{h(1-\varphi)/(1-q)}&E_a^X\\0&1}\ \forall\,a\in\cO_K\x;
    \end{aligned}\\
    &\begin{aligned}\label{dim2 Eq Q2-3}
        &\pmat{b_{00}\inv&-b_{00}\inv b_{01}b_{11}\inv\\0&b_{11}\inv}\!\pmat{\bbra{f_{a_0}^{\LT}}^h&E_{a_0}^{\LT}\\0&1}\!\pmat{(a_0,\ldots,a_{f-1})(b_{00})&(a_0,\ldots,a_{f-1})(b_{01})\\0&(a_0,\ldots,a_{f-1})(b_{11})}\\
        &\hspace{1.5cm}=\pmat{1&0\\0&1}\ \forall\,(a_0,\ldots,a_{f-1})\in\Delta_1.
    \end{aligned}
    \end{align}

    \hspace{\fill}

    Comparing the (2,2)-entries of (\ref{dim2 Eq Q2-1}), we can take $b_{11}=1$. Then the equalities of the (2,2)-entries of (\ref{dim2 Eq Q2-2}) and (\ref{dim2 Eq Q2-3}) are clear.
    
    Comparing the (1,1)-entries of (\ref{dim2 Eq Q2-1}), we need to solve $\varphi_q(b_{00})b_{00}\inv=T_{K,0}^{(q-1)h} X_0^{h(1-\varphi)}$. By Lemma \ref{dim2 Lem u}(i),(iv) we can take $b_{00}=\bigbra{uT_{K,0}\inv}^{-h}$. Then the equalities of the (1,1)-entries of (\ref{dim2 Eq Q2-2}) and (\ref{dim2 Eq Q2-3})
    follow directly from Lemma \ref{dim2 Lem u}(ii),(iii).

    Comparing the (1,2)-entries of (\ref{dim2 Eq Q2-1}), we need to solve
    \begin{equation*}
        b_{00}\inv\lambda_0T_{K,0}^{-(q-1)h}\varphi_q(b_{01})+b_{00}\inv\lambda_1D^{\LT}\varphi_q(b_{11})-b_{00}\inv b_{01}b_{11}\inv\lambda_1\varphi_q(b_{11})=\lambda_1D^X.
    \end{equation*}
    Replacing $b_{00}$, $b_{11}$ by their previous values, we get:
    \begin{equation}\label{dim2 Eq Q2-4}
        \bbra{\id-\lambda_0\lambda_1\inv T_{K,0}^{-(q-1)h}\varphi_q}(b_{01})=D_{01}\eqdef D^{\LT}-\bigbra{uT_{K,0}\inv}^{-h}D^X.
    \end{equation}
    Without loss of generality, we may assume that one of $c_0,\ldots,c_{f-1},c_{\tr},c_{\unr}$ is $1$ and the others are $0$. We give the proof of the following case needed in \S\ref{dim2 Sec proof} and leave the other cases to Appendix \ref{dim2 Sec app}.

    \paragraph{\textbf{Case 1:}} $c_j=1$ for some $0\leq j\leq f-1$, $h_j\neq0$ and $h_{j-1}\neq p-1$.

    By definition, we have
    \begin{align}
        D_{01}&=D_j^{\LT}-\bigbra{uT_{K,0}\inv}^{-h}D_j^X\notag\\
        &=T_{K,0}^{-(q-1)[h]_{j-1}}-\bigbra{uT_{K,0}\inv}^{-h}X_0^{[h]_{j-1}(1-\varphi)}\notag\\
        &=T_{K,0}^{-(q-1)[h]_{j-1}}\bbbra{1-\bigbra{uT_{K,0}\inv}^{-\bbra{h+(q-1)[h]_{j-1}}}}\notag\\
        &\in T_{K,0}^{-(q-1)[h]_{j-1}}\bbbra{1-\bbbra{1+B\!\bbra{\tfrac{(q-1)(p-1)}{p}}}^{p^j\ZZ}}\notag\\
        &\subseteq T_{K,0}^{-(q-1)[h]_{j-1}}B\!\bbra{(q\!-\!1)(p\!-\!1)p^{j-1}}\subseteq B\cc\bra{h}\label{dim2 Eq Q2-8},
    \end{align}
    where the third equality uses Lemma \ref{dim2 Lem u}(i), the first inclusion follows from Lemma \ref{dim2 Lem uTinv simple}, and the last inclusion follows from Lemma \ref{dim2 Lem hj}(i). By Lemma \ref{dim2 Lem unique Ainfty'}, we take $b_{01}\in A_{\infty}'$ to be the unique solution of (\ref{dim2 Eq Q2-4}) satisfying $\abs{b_{01}}<p^{-h}$.

    \hspace{\fill}

    Then we check the equality of the (1,2)-entries of (\ref{dim2 Eq Q2-2}) for the previous values of $b_{00},b_{01},b_{11}$, or equivalently (for $a\in\cO_K\x$)
    \begin{equation}\label{dim2 Eq Q2-5}
        \bbra{f_a^{\LT}}^h(a,1,\ldots,1)(b_{01})+E_a^{\LT}-b_{01}=\bigbra{uT_{K,0}\inv}^{-h}E_a^X.
    \end{equation}
    By Lemma \ref{dim2 Lem norm}(i),(iii),(iv) and $q-1>h$, each term of (\ref{dim2 Eq Q2-5}) has norm $<p^{-h}$, hence by Lemma \ref{dim2 Lem unique Ainfty'} it suffices to check the equality after applying the operator $\bigbra{\id-\lambda_0\lambda_1\inv T_{K,0}^{-(q-1)h}\varphi_q}$. We have
    \begin{align*}
    &\begin{aligned}
        &\bbra{\id-\lambda_0\lambda_1\inv T_{K,0}^{-(q-1)h}\varphi_q}\!\bbra{\bbra{f_a^{\LT}}^h(a,1,\ldots,1)(b_{01})}\\
        &\hspace{1.5cm}=\bbra{f_a^{\LT}}^h(a,1,\ldots,1)\bbra{\id-\lambda_0\lambda_1\inv T_{K,0}^{-(q-1)h}\varphi_q}\bbra{b_{01}}\qquad\text{(by Lemma \ref{dim2 Lem op}(i)})\\
        &\hspace{1.5cm}=\bbra{f_a^{\LT}}^h(a,1,\ldots,1)\bbra{D^{\LT}-\bigbra{uT_{K,0}\inv}^{-h}D^X}\qquad\text{(by \eqref{dim2 Eq Q2-4})}\\
        &\hspace{1.5cm}=\bbra{f_a^{\LT}}^ha(D^{\LT})-f_{a,0}^{h(1-\varphi)/(1-q)}\bigbra{uT_{K,0}\inv}^{-h}a(D^X);\qquad\text{(by Lemma \ref{dim2 Lem u}(iii)})
    \end{aligned}\\
    &\begin{aligned}
        \bbra{\id-\lambda_0\lambda_1\inv T_{K,0}^{-(q-1)h}\varphi_q}\!\bigbra{E_a^{\LT}}=D^{\LT}-\bbra{f_a^{\LT}}^ha(D^{\LT});\qquad\text{(by Proposition \ref{dim2 Prop pair LT}(i)})
    \end{aligned}\\
    &\begin{aligned}
        \bbra{\id-\lambda_0\lambda_1\inv T_{K,0}^{-(q-1)h}\varphi_q}\bbra{b_{01}}=D^{\LT}-\bigbra{uT_{K,0}\inv}^{-h}D^X;\qquad\text{(by \eqref{dim2 Eq Q2-4})}
    \end{aligned}\\
    &\begin{aligned}
        &\bbra{\id-\lambda_0\lambda_1\inv T_{K,0}^{-(q-1)h}\varphi_q}\!\bbra{\bigbra{uT_{K,0}\inv}^{-h}E_a^X}\\
        &\hspace{1.5cm}=\bigbra{uT_{K,0}\inv}^{-h}\bbra{\id-\lambda_0\lambda_1\inv X_0^{h(1-\varphi)}\varphi_q}\bbra{E_a^X}\qquad\text{(by Lemma \ref{dim2 Lem op}(iii))}\\
        &\hspace{1.5cm}=\bigbra{uT_{K,0}\inv}^{-h}\bbra{D^X-f_{a,0}^{h(1-\varphi)/(1-q)}a(D^X)}.\qquad\text{(by Proposition \ref{dim2 Prop pair A}(i))}
    \end{aligned}
    \end{align*}
    Hence the equality (\ref{dim2 Eq Q2-5}) holds.

    \hspace{\fill}

    Finally, we check the equality of the (1,2)-entries of (\ref{dim2 Eq Q2-3}) for the previous values of $b_{00},b_{01},b_{11}$, or equivalently (for $(a_0,\ldots,a_{f-1})\in\Delta_1$)
    \begin{equation}\label{dim2 Eq Q2-6}
        \bbra{f_{a_0}^{\LT}}^h(a_0,\ldots,a_{f-1})(b_{01})+E_{a_0}-b_{01}=0.
    \end{equation}
    By Lemma \ref{dim2 Lem norm}(i),(iii) and $q-1>h$, each term of (\ref{dim2 Eq Q2-6}) has norm $<p^{-h}$, hence by Lemma \ref{dim2 Lem unique Ainfty'} it suffices to check the equality after applying the operator $\bigbra{\id-\lambda_0\lambda_1\inv T_{K,0}^{-(q-1)h}\varphi_q}$. We have
    \begin{align*}
        \begin{aligned}
        &\bbra{\id-\lambda_0\lambda_1\inv T_{K,0}^{-(q-1)h}\varphi_q}\!\bbra{\bbra{f_{a_0}^{\LT}}^h(a_0,\ldots,a_{f-1})(b_{01})}\\
        &\hspace{1.5cm}=\bbra{f_{a_0}^{\LT}}^h(a_0,\ldots,a_{f-1})\bbra{D^{\LT}-\bigbra{uT_{K,0}\inv}^{-h}D^X}\qquad\text{(by Lemma \ref{dim2 Lem op}(ii)})\\
        &\hspace{1.5cm}=\bbra{f_{a_0}^{\LT}}^ha_0(D^{\LT})-\bigbra{uT_{K,0}\inv}^{-h}D^X.\qquad\text{(by Lemma \ref{dim2 Lem u}(ii)})
    \end{aligned}
    \end{align*}
    Here we recall that $D^X\in A$, hence is invariant under $\Delta_1$. We also have
    \begin{align*}
    &\begin{aligned}
        \bbra{\id-\lambda_0\lambda_1\inv T_{K,0}^{-(q-1)h}\varphi_q}\!\bbra{E_{a_0}^{\LT}}=D^{\LT}-\bbra{f_{a_0}^{\LT}}^ha_0(D^{\LT});\qquad\text{(by Proposition \ref{dim2 Prop pair LT}(i))}
    \end{aligned}\\
    &\begin{aligned}
        \bbra{\id-\lambda_0\lambda_1\inv T_{K,0}^{-(q-1)h}\varphi_q}\bbra{b_{01}}=D^{\LT}-\bigbra{uT_{K,0}\inv}^{-h}D^X.\qquad\text{(by \eqref{dim2 Eq Q2-4})}
    \end{aligned}
    \end{align*}
    Hence the equality (\ref{dim2 Eq Q2-6}) holds.
\end{proof}

\begin{remark}
    By \cite[Cor.~2.6.7]{BHHMS3}, the functor $\rhobar\mapsto D_A^{(0)}(\rhobar)$ is compatible with tensor products. Since we have $D_A^{(0)}\bigbra{\omega_f^h\unr(\lambda)}\cong D_{A,\sigma_0}\bigbra{\omega_f^h\unr(\lambda)}$ for all $h\in\ZZ$ and $\lambda\in\FF\x$ by \cite[Thm.~2.9.5]{BHHMS3} and since any reducible $2$-dimensional mod $p$ representation of $G_K$ is isomorphic to $\rhobar$ as in (\ref{dim2 Eq rhobar}) up to twist, we know $D_A^{(0)}(\rhobar)$ for all $2$-dimensional mod $p$ representations $\rhobar$ of $G_K$ (the irreducible case being treated in \cite[Thm.~2.9.5]{BHHMS3}) when $p\geq5$.
\end{remark}

\section{The main theorem on \texorpdfstring{$D_A(\pi)$}.}\label{dim2 Sec proof}

In this section, we recall the results of \cite{Wang2} on $D_A(\pi)$ and finish the proof of Theorem \ref{dim2 Thm main intro}. To do this, we need to prove that certain constants appearing on $D_A(\pi)$ and on $D_A^{\otimes}(\rhobar)$ match, see Proposition \ref{dim2 Prop mu/mumu}.

We let $\rhobar:G_K\to\GL_2(\FF)$ be of the following form:
\begin{equation}\label{dim2 Eq rhobar rj}
    \rhobar\cong\pmat{\omega_f^{\sum\nolimits_{j=0}^{f-1}(r_j+1)p^j}\unr(\xi)&*\\0&\unr(\xi\inv)}
\end{equation}
with $\xi\in\FF\x$, $0\leq r_j\leq p-3$ for $0\leq j\leq f-1$ and $r_j\neq0$ for some $j$. Up to enlarging $\FF$, we fix an $f$-th root $\sqrt[f]{\xi}\in\FF\x$ of $\xi$. By Theorem \ref{dim2 Thm basis LT}(iii) (with $h_j=r_j+1$, $\lambda_0=\xi$ and $\lambda_1=\xi\inv$), the Lubin--Tate $(\varphi,\OK\x)$-module $D_K(\rhobar)$ associated to $\rhobar$ has the following form ($a\in\OK\x$):

\begin{equation}\label{dim2 Eq Lubin-Tate rhobar}
\left\{\begin{array}{cll}
    D_K(\rhobar)&=&\prod\limits_{j=0}^{f-1}D_{K,\sigma_j}(\rhobar)=\prod\limits_{j=0}^{f-1}\bbra{\FF\dbra{T_{K,\sigma_j}}e_0^{(j)}\oplus\FF\dbra{T_{K,\sigma_j}}e_1^{(j)}}\\
    \varphi\bra{e_0^{(j+1)}\ e_1^{(j+1)}}&=&(e_0^{(j)}\ e_1^{(j)})\Mat(\varphi^{(j)})\\
    a(e_0^{(j)}\ e_1^{(j)})&=&(e_0^{(j)}\ e_1^{(j)})\Mat(a^{(j)}),
\end{array}\right.
\end{equation}
where 
\begin{equation}\label{dim2 Eq Lubin-Tate rhobar phi action}
    \Mat(\varphi^{(j)})=\pmat{\sqrt[f]{\xi}\,T_{K,\sigma_j}^{-(q-1)(r_j+1)}&\sqrt[f]{\xi}\inv d_j\\0&\sqrt[f]{\xi}\inv}
\end{equation}
for some $d_j\in\FF$ and $\Mat(a^{(j)})\in I_2+\M_2\bigbra{T_{K,\sigma_j}^{q-1}\FF\ddbra{T_{K,\sigma_j}^{q-1}}}$ which uniquely determines $\Mat(a^{(j)})$. By Theorem \ref{dim2 Thm main}, Proposition \ref{dim2 Prop XY} and the assumption on $\rhobar$, the \'etale $(\varphi,\OK\x)$-module $D_A^{\otimes}(\rhobar)$ is obtained from $\bigotimes\nolimits_{i=0}^{f-1}D_{K,\sigma_j}(\rhobar)$ by the recipe $T_{K,\sigma_j}^{q-1}\mapsto \varphi(Y_j)/Y_j$. Hence, if we consider the $A$-basis $\bigset{e_J\eqdef\bigotimes\nolimits_{j=0}^{f-1}e_{\delta_{j\in J}}^{(j)}}_{J\subseteq\cJ}$ for $D_A^{\otimes}(\rhobar)$, the corresponding matrix $\Mat(\varphi)\in\GL_{2^f}(A)$ (with its rows and columns indexed by the subsets of $\cJ$) for the $\varphi$-action is given by
\begin{equation}\label{dim2 Eq Mat phi}
    \Mat(\varphi)_{J',J+1}=
    \begin{cases}
        \nu_{J+1,J'}\prod\limits_{j\notin J}Y_j^{(r_j+1)(1-\varphi)}&\text{if}~J'\subseteq J\\
        0&\text{if}~J'\nsubseteq J,
    \end{cases}
\end{equation}
where $\nu_{J,J'}\eqdef\sqrt[f]{\xi}^{|J^c|-|J|}\prod\nolimits_{j\in(J-1)\setminus J'}d_j$ for $J'\subseteq J-1$. Also, the corresponding matrices for the $\OK\x$-action satisfy $\Mat(a)\in I_{2^f}+\M_{2^f}(F_{1-p}A)$ for all $a\in\OK\x$.

We also describe the Fontaine--Laffaille module associated to $\rhobar$ (see \cite{FL82}).

\begin{lemma}\label{dim2 Lem FL}
    The Fontaine--Laffaille module $FL(\rhobar)$ associated to $\rhobar$ has the following form:
    \begin{equation}\label{dim2 Eq FL rhobar}
        \left\{\begin{array}{cll}
        FL(\rhobar)&=&\prod\limits_{j=0}^{f-1}FL_{\sigma_j}(\rhobar)=\prod\limits_{j=0}^{f-1}\bbra{\FF e_0^{(j)}\oplus\FF e_1^{(j)}}\\
        \Fil^{r_j+1}FL_{\sigma_j}(\rhobar)&=&\FF e_0^{(j)}\\
        \varphi_{r_{j+1}+1}(e_0^{(j+1)})&=&\sqrt[f]{\xi}\inv(e_0^{(j)}-d_{j+1}e_1^{(j)})\\
        \varphi(e_1^{(j+1)})&=&\sqrt[f]{\xi}\,e_1^{(j)},
    \end{array}\right.
\end{equation}
where $d_j\in\FF\x$ is as in (\ref{dim2 Eq Lubin-Tate rhobar phi action}).
\end{lemma}

\begin{proof}
    Let $T$ be the formal variable of the formal group $\GG_m$ such that the logarithm \cite[\S8.6]{Lan90} is given by the power series $\sum\nolimits_{n=0}^{\infty}p^{-n}T^{p^n}$. In particular, the uniformizer is $p$, hence $\Zp\ddbra{T}=\Zp\ddbra{X}$ where $X$ is the usual variable corresponding to the formal group law $(1+X)^p-1$. For $a\in\Zp$ we have power series $a_{\cyc}(T)\in aT+T^2\Zp\ddbra{T}$. Similar to \S\ref{dim2 Sec LT}, there is a covariant equivalence of categories between the category of finite-dimensional continuous representations of $\Gal(\ovl{K}/K)$ over $\FF$ and the category of \'etale $(\varphi,\Zp\x)$-modules over $\FF\otimes_{\Fp}\Fq\dbra{T}$, which is also equivalent to the category of \'etale $(\varphi_q,\Zp\x)$-modules over $\FF\dbra{T_{\sigma_i}}$ for each $0\leq i\leq f-1$.
    
    \hspace{\fill}

    \noindent\textbf{Claim.} The \'etale $(\varphi,\Zp\x)$-module $D(\rhobar)$ associated to $\rhobar$ has the following form ($a\in\Zp\x$):
    \begin{equation}\label{dim2 Eq FL claim}
    \left\{\begin{array}{cll}
        D(\rhobar)&=&\prod\limits_{j=0}^{f-1}D_{\sigma_{j}}(\rhobar)=\prod\limits_{j=0}^{f-1}\bbra{\FF\dbra{T_{\sigma_j}}e_0^{(j)}\oplus\FF\dbra{T_{\sigma_j}}e_1^{(j)}}\\
        \varphi(e_0^{(j+1)}\ e_1^{(j+1)})&=&(e_0^{(j)}\ e_1^{(j)})\Mat(\varphi^{(j)})\\
        a(e_0^{(j)}\ e_1^{(j)})&=&(e_0^{(j)}\ e_1^{(j)})\Mat(a^{(j)}),
    \end{array}\right.
    \end{equation}
    where 
    \begin{equation*}
        \Mat(\varphi^{(j)})=\pmat{\sqrt[f]{\xi}\,T_{\sigma_j}^{-(p-1)(r_{j+1}+1)}&\sqrt[f]{\xi}\inv d_{j+1}\\0&\sqrt[f]{\xi}\inv}
    \end{equation*}
    for the same $d_j$ as in (\ref{dim2 Eq Lubin-Tate rhobar phi action}) and $\Mat(a^{(j)})\in I_2+\M_2(T_{\sigma_j}^{p-1}\FF\ddbra{T_{\sigma_j}^{p-1}})$ which uniquely determines $\Mat(a^{(j)})$.

    \proof Recall from the proof of \cite[Prop.~2.8.1]{BHHMS3} that the canonical inclusion 
    \begin{equation*}
        \bB^+(R)^{\varphi=p}\into\bB^+(R)^{\varphi_q=p^f}
    \end{equation*}
    for any perfectoid $\FF$-algebra $R$ induces a map $Z_{\Zp}\to Z_{\OK}$ of perfectoid spaces over $\FF$, which is induced by the map 
    \begin{equation*}
        \tr:A_{\infty}\onto\FF\dbra{T^{p^{-\infty}}}
    \end{equation*}
    coming from the trace map $\FF\ddbra{K}\stackrel{\tr}{\onto}\FF\ddbra{\Qp}\cong\FF\ddbra{T^{p^{\infty}}}$. By the definition of $T$ and $X_i~(0\leq i\leq f-1)$, we have the relation in $\bB^+\big(\FF\dbra{T^{p^{-\infty}}}\big)$ which is analogous to \cite[(62)]{BHHMS3}:
    \begin{equation*}
        \sum\limits_{n\in\ZZ}[T^{p^{-n}}]p^n=\sum\limits_{n\in\ZZ}\sum\limits_{i=0}^{f-1}[\tr(X_i)^{p^{-nf-i}}]p^{nf+i}.
    \end{equation*}
    Hence we deduce that 
    \begin{equation}\label{dim2 Eq trXi}
        \tr(X_i)=T\quad\forall\ 0\leq i\leq f-1.
    \end{equation}
    By Theorem \ref{dim2 Thm main}, the \'etale $(\varphi_q,\OK\x)$-module $D_A^{(0)}(\rhobar)$ is obtained from $D_{K,\sigma_0}(\rhobar)$ (see (\ref{dim2 Eq Lubin-Tate rhobar})) by the recipe $T_{K,\sigma_0}\mapsto\varphi(X_0)/X_0$. Then by \cite[Prop.~2.8.1]{BHHMS3}, \cite[Remark~2.8.2]{BHHMS3} and (\ref{dim2 Eq trXi}), we conclude that the \'etale $(\varphi_q,\Zp\x)$ module $D_{\sigma_{f-1}}(\rhobar)$ is precisely as in (\ref{dim2 Eq FL claim}).\qed
    
    \hspace{\fill}

    Let $Q\eqdef\varphi(T)/T\in T^{p-1}+p(1+T\Zp\ddbra{T})$, where $\varphi$ acts on $\Zp\ddbra{T}$ as $p_{\cyc}$. By the proof of Lemma \ref{dim2 Lem congruence LT} in the case $f=1$, we have $a_{\cyc}(T)=aT$ for $a\in[\Fp\x]$. Then the commutativity of the action of $a\in\Zp\x$ with $[\Fp\x]$ implies that $a_{\cyc}(T)\in aT\bbra{1+T^{p-1}\Zp\ddbra{T^{p-1}}}$. We let
    \begin{equation*}
        \Lambda_a\eqdef\prod\limits_{i\geq0}\varphi^{1+if}(Q/a_{\cyc}(Q))\in1+T^{p-1}\Zp\ddbra{T^{p-1}}.
    \end{equation*}
    We construct a Wach module (see e.g.\,\cite[\S2.4]{CD11}) over $W(\FF)\otimes_{\Zp}\OK\ddbra{T}$ of the form ($a\in\Zp\x$):
    \begin{equation*}
    \left\{\begin{array}{cll}
        M&=&\prod\limits_{j=0}^{f-1}M^{(j)}=\prod\limits_{j=0}^{f-1}\bbra{W(\FF)\ddbra{T}e_0^{(j)}\oplus W(\FF)\ddbra{T}e_1^{(j)}}\\
        \varphi(e_0^{(j+1)}\ e_1^{(j+1)})&=&(e_0^{(j)}\ e_1^{(j)})\Mat(\varphi^{(j)})\\
        a(e_0^{(j)}\ e_1^{(j)})&=&(e_0^{(j)}\ e_1^{(j)})\Mat(a^{(j)})
    \end{array}\right.
    \end{equation*}
    with
    \begin{equation*}
    \begin{cases}
        \Mat(\varphi^{(j)})=\pmat{[\sqrt[f]{\xi}]\inv Q^{r_{j+1}+1}&0\\ [\sqrt[f]{\xi}]\inv[d_{j+1}]Q^{r_{j+1}+1}&[\sqrt[f]{\xi}]}\\
        \Mat(a^{(j)})=\pmat{P_a^{(j)}&0\\P_a^{(j)}E_a^{(j)}&1},
    \end{cases}
    \end{equation*}
    where $P_a^{(j)}\eqdef\prod\nolimits_{i=0}^{f-1}\varphi^i(\Lambda_a)^{r_{i+j+1}+1}\in1+T^{p-1}\Zp\ddbra{T^{p-1}}$, and $E_a^{(j)}\in T^{p-1}\Zp\ddbra{T^{p-1}}$ is the unique solution for the system of equations ($j\in\cJ$)
    \begin{equation*}
        E_a^{(j)}-[\sqrt[f]{\xi}]^2Q^{-(r_{j+1}+1)}\varphi(E_a^{(j+1)})=[d_{j+1}]\bbra{(P_a^{(j)})\inv-1}.
    \end{equation*}
    To prove uniqueness, up to dividing $p$ we may assume that $p\nmid(E_a^{(j)}-E_a^{\prime(j)})$ for some $j$, then we reduce modulo $p$ and compare the degrees in $T$. The existence of the solution follows as in the proof of \cite[Lemma~B.2(iv)]{Wang2}. Then one can check that $M$ is a Wach module over $W(\FF)\otimes_{\Zp}\OK\ddbra{T}$ such that $M\otimes_{\Zp\ddbra{T}}\FF\dbra{T}$ is the dual \'etale $(\varphi,\OK\x)$-module of $D(\rhobar)$.
    
    We give $M$ a filtration defined by 
    \begin{equation*}
        \Fil^iM\eqdef\set{x\in M:\varphi(x)\in Q^iM}.
    \end{equation*}
    Then for $f(T),g(T)\in\Zp\ddbra{T}$, we have
    \begin{align*}
        &f(T)e_0^{(j)}+g(T)e_1^{(j)}\in\Fil^iM^{(j)}\iff\\&\varphi(f(T))\bbra{[\sqrt[f]{\xi}]\inv Q^{r_j+1}e_0^{(j-1)}-[\sqrt[f]{\xi}]\inv[d_{j+1}]Q^{r_j+1}e_1^{(j-1)}}+\varphi(g(T))[\sqrt[f]{\xi}]e_1^{(j-1)}\in Q^iM^{(j-1)}.
    \end{align*}
    If $i\leq0$, this is automatic. If $1\leq i\leq r_j+1$, then we need $Q^i|\varphi(g(T))$, which is equivalent to $T^i|g(T)$. If $i>r_j+1$, then we need $Q^{i-(r_j+1)}|\varphi(f(T))$ and $Q^i|\varphi(g(T))$, which is equivalent to $T^{i-(r_j+1)}|f(T)$ and $T^i|g(T)$. To summarize, we have
    \begin{equation*}
        \Fil^iM^{(j)}=
        \begin{cases}
            W(\FF)\ddbra{T}e_0^{(j)}\oplus W(\FF)\ddbra{T}e_1^{(j)}&\text{if}~i\geq0\\
            W(\FF)\ddbra{T}e_0^{(j)}\oplus T^iW(\FF)\ddbra{T}e_1^{(j)}&\text{if}~1\leq i\leq r_j+1\\
            T^{i-(r_j+1)}W(\FF)\ddbra{T}e_0^{(j)}\oplus T^iW(\FF)\ddbra{T}e_1^{(j)}&\text{if}~i>r_j+1.
        \end{cases}
    \end{equation*}
    Then the ``module filtr\'e" over $W(\FF)$  associated to $M$ in \cite[Thm.~3]{Wac97} is of the form:
    \begin{equation*}
    \left\{\begin{array}{cll}
        M/TM&=&\prod\limits_{j=0}^{f-1}\bbra{W(\FF)e_0^{(j)}\oplus W(\FF)e_1^{(j)}}\\
        \Fil^{r_j+1}(M^{(j)}/TM^{(j)})&=&\FF e_0^{(j)}\\
        \varphi_{r_{j+1}+1}(e_0^{(j+1)})&=&[\sqrt[f]{\xi}]\inv(e_0^{(j)}-[d_{j+1}]e_1^{(j)})\\
        \varphi(e_1^{(j+1)})&=&[\sqrt[f]{\xi}]\,e_1^{(j)}.
    \end{array}\right.
    \end{equation*}
    Its reduction modulo $p$ is the Fontaine--Laffaille module in (\ref{dim2 Eq FL rhobar}), which is also the Fontaine--Laffaille module of $\rhobar$ by \cite[Thm.~1']{Wac97}. This completes the proof.
\end{proof}

\hspace{\fill}

Then we recall some results on $D_A(\pi)$ following \cite{Wang2}. Keep the notation of \S\ref{dim2 Sec intro}. We let $\pi$ be as in (\ref{dim2 Eq local factor}) with $\rbar$ satisfying the assumptions (i)-(v) above Theorem \ref{dim2 Thm main intro}. By \cite[Thm.~1.1]{DL21} we have $\pi^{K_1}=D_0(\rbar_v^{\vee})$ as $K\x\GL_2(\OK)$-representations, where $D_0(\rbar_v^{\vee})$ is the representation of $\GL_2(\Fq)$ defined in \cite[\S13]{BP12} and is viewed as a representation of $\GL_2(\OK)$ by inflation, and $K\x$ acts on $D_0(\rbar_v^{\vee})$ by the character $\det(\rbar_v^{\vee})\omega\inv$, where $\omega$ is the mod $p$ cyclotomic character. Since $12\leq r_j\leq p-15$ for all $j$, the proof of \cite[Thm.~6.3(i)]{Wang1} shows that $\pi$ satisfies (i),(ii),(iii) of \cite[Thm.~5.1]{Wang1}, hence satisfies the conditions (a),(b),(c) of \cite[\S6.4]{BHHMS1}. By \cite[Prop.~6.4.6]{BHHMS1} we deduce that $[\pi[\fm_{I_1}^3]:\chi]=1$ for any character $\chi:I\to\FF\x$ appearing in $\pi^{I_1}$, where $\fm_{I_1}$ is the maximal ideal of $\FF\ddbra{I_1}$, $\pi[\fm_{I_1}^3]$ is the set of elements of $\pi$ annihilated by $\fm_{I_1}^3$, and $[\pi[\fm_{I_1}^3]:\chi]$ is the multiplicity of $\chi$ in the semisimplification of $\pi[\fm_{I_1}^3]$ as $I$-representations. In particular, $\pi$ satisfies the conditions (i),(ii) above \cite[Thm.~1.1]{Wang2} with $\rhobar=\rbar_v^{\vee}$. Twisting $\rhobar$ and $\pi$ using \cite[Lemma~2.9.7]{BHHMS3} and \cite[Lemma~3.1.1]{BHHMS3}, we may assume that $\rhobar$ is as in (\ref{dim2 Eq rhobar rj}) with $\max\set{12,2f+1}\leq r_j\leq p-\max\set{15,2f+3}$ for all $j$. In particular, $p$ acts trivially on $\pi$.

From now on, we assume that $|W(\rhobar)|=1$, which is equivalent to $J_{\rhobar}=\emptyset$ by \cite[Prop.~A.3]{Bre14}, where $J_{\rhobar}\subseteq\cJ$ is the subset defined in \cite[(17)]{Bre14}. In particular, by \cite[(18)]{Bre14} with $e^j=e_1^{(f-j)}$, $f^j=e_0^{(f-j)}$, $\alpha_j=\sqrt[f]{\xi}$, $\beta_j=\sqrt[f]{\xi}\inv$ and $\mu_j=d_{f+1-j}$ for all $j\in\cJ$ in \cite[(16)]{Bre14}, we deduce that $d_j\in\FF\x$ for all $j\in\cJ$ (see (\ref{dim2 Eq Lubin-Tate rhobar phi action}) for $d_j$). We denote $\sigma_{\emptyset}\eqdef\soc_{\GL_2(\OK)}\pi$.

We write $\un{i}$ for an element $(i_0,\ldots,i_{f-1})\in\ZZ^f$, and we write $\un{Y}^{\un{i}}$ for $\prod\nolimits_{j=0}^{f-1}Y_{j}^{i_j}\in A$. For $J\subseteq\cJ$, we define $\un{e}^J\in\ZZ^f$ by $e^J_j\eqdef\delta_{j\in J}$. We say that $\un{i}\leq\un{i}'$ if $i_j\leq i'_j$ for all $j$. For each $J\subseteq\cJ$, we define $\un{s}^J,\un{r}^J\in\ZZ^f$ by 
\begin{align}
    \label{dim2 Eq sJ}s^J_j&\eqdef
    \begin{cases}
        r_j,&\text{if}~j\notin J,\ j+1\notin J\\
        r_j+1,&\text{if}~j\in J,\ j+1\notin J\\
        p-2-r_j,&\text{if}~j\notin J,\ j+1\in J\\
        p-1-r_j,&\text{if}~j\in J,\ j+1\in J;
    \end{cases}\\
    \label{dim2 Eq rJ}r^J_j&\eqdef
    \begin{cases}
        0,&\text{if}~j\notin J,\ j+1\notin J\\
        -1,&\text{if}~j\in J,\ j+1\notin J\\
        r_j+1,&\text{if}~j\notin J,\ j+1\in J\\
        r_j,&\text{if}~j\in J,\ j+1\in J.
    \end{cases}
\end{align}
We define the character $\chi_J:I\to\FF\x$ by $\smat{a&b\\pc&d}\mapsto(\ovl{a})^{\un{s}^J+\un{r}^J}(\ovl{d})^{\un{r}^J}$. Here, for $x\in\FF$ and $\un{i}\in\ZZ^f$ we define $x^{\un{i}}\eqdef x^{\sum\nolimits_{j=0}^{f-1}i_jp^j}$. We identify $\pi^{K_1}$ with $D_0(\rhobar)$. Then by the proof of \cite[Lemma~4.1(ii)]{Wang2} we have $\pi^{I_1}=D_0(\rhobar)^{I_1}=\bigoplus_{J\subseteq\cJ}\chi_J$ as $I$-representations. For each $J\subseteq\cJ$ we fix a choice of $0\neq v_J\in D_0(\rhobar)^{I_1}$ with $I$-character $\chi_J$, which is unique up to scalar. We recall the following results of \cite{Wang2} in the case $J_{\rhobar}=\emptyset$.

\begin{proposition}\label{dim2 Prop Wang2}
\begin{enumerate}
    \item (\cite[Prop.~4.2]{Wang2})
    Let $J\subseteq\cJ$ and $\un{i}\in\ZZ^f$ such that $\un{0}\leq\un{i}\leq\un{f}$. Then there exists a unique $H$-eigenvector $\un{Y}^{-\un{i}}v_J\in D_0(\rhobar)$ satisfying
    \begin{enumerate}
        \item 
        $Y_j^{i_j+1}\bbra{\un{Y}^{-\un{i}}v_J}=0\ \forall\,j\in\cJ$;
        \item 
        $\un{Y}^{\un{i}}\bbra{\un{Y}^{-\un{i}}v_J}=v_J$.
    \end{enumerate}
    \item (\cite[Prop.~5.10]{Wang2})
    Let $J,J'\subseteq\cJ$ such that $J'\neq\cJ$ and $J'+1\subseteq J\Delta J'\eqdef(J\setminus J')\sqcup(J'\setminus J)$. Then there exists a unique element $\mu_{J,J'}\in\FF\x$, such that
    \begin{equation*}
        \bbbra{\prod\limits_{j+1\in J\Delta J'}Y_j^{s^{J'}_j}\prod\limits_{j+1\notin J\Delta J'}Y_j^{p-1}}\smat{p&0\\0&1}\bbra{\un{Y}^{-\un{e}^{J\cap J'}}v_J}=\mu_{J,J'}v_{J'}.
    \end{equation*}
    \item (\cite[Prop.~5.12]{Wang2})
    We write $x_{\emptyset,\un{r}}\eqdef\mu_{\emptyset,\emptyset}\inv\un{Y}^{\un{p}-\un{1}-\un{r}}\smat{p&0\\0&1}v_{\emptyset}$ so that $\un{Y}^{\un{r}}x_{\emptyset,\un{r}}=v_{\emptyset}$ by (ii). Then for $\emptyset\neq J\subseteq\cJ$, there exists a unique element $\mu_{J,\cJ}\in\FF\x$ such that
    \begin{equation*}
        \prod\limits_{j+1\notin J}Y_j^{p-1-r_j}\smat{p&0\\0&1}v_J=\mu_{J,\cJ}v_{\cJ}+\mu_{J,\emptyset}x_{\emptyset,\un{r}},
    \end{equation*}
    where $\mu_{J,\emptyset}$ is defined in (ii).
    \item (\cite[Lemma~5.13]{Wang2})
    Let $J_1,J_2,J_3,J_4\subseteq\cJ$. Then we have
    \begin{equation}\label{dim2 Eq ratio}
        \frac{\mu_{J_1,J_3}}{\mu_{J_1,J_4}}=\frac{\mu_{J_2,J_3}}{\mu_{J_2,J_4}}
    \end{equation}
    whenever all of them are defined in either (ii) or (iii).
\end{enumerate}
\end{proposition}

We extend the definition of $\mu_{J,J'}$ to 
arbitrary $J,J'\subseteq\cJ$ by the formula
\begin{equation*}
    \begin{cases}
        \mu_{J,J'}\eqdef\mu_{(J')^c,J'}\mu_{J,\emptyset}/{\mu_{(J')^c,\emptyset}}~\text{if}~J'\neq\cJ;\\
        \mu_{\emptyset,\cJ}\eqdef\mu_{\emptyset,\emptyset}\mu_{\cJ,\cJ}/{\mu_{J,\emptyset}}
    \end{cases}
\end{equation*}
(and $\mu_{J,\cJ}$ as in Proposition \ref{dim2 Prop Wang2}(iii) for $J\neq\emptyset$). Then the equation (\ref{dim2 Eq ratio}) holds for arbitrary $J_1,J_2,J_3,J_4\subseteq\cJ$. By \cite[Thm.~1.1]{Wang2} and the construction of \cite[\S3.2]{BHHMS3}, $\Hom_A(D_A(\pi),A)(1)$ is an \'etale $(\varphi,\OK\x)$-module over $A$ of rank $2^f$. Here for $D$ a $(\varphi,\OK\x)$-module over $A$, we write $D(1)$ to be $D$ with the action of $\varphi$ unchanged and the action of $a\in\OK\x$ multiplied by $N_{\Fq/\Fp}(a)$. Moreover, by \cite[Prop.~B.3(i),(iii)]{Wang2} and \cite[Cor.~B.4]{Wang2} there is an $A$-basis of $\Hom_A(D_A(\pi),A)(1)$ such that 
\begin{enumerate}
    \item 
    the corresponding matrix $\Mat(\varphi)'\in\GL_{2^f}(A)$ for the $\varphi$-action is given by 
    \begin{equation}\label{dim2 Eq Mat' phi}
        \Mat(\varphi)'_{J',J+1}=
        \begin{cases}
            \gamma_{J+1,J'}\prod\limits_{j\notin J}Y_j^{(r_j+1)(1-\varphi)}&\text{if}~J'\subseteq J\\
            0&\text{if}~J'\nsubseteq J,
        \end{cases}
    \end{equation}
    where $\gamma_{J,J'}\eqdef(-1)^{f-1}\varepsilon_{J'}\mu_{J,J'}$ with $\varepsilon_{J}\eqdef(-1)^{|J\cap(J-1)|}$ if $J\neq\cJ$ and $\varepsilon_{\cJ}\eqdef(-1)^{f-1}$.
    \item 
    the corresponding matrices $\Mat(a)'$ for the $\OK\x$-action satisfy $\Mat(a)'_{J,J}\in1+F_{1-p}A$ for all $a\in\OK\x$ and $J\subseteq\cJ$, which uniquely determines $\Mat(a)'$.
\end{enumerate}

We also extend the definition of $\nu_{J,J'}$ (see \eqref{dim2 Eq Lubin-Tate rhobar phi action}) to all $J,J'\subseteq\cJ$ by the formula
\begin{equation*}
    \nu_{J,J'}\eqdef\sqrt[f]{\xi}^{|J^c|-|J|}\frac{\prod\limits_{j\notin J'}d_j}{\prod\limits_{j+1\notin J}d_j},
\end{equation*}
where $d_j\in\FF\x$ is as in (\ref{dim2 Eq Lubin-Tate rhobar phi action}). Then it is easy to check that (\ref{dim2 Eq ratio}) holds for $\nu_{J,J'}$, and that 
\begin{equation}\label{dim2 Eq nu/nunu}
    \frac{\nu_{J,\emptyset}}{\nu_{J^c,\emptyset}\nu_{J,J^c}}=\sqrt[f]{\xi}^{|J^c|-|J|}\frac{\prod\limits_{j\notin J,j+1\in J}d_j}{\prod\limits_{j\in J,j+1\notin J}d_j}.
\end{equation}

\begin{proposition}\label{dim2 Prop mu/mumu}
    Keep the assumptions of $\pi$ and assume that $|W(\rhobar)|=1$. Then for $J\subseteq\cJ$ we have (see \eqref{dim2 Eq Mat' phi} for $\gamma_{J,J'}$ and $\varepsilon_{J}$)
    \begin{equation}\label{dim2 Eq constants}
        \frac{\gamma_{J,\emptyset}}{\gamma_{J^c,\emptyset}\gamma_{J,J^c}}=(-1)^{f-1}\varepsilon_{J^c}\frac{\mu_{J,\emptyset}}{\mu_{J^c,\emptyset}\mu_{J,J^c}}=\sqrt[f]{\xi}^{|J^c|-|J|}\frac{\prod\limits_{j\notin J,j+1\in J}d_j}{\prod\limits_{j\in J,j+1\notin J}d_j}.
    \end{equation}
\end{proposition}

\begin{proof}
    The first equality follows directly from the definition. Then we prove the second equality. Since the LHS of (\ref{dim2 Eq constants}) is unchanged when we rescale the basis $(v_J)_{J\in\cJ}$ and since $\chi_{J^c}$ is the conjugation of $\chi_J$ by the matrix $\smat{0&1\\p&0}$, we may assume that $\smat{0&1\\p&0}v_J=v_{J^c}$ for all $J$ (note that $p$ acts trivially on $\pi$). 
    
    First we compute $\mu_{J,\emptyset}/\mu_{J^c,\emptyset}$. We apply \cite[Thm.~1.1]{BD14} with $J$ replaced by $J-1$ and $v=v_J$. Together with \cite[Lemma~3.2.2.5(i)]{BHHMS2}, we get
    \begin{multline}\label{dim2 Eq mu/mumu 1}
        (-1)^{f-1}\bbbra{\scalebox{1}{$\prod\limits_{j+1\in J}$}(p-1-r_j)!\scalebox{1}{$\prod\limits_{j+1\in J}$}Y_j^{r_j}\scalebox{1}{$\prod\limits_{j+1\notin J}$}Y_j^{p-1}}\smat{p&0\\0&1}v_J\\
        =x(J-1)(-1)^{f-1}\bbbra{\scalebox{1}{$\prod\limits_{j+1\notin J}$}(p-1-r_j)!\scalebox{1}{$\prod\limits_{j+1\notin J}$}Y_j^{r_j}\scalebox{1}{$\prod\limits_{j+1\in J}$}Y_j^{p-1}}\smat{p&0\\0&1}v_{J^c},
    \end{multline}
    where $x(J-1)$ is computed by \cite[Thm.~1.2]{BD14} with $\alpha_{v,\sigma_j}=\sqrt[f]{\xi}$, $\beta_{v,\sigma_j}=\sqrt[f]{\xi}\inv$ and $x_{v,\sigma_j}=-d_j$ by Lemma \ref{dim2 Lem FL}. By Proposition \ref{dim2 Prop Wang2}(ii) applied to $(J,\emptyset)$ and $(J^c,\emptyset)$, we deduce from (\ref{dim2 Eq mu/mumu 1}) that
    \begin{align}\label{dim2 Eq mu/mu}
        \frac{\mu_{J,\emptyset}}{\mu_{J^c,\emptyset}}&=x(J-1)\frac{\prod\limits_{j+1\notin J}(p-1-r_j)!}{\prod\limits_{j+1\in J}(p-1-r_j)!}\notag\\
        &=\bbbra{-\sqrt[f]{\xi}^{|J^c|-|J|}\frac{\prod\limits_{j+1\in J,j\notin J}(-d_j)(r_j+1)}{\prod\limits_{j+1\notin J,j\in J}(-d_j)(r_j+1)}}\frac{\prod\limits_{j+1\in J}(-1)^{r_j+1}r_j!}{\prod\limits_{j+1\notin J}(-1)^{r_j+1}r_j!}\notag\\
        &=-\sqrt[f]{\xi}^{|J^c|-|J|}\frac{\bbbra{\prod\limits_{j\in J,j+1\in J}(-1)^{r_j+1}r_j!}\!\!\bbbra{\prod\limits_{j\notin J,j+1\in J}(-1)^{r_j}(r_j+1)!d_j}}{\bbbra{\prod\limits_{j\notin J,j+1\notin J}(-1)^{r_j+1}r_j!}\!\!\bbbra{\prod\limits_{j\in J,j+1\notin J}(-1)^{r_j}(r_j+1)!d_j}},
    \end{align}
    where the second equality follows from \cite[Thm.~1.2]{BD14} and 
    \begin{equation}\label{dim2 Eq factorial congruence}
        \bigbra{(p-1-r)!}\inv\equiv(-1)^{r+1}r!~\mod p~\forall\,0\leq r\leq p-1.
    \end{equation}

    Next we compute $\mu_{J,J^c}$ for $J\neq\emptyset$. By \cite[Lemma 5.1(ii)]{Wang2} and its proof (with $J_{\rhobar}=\emptyset$), there is a $\GL_2(\OK)$-equivariant surjection (see \cite[\S3]{Wang2} for the element $\phi\in \Ind_I^{\GL_2(\OK)}(\chi_J^s)$)
    \begin{equation*}
    \begin{aligned}
        \Ind_I^{\GL_2(\OK)}(\chi_J^s)&\onto\bang{\GL_2(\OK)\smat{p&0\\0&1}v_J}\\
        \phi&\mapsto\smat{0&1\\p&0}v_J=v_{J^c}
    \end{aligned}
    \end{equation*}
    which is not an isomorphism when $J\neq\emptyset$, hence it maps the socle of $\Ind_I^{\GL_2(\OK)}(\chi_J^s)$ to zero. By definition, it is elementary to check that $(-1)^{\un{s}^J+\un{r}^J}=(-1)^{\un{r}^{J^c}}$ (see (\ref{dim2 Eq sJ}) for $\un{s}^J$ and (\ref{dim2 Eq rJ}) for $\un{r}^J$). Then we deduce from \cite[Lemma~3.2(iii)(a)]{Wang2} that
    \begin{equation}\label{dim2 Eq mu/mumu 2}
        \un{Y}^{\un{p}-\un{1}-\un{s}^J}\smat{p&0\\0&1}v_J+(-1)^{f-1}(-1)^{\un{r}^{J^c}}\bbbra{\scalebox{1}{$\prod\limits_{j=0}^{f-1}$}(s^J_j)!}v_{J^c}=0.
    \end{equation}
    By Proposition \ref{dim2 Prop Wang2}(ii) applied to $(J,J^c)$, we deduce from (\ref{dim2 Eq mu/mumu 2}) that
    \begin{equation}\label{dim2 Eq mu J,Jc}
        \mu_{J,J^c}=(-1)^{\un{r}^{J^c}+\un{1}}/\bbbra{\scalebox{1}{$\prod\limits_{j=0}^{f-1}$}(s^J_j)!}=\frac{\bbbra{\prod\limits_{j\in J,j+1\in J}(-1)^{r_j}r_j!}\!\!\bbbra{\prod\limits_{j\notin J,j+1\in J}(-1)^{r_j}(r_j+1)!}}{\bbbra{\prod\limits_{j\notin J,j+1\notin J}(-1)^{r_j+1}r_j!}\!\!\bbbra{\prod\limits_{j\in J,j+1\notin J}(-1)^{r_j}(r_j+1)!}},
    \end{equation}
    where the second equality follows from (\ref{dim2 Eq sJ}), (\ref{dim2 Eq rJ}) and (\ref{dim2 Eq factorial congruence}). Combining (\ref{dim2 Eq mu/mu}) and (\ref{dim2 Eq mu J,Jc}), we get
    \begin{equation*}
        \frac{\mu_{J,\emptyset}}{\mu_{J^c,\emptyset}\mu_{J,J^c}}=(-1)^{|J\cap(J-1)|+1}\sqrt[f]{\xi}^{|J^c|-|J|}\frac{\prod\limits_{j\notin J,j+1\in J}d_j}{\prod\limits_{j\in J,j+1\notin J}d_j}.
    \end{equation*}
    By definition, it is elementary to check that $(-1)^{f-1}\varepsilon_{J^c}=(-1)^{|J\cap(J-1)|+1}$ for $J\neq\emptyset$. This proves the proposition for $J\neq\emptyset$.

    It remains to prove the proposition for $J=\emptyset$. By (\ref{dim2 Eq ratio}) we have $\mu_{\emptyset,\emptyset}/\bra{\mu_{\cJ,\emptyset}\mu_{\emptyset,\cJ}}=\mu_{\cJ,\cJ}\inv$, hence it suffices to show that $\mu_{\cJ,\cJ}=\xi\inv$. We let
    \begin{equation}\label{dim2 Eq mu/mumu 3}
        y\eqdef\un{Y}^{\un{p}-\un{1}-\un{r}}\smat{p&0\\0&1}v_{\emptyset}+(-1)^{f-1}(-1)^{\un{r}}\bbbra{\scalebox{1}{$\prod\limits_{j=0}^{f-1}$}r_j!}\inv\!\!\!\!v_{\cJ}\in\pi.
    \end{equation}
    By \cite[Lemma~3.2(iii)(a)]{Wang2}, both the elements $y$ and $\smat{p&0\\0&1}v_{\cJ}=\smat{0&1\\1&0}v_{\emptyset}$ are nonzero and lie in the $I$-cosocle of $\sigma_{\emptyset}=\soc_{\GL_2(\OK)}\pi$, hence they are equal up to a scalar. By Proposition \ref{dim2 Prop Wang2}(ii) applied to $(\emptyset,\emptyset)$ and since $v_{\cJ}\in\pi^{I_1}$, we have $\un{Y}^{\un{r}}y=\mu_{\emptyset,\emptyset}v_{\emptyset}$. By Proposition \ref{dim2 Prop Wang2}(iii) applied to $J=\cJ$, we have (see Proposition \ref{dim2 Prop Wang2}(iii) for $x_{\emptyset,\un{r}}$)
    \begin{equation}\label{dim2 Eq mu/mumu 4}
        \un{Y}^{\un{r}}\smat{p&0\\0&1}v_{\cJ}=\mu_{\cJ,\cJ}\un{Y}^{\un{r}}v_{\cJ}+\mu_{\cJ,\emptyset}\un{Y}^{\un{r}}x_{\emptyset,\un{r}}=\mu_{\cJ,\emptyset}v_{\emptyset},
    \end{equation}
    where the second equality uses $v_{\cJ}\in\pi^{I_1}$. Then we deduce from $\un{Y}^{\un{r}}y=\mu_{\emptyset,\emptyset}v_{\emptyset}$ and (\ref{dim2 Eq mu/mumu 4}) that $\smat{p&0\\0&1}v_{\cJ}=\bra{\mu_{\cJ,\emptyset}/\mu_{\emptyset,\emptyset}}y$, hence we have
    \begin{equation*}
        \mu_{\cJ,\cJ}v_{\cJ}+\mu_{\cJ,\emptyset}x_{\emptyset,\un{r}}=\smat{p&0\\0&1}v_{\cJ}=\frac{\mu_{\cJ,\emptyset}}{\mu_{\emptyset,\emptyset}}y=\mu_{\cJ,\emptyset}x_{\emptyset,\un{r}}+\frac{\mu_{\cJ,\emptyset}}{\mu_{\emptyset,\emptyset}}(-1)^{f-1}(-1)^{\un{r}}\bbbra{\scalebox{1}{$\prod\limits_{j=0}^{f-1}$}r_j!}\inv v_{\cJ},
    \end{equation*}
    where the first equality follows from Proposition \ref{dim2 Prop Wang2}(iii) applied to $J=\cJ$ and the last equality follows from (\ref{dim2 Eq mu/mumu 3}), which implies that
    \begin{equation*}
        \mu_{\cJ,\cJ}=\frac{\mu_{\cJ,\emptyset}}{\mu_{\emptyset,\emptyset}}(-1)^{f-1}(-1)^{\un{r}}\bbbra{\scalebox{1}{$\prod\limits_{j=0}^{f-1}$}r_j!}\inv=\xi\inv,
    \end{equation*}
    where the last equality follows from (\ref{dim2 Eq mu/mu}) applied to $J=\cJ$. This completes the proof.
\end{proof}

Finally, we need the following lemma.

\begin{lemma}\label{dim2 Lem A/AA}
    Let $B\in\M_{2^f}(\FF)$ with nonzero entries whose rows and columns are indexed by the subsets of $\cJ$ and satisfies $B_{J_1,J_3}/B_{J_1,J_4}=B_{J_2,J_3}/B_{J_2,J_4}$ for all $J_1,J_2,J_3,J_4\subseteq\cJ$. Then up to conjugation by diagonal matrices, $B$ is uniquely determined by the quantities 
    \begin{equation}\label{dim2 Eq A/AA statement}
        \sset{\frac{B_{J,\emptyset}}{B_{J^c,\emptyset}B_{J,J^c}}}_{J\subseteq\cJ}.
    \end{equation}
\end{lemma}

\begin{proof}
    First, it is easy to check that conjugation by a diagonal matrix does not change these quantities. Next, given such a matrix $B$, after conjugation we may assume that $B_{J,\emptyset}=1$ for all $J\neq\emptyset$. Then $B_{\emptyset,
    \emptyset}$ is determined by letting $J=\cJ$ in (\ref{dim2 Eq A/AA statement}), and the rest of the entries of $B$ are determined by the formula (for $J'\neq\emptyset$)
    \begin{equation*}
        B_{J,J'}=B_{(J')^c,J'}\frac{B_{J,\emptyset}}{B_{(J')^c,\emptyset}}=\bbra{\frac{B_{(J')^c,\emptyset}}{B_{J',\emptyset}B_{(J')^c,J'}}}\inv\frac{B_{J,\emptyset}}{B_{J',\emptyset}}.
    \end{equation*}
    This completes the proof.
\end{proof}

Suppose that the matrices $\bra{\gamma_{J,J'}}$ and $(\nu_{J,J'})$ are conjugated by the diagonal matrix $Q$, then the matrices $\bigbra{\gamma_{J,J'}\delta_{J'\subseteq J-1}}$ and $\bigbra{\nu_{J,J'}\delta_{J'\subseteq J-1}}$ are also conjugated by $Q$.

\begin{proof}[Proof of Theorem \ref{dim2 Thm main intro}]
    We prove that $D_A(\pi)\cong D_A^{\otimes}(\rhobar^{\vee}(1))$ as \'etale $(\varphi,\OK\x)$-modules over $A$. Since $D_K(\rhobar^{\vee})$ is dual to $D_K(\rhobar)$ as \'etale $(\varphi,\OK\x)$-modules, by definition and the equivalence of categories \cite[Thm.~2.5.1]{BHHMS3} and Proposition \ref{dim2 Prop equiv}, there is a perfect pairing $D_A^{\otimes}(\rhobar)\times D_A^{\otimes}(\rhobar^{\vee})\to A$ which is equivariant for the actions of $\varphi$ and $\OK\x$. Hence it suffices to show that $\Hom_A(D_A(\pi),A)\cong D_A^{\otimes}(\rhobar(-1))\cong D_A^{\otimes}(\rhobar)(-1)$, or equivalently, $\Hom_A(D_A(\pi),A)(1)\cong D_A^{\otimes}(\rhobar)$. By \cite[Prop.~B.3(iii)]{Wang2} and \cite[Cor.~B.4]{Wang2}, it suffices to compare the matrices $\Mat(\varphi)$ (see (\ref{dim2 Eq Mat phi})) and $\Mat(\varphi)'$ (see (\ref{dim2 Eq Mat' phi})). Then by Lemma \ref{dim2 Lem A/AA} it suffices to show that $\gamma_{J,\emptyset}/\bra{\gamma_{J^c,\emptyset}\gamma_{J,J^c}}=\nu_{J,\emptyset}/\bra{\nu_{J^c,\emptyset}\nu_{J,J^c}}$ for all $J\subseteq\cJ$. This is a consequence of (\ref{dim2 Eq nu/nunu}) and Proposition \ref{dim2 Prop mu/mumu}.
\end{proof}

\appendix
\section*{Appendix}

\section{Proof of Theorem \ref{dim2 Thm main} in the non-generic case}\label{dim2 Sec app}

In this appendix, we finish the proof of Theorem \ref{dim2 Thm main}.

For $r\in\RR_{>0}$, we denote $B(r)\eqdef\bigset{x\in A_{\infty}':|x|\leq p^{-r}}$ and $B\cc(r)\eqdef\bigset{x\in A_{\infty}':|x|<p^{-r}}$. 

\begin{lemma}\label{dim2 Lem uTinv}
    We have the following relations in $A_{\infty}'$.
    \begin{enumerate}
        \item
        We have
        \begin{equation*}
        \begin{aligned}
            X_1^{1-\varphi}&\in\scalebox{1}{$\sum\limits_{i=0}^{f-1}T_{K,i}^{-p(1-q\inv)}$}-T_{K,0}^{q-1}\bbra{\scalebox{1}{$\sum\limits_{i=1}^{f-1}$}T_{K,i}^{-(1-q\inv)}}+B\!\bbra{\tfrac{(q-1)(2p-2)}{p}}\\
            &\subseteq\scalebox{1}{$\sum\limits_{i=0}^{f-1}$}T_{K,i}^{-p(1-q\inv)}+B\!\bbra{\tfrac{(q-1)(p-1)}{p}}\subseteq T_{K,f-1}^{-p(1-q\inv)}\bbbra{1+B\!\bbra{\tfrac{(q-1)(p-1)}{p}}}.
        \end{aligned}   
        \end{equation*}
        \item 
        Let $u\in A_{\infty}'$ be as in Lemma \ref{dim2 Lem u}, then we have
        \begin{equation*}
            uT_{K,0}\inv\in1+T_{K,0}^{q-1}\bbra{\scalebox{1}{$\sum\limits_{i=1}^{f-1}$}T_{K,i}^{-(1-q\inv)}}+B\!\bbra{\tfrac{(q-1)(2p-1)}{p}}\subseteq1+B\!\bbra{\tfrac{(q-1)(p-1)}{p}}.
        \end{equation*}
    \end{enumerate}
\end{lemma}

\begin{proof}
    Recall from the proof of \cite[Lemma~2.9.2]{BHHMS3} (especially the second formula before \cite[(63)]{BHHMS3}) that the element
    \begin{equation}\label{dim2 Eq uTinv-0}
        \sum\limits_{n=0}^{\infty}[x_n]p^n\eqdef\prod\limits_{i=0}^{f-1}\sum\limits_{n\geq0}[T_{K,i}^{q^{-n}}]p^n-\sum\limits_{n\geq0}\sum\limits_{i=0}^{f-1}[X_i^{-nf-i}]p^{nf+i}\in W((A_{\infty}')\cc)
    \end{equation}
    satisfies $|x_i|<p^{-c}$ for all $i\geq0$, and the proof of \loc\;shows that we can take $c=q-1$. In particular, we have
    \begin{equation*}
        \babs{x_0}=\babs{T_{K,0}\cdots T_{K,f-1}-X_0}<p^{-c},
    \end{equation*}
    hence
    \begin{equation}\label{dim2 Eq uTinv-1}
        X_0\in T_{K,0}\cdots T_{K,f-1}\bbbra{1+B\cc\!\bbra{c\!-\!(1\!+\!p\!+\!\cdots\!+\!p^{f-1})}}.
    \end{equation}
    
    By a direct computation in the ring of Witt vectors, we have from (\ref{dim2 Eq uTinv-0})
    \begin{equation*}
        \babs{x_1}=\babs{\sum\limits_{i=0}^{f-1}T_{K,0}\cdots T_{K,i}^{q\inv}\cdots T_{K,f-1}-X_1^{p\inv}-\sum\limits_{s=1}^{p-1}\frac{\binom{p}{s}}{p}(T_{K,0}\cdots T_{K,f-1})^{(p-s)/p}(-X_0)^{s/p}}<p^{-c},
    \end{equation*}
    hence
    \begin{align}    
        X_1^{p\inv}
        &\in T_{K,0}\cdots T_{K,f-1}\bbbra{\sum\limits_{i=0}^{f-1}T_{K,i}^{-(1-q\inv)}-\sum\limits_{s=1}^{p-1}\frac{\binom{p}{s}}{p}(-1)^s\bbbra{1\!+\!\bbra{\tfrac{X_0}{T_{K,0}\cdots T_{K,f-1}}\!-\!1}}^{s/p}}+B\cc(c)\label{dim2 Eq uTinv-7}\\
        &\subseteq T_{K,0}\cdots T_{K,f-1}\bbbra{\sum\limits_{i=0}^{f-1}T_{K,i}^{-(1-q\inv)}-\sum\limits_{s=1}^{p-1}\frac{\binom{p}{s}}{p}(-1)^s\bbbra{1\!+\!B\cc(c')}^{s}}\notag\\
        &\subseteq T_{K,0}\cdots T_{K,f-1}\bbbra{\sum\limits_{i=0}^{f-1}T_{K,i}^{-(1-q\inv)}+B\cc(c')}\label{dim2 Eq uTinv-2}
    \end{align}
    with $c'\eqdef\bigbra{c-(1+p+\cdots+p^{f-1})}/p$, where the second inclusion follows from  \eqref{dim2 Eq uTinv-1}, and the last inclusion uses $\sum\nolimits_{s=1}^{p-1}p\inv\binom{p}{s}(-1)^s=0$ (since $p\geq3$ is odd). Applying $\varphi$ to (\ref{dim2 Eq uTinv-2}) using (\ref{dim2 Eq phi TKi}) and Lemma \ref{dim2 Lem norm}(i),(ii), we get
    \begin{equation}\label{dim2 Eq uTinv-3}
        X_0\in T_{K,0}\cdots T_{K,f-1}\bbbra{1+T_{K,0}^{q-1}\bbra{\scalebox{1}{$\sum\limits_{i=1}^{f-1}$}T_{K,i}^{-(1-q\inv)}}+B\cc(pc'\!+\!q\!-\!1)}.
    \end{equation}
    Then we put (\ref{dim2 Eq uTinv-3}) into (\ref{dim2 Eq uTinv-7}). Since $c>1+p+\cdots+p^{f-1}+c'$, we get
    \begin{align*}
        X_1^{p\inv}&\in T_{K,0}\cdots T_{K,f-1}\bbbra{\sum\limits_{i=0}^{f-1}T_{K,i}^{-(1-q\inv)}-\sum\limits_{s=1}^{p-1}\frac{\binom{p}{s}}{p}(-1)^s\!\bbra{1+\scalebox{1}{$\sum\limits_{i=1}^{f-1}$}\tfrac{T_{K,0}^{(q-1)/p}}{T_{K,i}^{(1-q\inv)/p}}+B\cc\!\bbra{c'\!+\!\tfrac{q-1}{p}}\!}^{s}}\\
        &\subseteq T_{K,0}\cdots T_{K,f-1}\bbbra{\sum\limits_{i=0}^{f-1}T_{K,i}^{-(1-q\inv)}-\sum\limits_{s=1}^{p-1}\frac{\binom{p}{s}}{p}(-1)^s\!\bbra{1+s\scalebox{1}{$\sum\limits_{i=1}^{f-1}$}\tfrac{T_{K,0}^{(q-1)/p}}{T_{K,i}^{(1-q\inv)/p}}+B\!\bbra{\tfrac{(q-1)(2p-2)}{p^2}}\!}},
    \end{align*}
    where the last inclusion uses $(q-1)(2p-2)/p^2<c'+(q-1)/p$. Using the fact that $\sum\nolimits_{s=1}^{p-1}p\inv\binom{p}{s}(-1)^s=0$ and $\sum\nolimits_{s=1}^{p-1}p\inv\binom{p}{s}(-1)^ss=1$ (since $p\geq3$ is odd), we get
    \begin{align}
        X_1^{p\inv}&\in T_{K,0}\cdots T_{K,f-1}\bbbra{\scalebox{1}{$\sum\limits_{i=0}^{f-1}$}T_{K,i}^{-(1-q\inv)}-T_{K,0}^{(q-1)/p}\bbra{\scalebox{1}{$\sum\limits_{i=1}^{f-1}$}T_{K,i}^{-(1-q\inv)/p}}+B\!\bbra{\tfrac{(q-1)(2p-2)}{p^2}}}\label{dim2 Eq uTinv-4}\\
        &\subseteq T_{K,0}\cdots T_{K,f-1}\bbbra{\scalebox{1}{$\sum\limits_{i=0}^{f-1}$}T_{K,i}^{-(1-q\inv)}+B\!\bbra{\tfrac{(q-1)(p-1)}{p^2}}}\label{dim2 Eq uTinv-9}\\
        &\subseteq T_{K,0}\cdots T_{K,f-2}T_{K,f-1}^{q\inv}\bbbra{1+B\!\bbra{\tfrac{(q-1)(p-1)}{p^2}}}.\label{dim2 Eq uTinv-8}
    \end{align}
    Applying $\varphi$ to (\ref{dim2 Eq uTinv-8}) using (\ref{dim2 Eq phi TKi}) and Lemma \ref{dim2 Lem norm}(ii), we get
    \begin{equation}\label{dim2 Eq uTinv-5}
        X_0=\varphi(X_1^{p\inv})\in T_{K,0}\cdots T_{K,f-1}\bbbra{1+B\!\bbra{\tfrac{(q-1)(p-1)}{p}}}.
    \end{equation}
    Dividing (\ref{dim2 Eq uTinv-4}) by (\ref{dim2 Eq uTinv-5}) and then raising to the $p$-th power, we get
    \begin{equation*}
        X_1^{1-\varphi}\in\scalebox{1}{$\sum\limits_{i=0}^{f-1}T_{K,i}^{-p(1-q\inv)}$}-T_{K,0}^{q-1}\bbra{\scalebox{1}{$\sum\limits_{i=1}^{f-1}$}T_{K,i}^{-(1-q\inv)}}+B\!\bbra{\tfrac{(q-1)(2p-2)}{p}},
    \end{equation*}    
    which proves (i).

    \hspace{\fill}

    Dividing (\ref{dim2 Eq uTinv-9}) by (\ref{dim2 Eq uTinv-5}) and then applying $\varphi$, we get
    \begin{equation}\label{dim2 Eq uTinv-6}
        X_0^{1-\varphi}\in T_{K,0}^{-(q-1)}+\scalebox{1}{$\sum\limits_{i=1}^{f-1}$}T_{K,i}^{-(1-q\inv)}+B\!\bbra{\tfrac{(q-1)(2p-1)}{p}}.
    \end{equation}
    By the definition of $u$ (see the lines below \cite[(64)]{BHHMS3}) and using (\ref{dim2 Eq uTinv-6}), we get
    \begin{equation*}
    \begin{aligned}
        uT_{K,0}^{-1}&\eqdef\bbra{X_0^{\varphi-1}/T_{K,0}^{q-1}}^{1/(q-1)}\in\bbra{X_0^{1-\varphi}T_{K,0}^{q-1}}^{1+q\Zp}\\
        &\subseteq1+T_{K,0}^{q-1}\bbra{\scalebox{1}{$\sum\limits_{i=1}^{f-1}$}T_{K,i}^{-(1-q\inv)}}+B\!\bbra{\tfrac{(q-1)(2p-1)}{p}},
    \end{aligned}
    \end{equation*}
    which proves (ii).
\end{proof}

\begin{lemma}\label{dim2 Lem hj}
    Let $0\leq h\leq q-2$ and $0\leq j\leq f-1$.
    \begin{enumerate}
        \item 
        If $h_{j-1}\neq p-1$, then we have $(q-1)\bigbra{(p-1)p^{j-1}-[h]_{j-1}}>h$.
        \item 
        We have $p^{j}-[h]_{j-1}-p^{j-f}>h$.
    \end{enumerate}
\end{lemma}

\begin{proof}
    (i). If $j\geq1$, then using $h_{j-1}\neq p-1$ we have
    \begin{equation}\label{dim2 Eq main 1-1}
        (q-1)\bigbra{(p-1)p^{j-1}-[h]_{j-1}}\geq q-1>h.
    \end{equation}
    If $j=0$, then using $h_{f-1}\neq p-1$ we have (since $[h]_{-1}=0$)
    \begin{equation}\label{dim2 Eq main 1-2}
        (q-1)\bigbra{(p-1)p^{j-1}-[h]_{j-1}}=(q-1)(p-1)/p>(p-1)p^{f-1}-1\geq h.
    \end{equation}

    (ii). If $[h]_{j-1}\neq(p-1)(1+p+\cdots+p^{j-1})$, then we have
    \begin{equation*}
        (q-1)(p^j-[h]_{j-1}-p^{j-f})>q-1>h;
    \end{equation*}
    If $[h]_{j-1}=(p-1)(1+p+\cdots+p^{j-1})$, then we can't have $h_j=h_{j+1}=\cdots=h_{f-1}=p-1$ (otherwise $h=q-1$), so we get
    \begin{equation*}
        (q-1)(p^j-[h]_{j-1}-p^{j-f})\geq(q-1)(1-p^{j-f})>q-1-p^j\geq h.
    \end{equation*}
    This completes the proof.
\end{proof}

\begin{proof}[Completion of the proof of Theorem \ref{dim2 Thm main}]
    We keep the notation of the proof of Theorem \ref{dim2 Thm main}. It is enough to prove that (see (\ref{dim2 Eq Q2-4}) for $D_{01}$)
    \begin{equation}\label{dim2 Eq Q2-9}
    \begin{aligned}
        D_{01}\in\bbra{\id-\lambda_0\lambda_1\inv T_{K,0}^{-(q-1)h}\varphi_q}(b)+B\cc(h)
    \end{aligned}
    \end{equation}
    for certain $b\in A_{\infty}'$. Indeed, by Lemma \ref{dim2 Lem unique Ainfty'} there is a unique choice of $b_{01}\in b+B\cc(h)\subseteq A_{\infty}'$ satisfying \eqref{dim2 Eq Q2-4}. Then one can check the equalities of the (1,2)-entries of \eqref{dim2 Eq Q2-2} and \eqref{dim2 Eq Q2-3} as in Case 1. We separate the following cases.

    \paragraph{\textbf{Case 2:}} $c_j=1$ for some $0\leq j\leq f-1$, $h_j\neq0$ and $h_{j-1}=p-1$.

    We have
    \begin{align}\label{dim2 Eq Q2-10}
        &\bigbra{uT_{K,0}\inv}^{-h}X_0^{[h]_{j-1}(1-\varphi)}=T_{K,0}^{-(q-1)[h]_{j-1}}\bigbra{uT_{K,0}\inv}^{-\bbra{h+(q-1)[h]_{j-1}}}\notag\\
        &\hspace{1.5cm}\in T_{K,0}^{-(q-1)[h]_{j-1}}\bbbra{1+T_{K,0}^{q-1}\bbra{\scalebox{1}{$\sum\limits_{i=1}^{f-1}$}T_{K,i}^{-(1-q\inv)}}+B\!\bbra{\tfrac{(q-1)(2p-1)}{p}}}^{-p^jh_j+p^{j+1}\ZZ}\notag\\
        &\hspace{1.5cm}\subseteq T_{K,0}^{-(q-1)[h]_{j-1}}\bbbra{1-h_jT_{K,0}^{(q-1)p^j}\bbra{\scalebox{1}{$\sum\limits_{i=1}^{f-1}$}T_{K,i}^{-p^j(1-q\inv)}}+B\!\bbra{(q\!-\!1)(2p^j\!-\!2p^{j-1})}}\notag\\
        &\hspace{1.5cm}\subseteq T_{K,0}^{-(q-1)[h]_{j-1}}-h_jT_{K,0}^{(q-1)(p^j-[h]_{j-1})}\bbra{\scalebox{1}{$\sum\limits_{i=1}^{f-1}$}T_{K,i}^{-p^j(1-q\inv)}}+B\cc(q\!-\!1),
    \end{align}
    where the first equality uses Lemma \ref{dim2 Lem u}(i), the first inclusion follows from Lemma \ref{dim2 Lem uTinv}(ii), and the last inclusion uses $2p^j-2p^{j-1}-[h]_{j-1}>1$. We also have
    \begin{align}\label{dim2 Eq Q2-11}
        &\bigbra{uT_{K,0}\inv}^{-h}X_0^{([h]_{j-2}-p^{j-1})(1-\varphi)}X_1^{p^{j-1}(1-\varphi)}\notag\\
        &\hspace{1.5cm}=T_{K,0}^{-(q-1)([h]_{j-2}-p^{j-1})}\bigbra{uT_{K,0}\inv}^{-\bbra{h+(q-1)([h]_{j-2}-p^{j-1})}}X_1^{p^{j-1}(1-\varphi)}\notag\\
        &\hspace{1.5cm}\in T_{K,0}^{(q-1)(p^{j-1}-[h]_{j-2})}\bbbra{1+B\!\bbra{\tfrac{(q-1)(p-1)}{p}}}^{p^j\ZZ}\bbbra{\scalebox{1}{$\sum\limits_{i=0}^{f-1}$}T_{K,i}^{-p(1-q\inv)}+B\!\bbra{\tfrac{(q-1)(p-1)}{p}}}^{p^{j-1}}\notag\\  
        &\hspace{1.5cm}\subseteq T_{K,0}^{(q-1)(p^{j-1}-[h]_{j-2})}\bbbra{\scalebox{1}{$\sum\limits_{i=0}^{f-1}$}T_{K,i}^{-p^j(1-q\inv)}+B\!\bbra{(q\!-\!1)(p^{j-1}\!-\!p^{j-2})}}\notag\\
        &\hspace{1.5cm}\subseteq T_{K,0}^{(q-1)(p^{j}-[h]_{j-1})}\bbra{\scalebox{1}{$\sum\limits_{i=0}^{f-1}$}T_{K,i}^{-p^j(1-q\inv)}}+B\cc(q\!-\!1),
    \end{align}
    where the first equality uses Lemma \ref{dim2 Lem u}(i), the first inclusion follows from Lemma \ref{dim2 Lem uTinv}(i),(ii), and the last inclusion uses $h_{j-1}=p-1$ (hence $p^{j-1}-[h]_{j-2}=p^j-[h]_{j-1}$, and $p^{j-1}-p^{j-2}+(p^{j-1}-[h]_{j-2})=(p^j-[h]_{j-1})+(p^{j-1}-p^{j-2})>1$). Combining \eqref{dim2 Eq Q2-10} and \eqref{dim2 Eq Q2-11}, we get
    \begin{equation}\label{dim2 Eq Q2-12}
    \begin{aligned}
        &T_{K,0}^{-(q-1)[h]_{j-1}}-\bigbra{uT_{K,0}\inv}^{-h}\bbra{X_0^{[h]_{j-1}}+h_jX_0^{([h]_{j-2}-p^{j-1})(1-\varphi)}X_1^{p^{j-1}(1-\varphi)}}\\
        &\hspace{3cm}\in-h_jT_{K,0}^{(q-1)(p^j-[h]_{j-1}-p^{j-f})}+B\cc(q\!-\!1)\subseteq B\cc(h),
    \end{aligned}
    \end{equation}
    where the last inclusion follows from Lemma \ref{dim2 Lem hj}(ii) and $h<q-1$. In particular, for $j\geq1$ we have $\abs{D_{01}}<p^{-h}$, which proves \eqref{dim2 Eq Q2-9} (with $b=0$).

    Next we assume that $j=0$, so that $h_{f-1}=p-1$. Recall that $[B_{-1}^{\prime X}]\eqdef\lambda_0\lambda_1\inv[B_{f-1}^{\prime X}]$ in $W^X$. Then the difference of $D_{01}$ and the LHS of \eqref{dim2 Eq Q2-12} is
    \begin{align*}
        &h_0\bigbra{uT_{K,0}\inv}^{-h}\bbbra{\lambda_0\lambda_1\inv X_0^{([h]_{f-2}-p^{f-1})(1-\varphi)}X_1^{p^{f-1}(1-\varphi)}-X_0^{-(1-\varphi)}X_1^{p^{-1}(1-\varphi)}}\\
        &\hspace{1.5cm}=-h_0\bigbra{uT_{K,0}\inv}^{-h}\bbra{\id-\lambda_0\lambda_1\inv X_0^{h(1-\varphi)}\varphi_q}\bbbra{X_0^{-(1-\varphi)}X_1^{p^{-1}(1-\varphi)}}\\
        &\hspace{1.5cm}=\bbra{\id-\lambda_0\lambda_1\inv T_{K,0}^{-(q-1)h}\varphi_q}\bbbra{-h_0\bigbra{uT_{K,0}\inv}^{-h}X_0^{-(1-\varphi)}X_1^{p^{-1}(1-\varphi)}},
    \end{align*}
    where the first equality uses $h_{f-1}=p-1$ (hence $[h]_{f-2}-p^{f-1}=h-q$), and the second equality uses Lemma \ref{dim2 Lem op}(iii). This proves \eqref{dim2 Eq Q2-9} (with $b=-h_0\bigbra{uT_{K,0}\inv}^{-h}X_0^{-(1-\varphi)}X_1^{p^{-1}(1-\varphi)}$).

    \paragraph{\textbf{Case 3:}} $c_j=1$ for some $0\leq j\leq f-1$, $h_j=0$ and $h_{j-1}\neq p-1$.

    Let $0\leq r\leq f-1$ such that $h_{j+1}=\cdots=h_{j+r}=1$ and $h_{j+r+1}\neq1$. We have
    \begin{align}\label{dim2 Eq Q2-13}
        &\bigbra{uT_{K,0}\inv}^{-h}X_0^{([h]_{j+r}+p^{j+r+1})(1-\varphi)}=T_{K,0}^{-(q-1)([h]_{j+r}+p^{j+r+1})}\bigbra{uT_{K,0}\inv}^{-\bbra{h+(q-1)([h]_{j+r}+p^{j+r+1})}}\notag\\
        &\hspace{0.5cm}\in T_{K,0}^{-(q-1)([h]_{j+r}+p^{j+r+1})}\!\bbbra{1\!+\!T_{K,0}^{q-1}\!\bbra{\scalebox{1}{$\sum\limits_{\ell=1}^{f-1}$}T_{K,\ell}^{-(1-q\inv)}}\!+\!B\!\bbra{\tfrac{(q-1)(2p-1)}{p}}}^{p^{j+r+1}(1-h_{j+r+1})+p^{j+r+2}\ZZ}\notag\\
        &\hspace{0.5cm}\subseteq T_{K,0}^{-(q-1)([h]_{j+r}+p^{j+r+1})}\bbbra{1\!-\!(h_{j+r+1}\!-\!1)\scalebox{1}{$\sum\limits_{\ell=1}^{f-1}$}\tfrac{T_{K,0}^{(q-1)p^{j+r+1}}}{T_{K,\ell}^{p^{j+r+1}(1-q\inv)}}+B\!\bbra{(q\!-\!1)(2p^{j+r+1}\!-\!2p^{j+r})}}\notag\\
        &\hspace{0.5cm}\subseteq T_{K,0}^{-(q-1)([h]_{j+r}+p^{j+r+1})}-(h_{j+r+1}\!-\!1)T_{K,0}^{-(q-1)[h]_{j+r}}\!\bbra{\scalebox{1}{$\sum\limits_{\ell=1}^{f-1}$}T_{K,\ell}^{-p^{j+r+1}(1-q\inv)}}+B\cc(q\!-\!1),
    \end{align}
    where the first equality uses Lemma \ref{dim2 Lem u}(i), the first inclusion follows from Lemma \ref{dim2 Lem uTinv}(ii), and the last inclusion uses $h_{j+r}=1$ and $p\geq5$ (hence $2p^{j+r+1}-2p^{j+r}-([h]_{j+r}+p^{j+r+1})>1$). For $0\leq i\leq r$, we have
    \begin{align}\label{dim2 Eq Q2-14}
        &\bigbra{uT_{K,0}\inv}^{-h}X_0^{[h]_{j+i}(1-\varphi)}X_1^{p^{j+i}(1-\varphi)}=T_{K,0}^{-(q-1)[h]_{j+i}}\bigbra{uT_{K,0}\inv}^{-\bbra{h+(q-1)[h]_{j+i}}}X_1^{p^{j+i}(1-\varphi)}\notag\\
        &\hspace{0.5cm}\in T_{K,0}^{-(q-1)[h]_{j+i}}\bbbra{1\!+\!B\!\bbra{\tfrac{(q-1)(p-1)}{p}}}^{p^{j+i+1}\ZZ}\bbbra{\scalebox{1}{$\sum\limits_{\ell=0}^{f-1}$}T_{K,\ell}^{-p(1-q\inv)}-\scalebox{1}{$\sum\limits_{\ell=1}^{f-1}$}\tfrac{T_{K,0}^{q-1}}{T_{K,\ell}^{1-q\inv}}\!+\!B\!\bbra{\tfrac{(q-1)(2p-2)}{p}}}^{p^{j+i}}\notag\\
        &\hspace{0.5cm}\subseteq T_{K,0}^{-(q-1)[h]_{j+i}}\bbbra{\scalebox{1}{$\sum\limits_{\ell=0}^{f-1}$}T_{K,\ell}^{-p^{j+i+1}(1-q\inv)}-\scalebox{1}{$\sum\limits_{\ell=1}^{f-1}$}\tfrac{T_{K,0}^{(q-1)p^{j+i}}}{T_{K,\ell}^{p^{j+i}(1-q\inv)}}+B\!\bbra{(q\!-\!1)(2p^{j+i}\!-\!2p^{j+i-1})}},
    \end{align}
    where the first equality uses Lemma \ref{dim2 Lem u}(i), and the first inclusion follows from Lemma \ref{dim2 Lem uTinv}(i),(ii).
    
    If $1\leq i\leq r$, then using $h_{j+i}=1$, $h_{j+i-1}\in\set{0,1}$ and $p\geq5$ (hence $[h]_{j+i}-p^{j+i}=[h]_{j+i-1}$ and $2p^{j+i}-2p^{j+i-1}-[h]_{j+i}>1$) we deduce from (\ref{dim2 Eq Q2-14}) that
    \begin{equation}\label{dim2 Eq Q2-15}
    \begin{aligned}
        &\bigbra{uT_{K,0}\inv}^{-h}X_0^{[h]_{j+i}(1-\varphi)}X_1^{p^{j+i}(1-\varphi)}\\
        &\hspace{0.8cm}\in T_{K,0}^{-(q-1)[h]_{j+i}}\!\bbra{\scalebox{1}{$\sum\limits_{\ell=0}^{f-1}$}T_{K,\ell}^{-p^{j+i+1}(1-q\inv)}\!}-T_{K,0}^{-(q-1)[h]_{j+i-1}}\!\bbra{\scalebox{1}{$\sum\limits_{\ell=1}^{f-1}$}T_{K,\ell}^{-p^{j+i}(1-q\inv)}\!}+B\cc(q\!-\!1).
    \end{aligned}
    \end{equation}
    If $i=0$, then using $h_j=0$ (hence $[h]_j=[h]_{j-1}$ and $2p^j-2p^{j-1}-[h]_j>1$) we deduce from (\ref{dim2 Eq Q2-14}) that
    \begin{equation}\label{dim2 Eq Q2-16}
    \begin{aligned}
        &\bigbra{uT_{K,0}\inv}^{-h}X_0^{[h]_{j}(1-\varphi)}X_1^{p^{j}(1-\varphi)}\\
        &\hspace{1cm}\in T_{K,0}^{-(q-1)[h]_{j}}\bbra{\scalebox{1}{$\sum\limits_{\ell=0}^{f-1}$}T_{K,\ell}^{-p^{j+1}(1-q\inv)}}-T_{K,0}^{(q-1)(p^j-[h]_{j-1})}\bbra{\scalebox{1}{$\sum\limits_{\ell=1}^{f-1}$}T_{K,\ell}^{-p^{j}(1-q\inv)}}+B\cc(q\!-\!1).
    \end{aligned}
    \end{equation}
    Since $h_{j-1}\neq p-1$ by assumption, we deduce from (\ref{dim2 Eq Q2-16}), Lemma \ref{dim2 Lem norm}(i) and Lemma \ref{dim2 Lem hj}(i) that 
    \begin{equation}\label{dim2 Eq Q2-17}
    \begin{aligned}
        \bigbra{uT_{K,0}\inv}^{-h}X_0^{[h]_{j}(1-\varphi)}X_1^{p^{j}(1-\varphi)}\in T_{K,0}^{-(q-1)[h]_{j}}\bbra{\scalebox{1}{$\sum\limits_{\ell=0}^{f-1}$}T_{K,\ell}^{-p^{j+1}(1-q\inv)}}+B\cc(h).
    \end{aligned}
    \end{equation}
    
    Combining (\ref{dim2 Eq Q2-14}), (\ref{dim2 Eq Q2-15}) (with $1\leq i\leq r$) and (\ref{dim2 Eq Q2-17}), we get 
    \begin{equation}\label{dim2 Eq Q2-18}
        \bigbra{uT_{K,0}\inv}^{-h}D_j^X\in D'+B\cc(h)
    \end{equation}
    with 
    \begin{equation}\label{dim2 Eq Q2-19}
        D'\eqdef T_{K,0}^{-(q-1)([h]_{j+r}+p^{j+r+1})}+(h_{j+r+1}-1)\scalebox{1}{$\sum\limits_{i=0}^r$}\,T_{K,0}^{-(q-1)([h]_{j+i}+p^{j+i+1-f})}.
    \end{equation}
    By the definition of $D_j^{\LT}$, we deduce from (\ref{dim2 Eq Q2-18}) that $D_{01}\in\bigbra{\id-\lambda_0\lambda_1\inv T_{K,0}^{-(q-1)h}\varphi_q}(-D')+B\cc(h)$, which proves \eqref{dim2 Eq Q2-9}.

    \paragraph{\textbf{Case 4:}} $c_j=1$ for some $0\leq j\leq f-1$, $h_j=0$ and $h_{j-1}=p-1$.

    Let $0\leq r\leq f-1$ such that $h_{j+1}=\cdots=h_{j+r}=1$ and $h_{j+r+1}\neq1$. For simplicity, we assume that $j\geq1$. The case $j=0$ can be treated as in Case 2. Combining (\ref{dim2 Eq Q2-14}), (\ref{dim2 Eq Q2-15}) (with $1\leq i\leq r$), (\ref{dim2 Eq Q2-16}) and (\ref{dim2 Eq Q2-11}), we get (for $D'$ as in \eqref{dim2 Eq Q2-19})
    \begin{equation*}
        \bigbra{uT_{K,0}\inv}^{-h}\bbra{D_j^X+(h_{j+r+1}-1)D_j^{\prime X}}\in D'+T_{K,0}^{(q-1)(p^j-[h]_{j-1}-p^{j-f})}+B\cc(q\!-\!1)\subseteq D'+B\cc(h),
    \end{equation*}
    where the last inclusion follows from Lemma \ref{dim2 Lem hj}(ii) and $h<q-1$. This proves \eqref{dim2 Eq Q2-9} (with $b=-D'$) as in Case 3.

    \paragraph{\textbf{Case 5:}} $h=1+p+\cdots+p^{f-1}$, $\lambda_0\lambda_1\inv=1$ and $c_{\tr}=1$.

    Since $h_j=1$ for all $j$, the relation (\ref{dim2 Eq Q2-15}) still holds for $j=0$ and $0\leq i\leq f-1$, from which we deduce that
    \begin{align*}
        &\bigbra{uT_{K,0}\inv}^{-h}D_{\tr}^X=\bigbra{uT_{K,0}\inv}^{-h}\bbra{\scalebox{1}{$\sum\limits_{i=0}^{f-1}$}X_0^{[h]_i(1-\varphi)}X_1^{p^i(1-\varphi)}}\\
        &\hspace{1.5cm}\in\scalebox{1}{$\sum\limits_{i=0}^{f-1}$}\bbbra{T_{K,0}^{-(q-1)[h]_i}\bbra{\scalebox{1}{$\sum\limits_{\ell=0}^{f-1}$}T_{K,\ell}^{-p^{i+1}(1-q\inv)}}-T_{K,0}^{-(q-1)[h]_{i-1}}\bbra{\scalebox{1}{$\sum\limits_{\ell=1}^{f-1}$}T_{K,\ell}^{-p^i(1-q\inv)}}}+B\cc(h)\\
        &\hspace{1.5cm}=-\scalebox{1}{$\sum\limits_{\ell=1}^{f-1}$}T_{K,\ell}^{-(1-q\inv)}+\scalebox{1}{$\sum\limits_{i=0}^{f-2}$}T_{K,0}^{-(q-1)([h]_i+p^{i+1-f})}+T_{K,0}^{-(q-1)h}\bbra{\scalebox{1}{$\sum\limits_{\ell=0}^{f-1}$}T_{K,\ell}^{-p^f(1-q\inv)}}+B\cc(h)\\
        &\hspace{1.5cm}=\bbra{\id-T_{K,0}^{-(q-1)h}\varphi_q}\bbbra{-\scalebox{1}{$\scalebox{1}{$\sum\limits_{\ell=1}^{f-1}$}$}T_{K,\ell}^{-(1-q\inv)}+\scalebox{1}{$\sum\limits_{\ell=0}^{f-2}$}T_{K,0}^{-(q-1)([h]_i+p^{i+1-f})}}+D_{\tr}^{\LT} +B\cc(h),
    \end{align*}
    which proves \eqref{dim2 Eq Q2-9}.

    \paragraph{\textbf{Case 6:}} $h=0$, $\lambda_0\lambda_1\inv=1$ and $c_{\unr}=1$.

    In this case, we can take $Q=\smat{b_{00}&b_{01}\\0&b_{11}}=\smat{1&0\\0&1}$. This completes the proof of Theorem \ref{dim2 Thm main}.
\end{proof}

\bibliography{1}
\bibliographystyle{alpha}

\end{document}